\documentclass[11pt,letterpaper]{article}

\usepackage[utf8]{inputenc}
\usepackage[T1]{fontenc}
\usepackage[margin=1in,headheight=14pt]{geometry}
\usepackage{mathpazo}
\usepackage{microtype}
\usepackage{setspace}
\usepackage{enumerate}
\usepackage{fancyhdr}

\usepackage{amsmath,amssymb,amsfonts,amsthm}
\usepackage{mathtools}
\usepackage{bm}
\numberwithin{equation}{section}

\usepackage{graphicx,comment,verbatim}
\usepackage{tikz}
\usetikzlibrary{arrows.meta,positioning,calc}
\usepackage{pgfplots}
\pgfplotsset{compat=1.17}
\usepackage[ruled,vlined]{algorithm2e}
\usepackage[font=small,labelfont=bf,margin=0.5em]{caption}

\usepackage[section]{placeins}

\usepackage{xcolor}
\definecolor{linkblue}{RGB}{0,51,153}
\usepackage[colorlinks=true,linkcolor=linkblue,citecolor=linkblue,
            urlcolor=linkblue,pdfusetitle]{hyperref}

\theoremstyle{plain}
\newtheorem{theorem}{Theorem}[section]
\newtheorem{proposition}[theorem]{Proposition}
\newtheorem{lemma}[theorem]{Lemma}
\newtheorem{corollary}[theorem]{Corollary}
\newtheorem{fact}[theorem]{Fact}
\theoremstyle{definition}
\newtheorem{definition}[theorem]{Definition}
\newtheorem{assumption}[theorem]{Assumption}
\newtheorem{example}[theorem]{Example}
\theoremstyle{remark}
\newtheorem{remark}[theorem]{Remark}

\fancypagestyle{plain}{\fancyhf{}\fancyfoot[C]{\small\thepage}%
  }

\title{\vspace{-1.5em}\textbf{From Optimal Transport to Optimal Quantization:\\
A Variational Study of the Witsenhausen Counterexample}}
\author{%
  Quanyan Zhu\thanks{Department of Electrical and Computer Engineering,
  Tandon School of Engineering, New York University, Brooklyn, NY 11201, USA.
  Email: \href{mailto:qz494@nyu.edu}{\texttt{qz494@nyu.edu}}.}}
\date{}

\begin{document}

\maketitle
\thispagestyle{plain}

\begin{abstract}
\noindent
Witsenhausen's counterexample is a canonical problem in decentralized
control: a two-stage linear-quadratic-Gaussian team problem in which
nonlinear policies strictly outperform every affine one. This paper
develops a variational and computational framework for the scalar
counterexample based on optimal transport. Building on the transport
formulation of Wu and Verd\'u, we absorb the first controller into a
transport map and recast the problem as
\(J^\star=\inf_{Q}\{k^2W_2^2(P,Q)+\operatorname{mmse}(Q)\}\), which trades
a quadratic Wasserstein transport cost against a minimum mean-square
estimation error. The optimal first controller is the monotone
rearrangement that pushes the prior \(P\) to a minimizer \(Q^\star\).
We establish the existence and absolute continuity of minimizers, derive
an Euler--Lagrange optimality condition, and obtain a semi-closed-form
Gaussian benchmark together with an explicit second-variation stability
threshold for the Gaussian branch. We then show that restricting to
finitely supported laws yields an \emph{MMSE-regularized optimal
quantizer}: its optimal values decrease to \(J^\star\), its weak limit
points are minimizers of the continuous problem, and its stationarity
conditions pair centroid levels with power (additively weighted Voronoi)
cells, recovering Lloyd--Max quantization as control becomes expensive.
This structure leads to an estimation-aware Lloyd algorithm with
signalling-based initialization. We further characterize the asymptotic
regimes. As \(k\to\infty\), the optimal controllers converge to the
classical affine pair. As \(k\to0\), \(J^\star(k)=\sigma^2k^2+o(k^2)\)
and the optimal laws collapse onto \(\delta_0\), showing that signalling
constellations are a moderate-\(k\) phenomenon. Numerical experiments
corroborate the theory, and the framework extends to the vector
counterexample, where the finite-level program becomes an
MMSE-regularized vector quantizer.
\end{abstract}

\vspace{0.5em}
\noindent\textbf{Keywords:} Witsenhausen counterexample; decentralized
control; nonclassical information structure; optimal transport;
Wasserstein distance; minimum mean-square error; optimal quantization;
Lloyd--Max algorithm.

\vspace{0.3em}
\noindent\textbf{MSC 2020:} 93E20, 49Q22, 94A34, 91A12, 62B10.

\vspace{1em}
\hrule
\vspace{0.5em}

\section{Introduction}

Decentralized decision-making is ubiquitous in modern engineered systems.
Networked control systems, multi-agent robotics, cyber-physical
infrastructures, and teams of autonomous agents all involve multiple
decision-makers who act on different information and whose actions
affect what the others can observe. A fundamental question in such
systems is how an agent should act when its control action also serves
as a means of communication to the agents downstream.
Witsenhausen's counterexample~\cite{witsenhausen1968counterexample} is
the canonical problem that captures this tension. It is a two-stage
linear-quadratic-Gaussian (LQG) team problem with a nonclassical
information pattern, and it shows that a \emph{nonlinear} controller can
strictly outperform every affine one, even though the dynamics are linear,
the costs are quadratic, and the primitive uncertainties are Gaussian.
(Under such a controller the state \(Y=f(X_0)\) is no longer Gaussian; it
is the \emph{primitives} that are.) Nearly six decades later, the
counterexample remains a central benchmark in team decision theory and
networked control~\cite{yuksel2024stochastic}. The optimal cost and the
optimal controllers are still unknown in closed form, and even careful
numerical optimization is delicate because the problem is
\emph{nonconvex}.

In this work, we study the scalar counterexample through the lens of
\emph{optimal transport}. Following Wu and
Verd\'u~\cite{wu2011witsenhausen}, we absorb the first controller into a
transport map, which recasts the problem as a variational problem over
probability measures:
\begin{equation}
J^\star=\inf_{Q\in\mathcal P_2(\mathbb R)}
\bigl\{\,k^2\,W_2^2(P,Q)+\operatorname{mmse}(Q)\,\bigr\}
\qquad(P=\mathrm{Law}(X_0)).
\label{eq:intro-variational}
\end{equation}
The objective captures the fundamental trade-off of the counterexample.
The first controller pays a quadratic Wasserstein \emph{transport} cost to
reshape the state distribution, and the second controller incurs a
minimum mean-square \emph{estimation} cost determined by that shape. A
minimizer \(Q^\star\) determines an optimal first controller as a
monotone transport map. Building on this viewpoint, we develop a
variational \emph{and} computational framework whose central message is
that the counterexample is, at its core, an MMSE-regularized
\emph{quantization} problem.

\paragraph{Contributions.}
The main contributions of this paper are summarized as follows.
\begin{itemize}
\item \emph{Variational characterization.} We prove the
reformulation~\eqref{eq:intro-variational} and identify the optimal first
controller as the monotone rearrangement that pushes \(P\) to \(Q^\star\)
(Theorem~\ref{thm:ot-reformulation}). We establish the existence and
absolute continuity of minimizers, so that \(Q^\star\) is never atomic,
and derive an Euler--Lagrange optimality condition. Along the way, we
record a Fisher-information form of the estimation cost that explains its
concavity and connects it to the I-MMSE relation.
\item \emph{Gaussian benchmark.} We obtain a semi-closed-form solution
within the Gaussian class, together with an explicit second-variation
stability threshold for the Gaussian/affine branch
(\S\ref{sec:characterization}).
\item \emph{Quantization structure.} As our central contribution, we show
that restricting to finitely supported laws yields an
\emph{MMSE-regularized optimal quantizer}, and that this restriction loses
nothing in the limit. The finite-level optimal values decrease
monotonically to \(J^\star\) (Proposition~\ref{prop:finite-level}), and
every weak limit point of the exact finite-level minimizers is a minimizer
of the continuous problem; hence \(Q_n^\star\rightharpoonup Q^\star\)
whenever the minimizer is unique, even though \(Q^\star\) itself carries
no atoms. The stationarity conditions pair centroid levels with additively
weighted (power) decision cells, namely Voronoi cells shifted by the
estimation sensitivity, and they reduce to the classical Lloyd--Max
quantizer as control becomes expensive.
\item \emph{Algorithm.} The quantizer viewpoint organizes the nonconvexity
as a deterministic-annealing picture in \(k\): the computed constellations
split at successive bifurcations as control cheapens through the
moderate-\(k\) window. This structure leads to an estimation-aware Lloyd
algorithm that, at each \(k\), is initialized from a signalling-seeded
finite-level solution (\S\ref{sec:finite-program-sec}).
\item \emph{Asymptotic analysis.} We characterize the small- and
large-\(k\) regimes. As control becomes expensive, the optimal controllers
converge to the classical affine pair. As control becomes cheap, we obtain
the exact leading order \(J^\star(k)=\sigma^2k^2+o(k^2)\) and show that the
optimal laws collapse onto \(\delta_0\). Consequently, the signalling
constellations returned by the solver are a moderate-\(k\) phenomenon
rather than the \(k\to0\) limit.
\item \emph{Numerical study and vector extension.} We illustrate the
theory numerically across the cheap- and expensive-control regimes
(\S\ref{sec:numerics}) and extend the framework to the \emph{vector}
counterexample (\S\ref{sec:vector}). There the finite-level program
becomes an MMSE-regularized vector quantizer whose stationary decision
regions are power (Laguerre) polytopes, and the solver returns a
two-dimensional signalling constellation. The reported vector
computations use a nearest-level assignment in place of the exact power
cells; we verify numerically that the correction is negligible at the
settings considered.
\end{itemize}

These results show that the optimal-transport lens provides more than a
reformulation. It offers a unifying structure that links decentralized
control, estimation, and quantization, and it explains why good
controllers for the counterexample take the shape they do.

\paragraph{Organization.}
The rest of the paper is organized as follows. The remainder of this
section formulates the problem and proves the
reformulation~\eqref{eq:intro-variational}. Section~\ref{sec:related}
reviews related work. Section~\ref{sec:characterization} characterizes
\(Q^\star\): it establishes existence and absolute continuity, derives the
MMSE--Fisher identity and the Euler--Lagrange condition in score form, and
develops the Gaussian benchmark. Section~\ref{sec:finite-program-sec}
builds the finite-dimensional quantization program and its solver.
Section~\ref{sec:numerics} presents the numerical study, the small- and
large-\(k\) asymptotics, and the limiting controllers.
Section~\ref{sec:vector} extends the construction to the vector
counterexample, and Section~\ref{sec:conclusion} concludes the paper and
discusses open problems. Appendix~\ref{app:solver} documents the
continuous transport-map solver, its validation, and the limits of its
reliability.

\subsection{Classical Formulation of Witsenhausen's Problem}

\begin{definition}[Witsenhausen problem]
\label{def:witsenhausen}
Let \(X_0\sim\mathcal{N}(0,\sigma^2)\) and \(N\sim\mathcal{N}(0,1)\) be
independent, and let \(k>0\). Controller~1 applies
\(U_1=\gamma_1(X_0)\), producing the state \(Y=X_0+U_1\); Controller~2
observes \(Z=Y+N\) and applies \(U_2=\gamma_2(Z)\). The Witsenhausen
problem is to minimize
\begin{equation}
J(\gamma_1,\gamma_2)
=
k^2\,\mathbb{E}[U_1^2]+\mathbb{E}\bigl[(Y-U_2)^2\bigr]
\label{eq:classical-cost}
\end{equation}
over measurable policies \(\gamma_1,\gamma_2\), and we write \(J^\star\)
for the optimal value.
\end{definition}

\begin{figure}[t]
\centering
\begin{tikzpicture}[>=Latex,
  box/.style={draw,rounded corners,minimum height=9mm,minimum width=11mm},
  sum/.style={draw,circle,inner sep=1.3pt}]
\node (x0) {\(X_0\)};
\node[box,right=1.1cm of x0] (c1) {\(\gamma_1\)};
\node[sum,right=1.2cm of c1] (s1) {\(+\)};
\node[sum,right=1.7cm of s1] (s2) {\(+\)};
\node[box,right=1.2cm of s2] (c2) {\(\gamma_2\)};
\node[right=1.0cm of c2] (u2) {\(U_2\)};
\coordinate (br) at ($(x0)!0.5!(c1)$);
\draw[->] (x0) -- (c1);
\draw[->] (c1) -- node[above,font=\small]{\(U_1\)} (s1);
\draw[->] (br) -- ++(0,-0.85) -| (s1);
\draw[->] (s1) -- node[above,font=\small]{\(Y\)} (s2);
\draw[->] (s2) -- node[above,font=\small]{\(Z\)} (c2);
\draw[->] (c2) -- (u2);
\node[above=0.75cm of s2] (n) {\(N\)};
\draw[->] (n) -- (s2);
\end{tikzpicture}
\caption{Information structure of Witsenhausen's problem. Controller~1
sees \(X_0\) and forms \(Y=X_0+\gamma_1(X_0)\) at quadratic control cost
\(k^2\mathbb{E}[U_1^2]\); Controller~2 sees only the noisy observation
\(Z=Y+N\) and must estimate \(Y\). Controller~2 does not observe
Controller~1's information, which creates the nonclassical information
structure.}
\label{fig:blockdiagram}
\end{figure}

It is convenient to absorb the first control into the map
\(f(x):=x+\gamma_1(x)\) (so that \(Y=f(X_0)\) and
\(U_1=f(X_0)-X_0\)) and to write \(g:=\gamma_2\). Then
\eqref{eq:classical-cost} becomes
\begin{equation}
J(f,g)
=
k^2\,\mathbb{E}\bigl[(f(X_0)-X_0)^2\bigr]
+\mathbb{E}\bigl[(f(X_0)-g(f(X_0)+N))^2\bigr].
\label{eq:cost-fg}
\end{equation}

\subsection{Notation and Imported Facts}
\label{sec:prelim}

We write \(P:=\mathrm{Law}(X_0)=\mathcal{N}(0,\sigma^2)\); \(\phi\) for
the standard normal density (and \(\phi_\sigma\) for the
\(\mathcal N(0,\sigma^2)\) density, \(\phi_d\) for the standard normal density
on \(\mathbb R^d\), used in \S\ref{sec:vector}); \(F_\mu\) for the cumulative
distribution
function of a measure \(\mu\); and \(\mathcal{P}_2(\mathbb{R})\) for the
Borel probability measures on \(\mathbb{R}\) with finite second moment.
The letter \(\varphi\) is reserved for Kantorovich potentials
(\(\varphi_Q\), Fact~\ref{fact:wass-var}) and is never a density. We fix
\(\operatorname{sign}(0):=1\); since \(P(\{0\})=0\), every identity below
involving \(\operatorname{sign}\) holds \(P\)-almost everywhere and is
unaffected by the choice.
For \(P,Q\in\mathcal{P}_2(\mathbb{R})\) the quadratic Wasserstein
distance is
\begin{equation}
W_2^2(P,Q)=\inf_{\pi\in\Pi(P,Q)}\int_{\mathbb{R}^2}(x-y)^2\,\pi(dx,dy),
\label{eq:W2-def}
\end{equation}
the infimum over couplings \(\pi\) with marginals \(P,Q\).

\emph{Notation for the optimizer.} We reserve capital letters for
probability measures and lowercase for their densities: \(Q\) denotes a
generic law, \(Q^\star\) a (henceforth fixed) minimizer
of~\eqref{eq:OT-variational-intro} --- uniqueness is not claimed, and
Remark~\ref{rem:existence-comments}(iii) explains why we phrase results for
\emph{every} minimizer wherever it matters ---
and \(q^\star\) the density of \(Q^\star\) (which exists by
Proposition~\ref{prop:ac-gauss}). In \S\ref{sec:finite-program-sec} we also use
\(Q_n^\star\) for an \emph{exact} minimizer of the program restricted to laws
supported on \emph{at most} \(n\) points (matching
Proposition~\ref{prop:finite-level}); these are atomic
approximations, and every weak limit point of \((Q_n^\star)\) is a minimizer of
\eqref{eq:OT-variational-intro} (Remark~\ref{rem:atomic-vs-ac}), so when the
minimizer is unique \(Q_n^\star\rightharpoonup Q^\star\) as \(n\to\infty\) ---
the \(Q_n^\star\) converge to the \emph{measure} \(Q^\star\), whose density is
\(q^\star\). Generic sequences of measures are written \(\nu_j\).

\emph{Stars and hats.} We follow one convention throughout, and it matters
because the programs solved here are nonconvex. A \emph{star} marks an object
whose optimality is \emph{proved}: \(J^\star,Q^\star,T^\star,f^\star\) for the
continuous problem, \(J_n^\star,Q_n^\star\) for its exact \(n\)-level
restriction. A \emph{hat} marks a \emph{solver output or computed candidate}:
\(\widehat J,\widehat J_n,\widehat Q_n,\widehat T,\widehat f_n\). A hat carries
\emph{no} claim of optimality, stationarity, or an upper-bound property by
itself --- which of those hold depends on the object, as follows.

\emph{Optimality and stationarity.} We do not build stationarity into the
notation: the projected
descent of Algorithm~\ref{alg:lloyd} returns an approximate KKT point
of~\eqref{eq:kkt-block} (approximate, because the stopping test uses a finite
tolerance), while the generalized Lloyd alternation of
Remark~\ref{rem:general-solver} is only known to be monotone in value
(stationarity of its accumulation points is not established there), and the
vector solver of \S\ref{sec:vector} is only a fixed point of an approximate
alternation. No computed quantity in this paper is a certified global optimum,
and we never write a star for one.

\emph{The bound property, which is \emph{not} universal.} For a computed
finite-level law \(\widehat Q_n\) --- an honest \(n\)-atom probability measure
--- the \emph{exact} cost \(J(\widehat Q_n)\) is an upper bound, since
\(\widehat Q_n\) competes in~\eqref{eq:OT-variational}:
\[
J(\widehat Q_n)\ \ge\ J_n^\star\ \ge\ J^\star .
\]
Here \(\widehat J_n\) abbreviates that exact value, not the rounded number our
quadrature prints for it. By contrast the objective \(\widehat J(k)\) returned
by the transport-map solver of Appendix~\ref{app:solver} is computed against a
\emph{discretized} prior and is \emph{not} a bound of any kind --- it can lie
below \(J^\star\). The next paragraph makes the distinction precise, and we
use it consistently thereafter.

\emph{Three kinds of number.} Because the distinction matters in
\S\ref{sec:numerics} and Appendix~\ref{app:solver}, we separate:
\begin{enumerate}
\item[(N1)] the \emph{exact cost of a feasible law} --- a genuine upper bound
on \(J^\star\), since the law competes in \eqref{eq:OT-variational};
\item[(N2)] a \emph{numerically evaluated} version of (N1), obtained by
quadrature. We call these \emph{evaluated feasible values}. They are highly
accurate evaluations of quantities that \emph{are} upper bounds, but we do
\emph{not} call them certified or rigorous: that would require an explicit
numerical error enclosure (interval arithmetic or a quadrature-error bound),
which we do not carry out;
\item[(N3)] the objective of the \emph{discretized} transport-map program,
which is what the continuous solver of Appendix~\ref{app:solver} minimizes.
This is a numerical \emph{estimate} and not a bound of any kind: the
discretization biases the estimation term slightly downward
(\S\ref{app:scope}), so it can sit \emph{below} \(J^\star\).
\end{enumerate}
The finite-level numbers of \S\ref{sec:finite-program-sec} are of type (N2);
the transport-map values \(\widehat J(k)\) are of type (N3), and
\S\ref{app:validation}(ii) converts each of them into an accompanying (N2)
value.

To keep the development self-contained, we collect here the four standard
results we invoke; everything else is proved from them.

\begin{fact}[One-dimensional optimal transport {\cite{santambrogio2015optimal}}]
\label{fact:1d-ot}
If \(P\) is absolutely continuous, then for every
\(Q\in\mathcal{P}_2(\mathbb{R})\),
\[
W_2^2(P,Q)=\int_0^1\bigl(F_P^{-1}(t)-F_Q^{-1}(t)\bigr)^2\,dt,
\]
and the infimum in \eqref{eq:W2-def} is attained by the (comonotone)
deterministic coupling induced by the monotone map
\(T=F_Q^{-1}\circ F_P\), which pushes \(P\) forward to \(Q\).
\end{fact}

\begin{fact}[First variation of the transport term {\cite{santambrogio2015optimal}}]
\label{fact:wass-var}
Let \(\varphi_Q\) be a Kantorovich potential for the cost
\(c(x,y)=\tfrac12(x-y)^2\), normalized by
\(\tfrac12 W_2^2(P,Q)=\int\varphi_Q\,dQ+\int\varphi_Q^{c}\,dP\). Then the
linear first variation of \(Q\mapsto\tfrac12 W_2^2(P,Q)\) is represented,
up to an additive constant, by \(\varphi_Q\); and in one dimension
\(\varphi_Q'(y)=y-T_Q(y)\) with \(T_Q=F_P^{-1}\circ F_Q\) the monotone map
pushing \(Q\) to \(P\).
\end{fact}

\begin{fact}[Functional properties of MMSE {\cite{wu2012functional}}]
\label{fact:mmse}
Under additive standard-Gaussian noise, the map
\(Q\mapsto\operatorname{mmse}(Q)\) on \(\mathcal{P}_2(\mathbb{R})\) is
concave and continuous with respect to weak convergence.
\end{fact}

\begin{fact}[Regularity of the Witsenhausen optimum {\cite{wu2011witsenhausen}}]
\label{fact:regularity}
For \(P=\mathcal{N}(0,\sigma^2)\) and \(k>0\), any optimal first
controller of the reformulated problem is odd and strictly increasing
with a real-analytic left inverse; consequently the optimal law
\(Q^\star\) is symmetric about the origin and absolutely continuous with
interval support.
\end{fact}

\subsection{MMSE Reduction of the Inner Problem}

We first eliminate the second controller. Throughout,
\[
\operatorname{mmse}(Q):=\mathbb{E}\bigl[\operatorname{Var}(Y\mid Y+N)\bigr],
\qquad Y\sim Q,\ N\sim\mathcal{N}(0,1)\ \text{independent}.
\]

\begin{lemma}[MMSE reduction]
\label{lem:mmse-reduction}
Fix a measurable \(f\) with \(\mathbb{E}[f(X_0)^2]<\infty\), and let
\(Y=f(X_0)\), \(Q=\mathrm{Law}(Y)\). The inner minimization over \(g\)
is attained by the conditional mean
\(g_f^\star(z)=\mathbb{E}[Y\mid Y+N=z]\), and
\begin{equation}
J^\star(f):=\inf_g J(f,g)
=k^2\,\mathbb{E}\bigl[(f(X_0)-X_0)^2\bigr]+\operatorname{mmse}(Q).
\label{eq:Jstar-f}
\end{equation}
In particular the estimation cost depends on \(f\) only through the law
\(Q\) of \(Y\).
\end{lemma}

\begin{proof}
Fix \(f\) and let \(Z=Y+N\). For any measurable \(g\), insert and subtract
the conditional mean \(\mathbb{E}[Y\mid Z]\) and expand the square:
\[
\begin{aligned}
\mathbb{E}\bigl[(Y-g(Z))^2\bigr]
&=\mathbb{E}\bigl[(Y-\mathbb{E}[Y\mid Z])^2\bigr]
+2\,\mathbb{E}\bigl[(Y-\mathbb{E}[Y\mid Z])(\mathbb{E}[Y\mid Z]-g(Z))\bigr]\\
&\quad+\mathbb{E}\bigl[(\mathbb{E}[Y\mid Z]-g(Z))^2\bigr].
\end{aligned}
\]
The cross term vanishes: conditioning on \(Z\) and using the tower
property, \(\mathbb{E}[(Y-\mathbb{E}[Y\mid Z])\,h(Z)]
=\mathbb{E}\bigl[h(Z)\,\mathbb{E}[Y-\mathbb{E}[Y\mid Z]\mid Z]\bigr]=0\) for
every \(h\) (here \(h(Z)=\mathbb{E}[Y\mid Z]-g(Z)\)). Hence
\[
\mathbb{E}\bigl[(Y-g(Z))^2\bigr]
=\underbrace{\mathbb{E}\bigl[(Y-\mathbb{E}[Y\mid Z])^2\bigr]}_{\text{independent of }g}
+\mathbb{E}\bigl[(\mathbb{E}[Y\mid Z]-g(Z))^2\bigr]
\ \ge\ \mathbb{E}\bigl[(Y-\mathbb{E}[Y\mid Z])^2\bigr],
\]
with equality iff \(g(Z)=\mathbb{E}[Y\mid Z]\) a.s.; so the minimizer is
\(g_f^\star(z)=\mathbb{E}[Y\mid Z=z]\). Its value is
\(\mathbb{E}[(Y-\mathbb{E}[Y\mid Z])^2]
=\mathbb{E}\bigl[\mathbb{E}[(Y-\mathbb{E}[Y\mid Z])^2\mid Z]\bigr]
=\mathbb{E}[\operatorname{Var}(Y\mid Z)]=\operatorname{mmse}(Q)\), again by
the tower property. Because \(Z=Y+N\) with \(N\perp Y\), the joint law of
\((Y,Z)\), hence \(\operatorname{mmse}(Q)\), depends on \(f\) only through
\(Q=\mathrm{Law}(Y)\). The control term
\(k^2\mathbb{E}[(f(X_0)-X_0)^2]\) is unaffected by \(g\); adding the two
gives~\eqref{eq:Jstar-f}.
\end{proof}

\subsection{Optimal-Transport Reformulation}
\label{sec:ot-reform}

Recall \(P=\mathrm{Law}(X_0)\) and the Wasserstein distance
\eqref{eq:W2-def} from \S\ref{sec:prelim}.

\begin{theorem}[Optimal-transport reformulation of Witsenhausen's problem]
\label{thm:ot-reformulation}
With \(P=\mathrm{Law}(X_0)\), the optimal value of the Witsenhausen
problem equals the variational problem
\begin{equation}
J^\star
=
\inf_{Q\in\mathcal{P}_2(\mathbb{R})}
\biggl\{
k^2 W_2^2(P,Q)
+
\operatorname{mmse}(Q)
\biggr\}.
\label{eq:OT-variational-intro}
\end{equation}
The infimum is attained (Proposition~\ref{prop:existence}), and any
minimizer \(Q^\star\) is realized by a \emph{deterministic} first
controller, namely the monotone rearrangement
\begin{equation}
f^\star(x)=F_{Q^\star}^{-1}\bigl(F_P(x)\bigr),
\qquad Y=f^\star(X_0),
\label{eq:monotone-rearrangement}
\end{equation}
which is the optimal transport map pushing \(P\) to \(Q^\star\)
(\(F_P,F_{Q^\star}\) are the distribution functions of \(P,Q^\star\)).
\end{theorem}

\begin{proof}
Write \(V:=\inf_{Q\in\mathcal P_2(\mathbb R)}\{k^2W_2^2(P,Q)
+\operatorname{mmse}(Q)\}\) for the right-hand side
of~\eqref{eq:OT-variational-intro}. We prove \(J^\star\ge V\) and
\(J^\star\le V\) separately. By Lemma~\ref{lem:mmse-reduction} the second
controller has already been eliminated: minimizing \(J\) over pairs
\((\gamma_1,\gamma_2)\) is the same as minimizing
\begin{equation}
\mathcal J(f):=k^2\,\mathbb{E}\bigl[(X_0-Y)^2\bigr]
+\operatorname{mmse}(\mathrm{Law}(Y)),
\qquad Y=f(X_0),
\label{eq:reduced-over-f}
\end{equation}
over measurable \(f\) with \(\mathbb E[f(X_0)^2]<\infty\); call these
\emph{admissible}. So \(J^\star=\inf_f\mathcal J(f)\).

\emph{Step 1: every admissible \(f\) induces a coupling.} Given admissible
\(f\), let \(Q_f:=f_\#P=\mathrm{Law}(Y)\), which lies in
\(\mathcal P_2(\mathbb R)\) precisely because
\(\mathbb E[f(X_0)^2]<\infty\), and let
\[
\pi_f:=(\mathrm{id},f)_\#P,
\qquad\text{i.e.}\qquad
\pi_f(A\times B)=P\bigl(A\cap f^{-1}(B)\bigr),
\]
the law of the pair \((X_0,f(X_0))\) --- the \emph{graph coupling} of \(f\).
Its first marginal is \(P\) and its second is \(Q_f\), so
\(\pi_f\in\Pi(P,Q_f)\), and by the change-of-variables formula
\(\mathbb E[(X_0-f(X_0))^2]=\int(x-y)^2\,d\pi_f\). Thus deterministic
controllers form a subset of the couplings, and
\begin{equation}
\mathcal J(f)=k^2\!\int(x-y)^2d\pi_f+\operatorname{mmse}(Q_f).
\label{eq:J-as-coupling}
\end{equation}

\emph{Step 2: lower bound \(J^\star\ge V\).} Enlarging the feasible set from
graph couplings to \emph{all} couplings can only lower the infimum, so
by~\eqref{eq:J-as-coupling}
\[
J^\star=\inf_f\mathcal J(f)
\ \ge\
\inf_{Q\in\mathcal P_2(\mathbb R)}\ \inf_{\pi\in\Pi(P,Q)}
\Bigl(k^2\!\int(x-y)^2d\pi+\operatorname{mmse}(Q)\Bigr).
\]
Here we have simply grouped the couplings according to their second marginal:
every \(\pi\) with first marginal \(P\) has \emph{some} second marginal
\(Q\), and the set of all such \(\pi\) is the disjoint union over \(Q\) of
the sets \(\Pi(P,Q)\), so minimizing over \(\pi\) is minimizing over \(Q\)
and then over \(\pi\in\Pi(P,Q)\). The point of the grouping is that, once
\(Q\) is fixed, the estimation term \(\operatorname{mmse}(Q)\) is a
\emph{constant} in the inner minimization --- this is exactly the content of
the last sentence of Lemma~\ref{lem:mmse-reduction}, that the estimation cost
sees \(f\) only through \(Q\). Hence it factors out,
\[
\inf_{\pi\in\Pi(P,Q)}\Bigl(k^2\!\int(x-y)^2d\pi+\operatorname{mmse}(Q)\Bigr)
=k^2\underbrace{\inf_{\pi\in\Pi(P,Q)}\int(x-y)^2d\pi}_{=\,W_2^2(P,Q)
\ \text{by }\eqref{eq:W2-def}}+\operatorname{mmse}(Q),
\]
and taking the infimum over \(Q\) gives \(J^\star\ge V\).

\emph{Step 3: upper bound \(J^\star\le V\), attained deterministically.} The
lower bound was obtained by relaxing, so a priori it could be strict; we now
show it is not, because the relaxed problem happens to be solved by a
coupling that \emph{is} of graph type. Let \(Q^\star\) attain \(V\), which it
does by Proposition~\ref{prop:existence}. Since \(P=\mathcal N(0,\sigma^2)\)
is absolutely continuous, Fact~\ref{fact:1d-ot} applies: the infimum defining
\(W_2^2(P,Q^\star)\) is attained by the deterministic coupling induced by the
monotone map
\[
f^\star:=F_{Q^\star}^{-1}\circ F_P,
\qquad f^\star_\#P=Q^\star,
\qquad
\mathbb E\bigl[(X_0-f^\star(X_0))^2\bigr]=W_2^2(P,Q^\star).
\]
This \(f^\star\) is admissible: it is measurable (a composition of monotone
functions) and \(\mathbb E[f^\star(X_0)^2]=\int y^2\,dQ^\star<\infty\)
because \(Q^\star\in\mathcal P_2(\mathbb R)\). Since
\(\mathrm{Law}(f^\star(X_0))=Q^\star\), evaluating~\eqref{eq:reduced-over-f}
at \(f^\star\) gives
\[
J^\star\ \le\ \mathcal J(f^\star)
=k^2\,W_2^2(P,Q^\star)+\operatorname{mmse}(Q^\star)=V .
\]

\emph{Conclusion.} The two bounds give \(J^\star=V\), which
is~\eqref{eq:OT-variational-intro}; and the chain in Step~3 holds with
equality throughout, so \(f^\star\) is optimal for the original problem,
which is~\eqref{eq:monotone-rearrangement}. The optimal second controller is
then \(g^\star(z)=\mathbb E[Y\mid Z=z]\) with \(Y=f^\star(X_0)\), by
Lemma~\ref{lem:mmse-reduction}.
\end{proof}

Thus the first controller chooses a distribution \(Q\) that balances a
\emph{transport cost} \(k^2 W_2^2(P,Q)\) against an \emph{estimation
cost} \(\operatorname{mmse}(Q)\), and \(Y=f^\star(X_0)\) is the optimal
representation signalled to the second controller
(Figure~\ref{fig:map}). The rest of the paper studies the variational
problem~\eqref{eq:OT-variational-intro}.

\begin{figure}[!b]
\centering
\begin{tikzpicture}
\begin{axis}[width=0.62\linewidth,height=5.2cm,
  axis lines=middle,xlabel={\(x\)},ylabel={\(f^\star(x)\)},
  xlabel style={below right},ylabel style={above left},
  xmin=-6,xmax=6,ymin=-6,ymax=6,xtick=\empty,ytick=\empty,
  legend style={at={(0.03,0.97)},anchor=north west,font=\small,draw=none,fill=none}]
\addplot[very thick,domain=-6:6,samples=200]{x+0.9*sin(deg(x))};
\addlegendentry{\(f^\star\) (schematic)}
\addplot[dashed,gray,domain=-6:6,samples=2]{x};
\addlegendentry{identity (\(U_1\equiv0\))}
\end{axis}
\end{tikzpicture}
\caption{The optimal first controller
\(f^\star=F_{Q^\star}^{-1}\circ F_P\) is a strictly increasing, generally
nonlinear rearrangement (Theorem~\ref{thm:ot-reformulation}). In the
\emph{moderate-\(k\) signalling regime} the numerically recovered
controller candidates develop the staircase-like profile sketched here,
spreading probability mass
into well-separated levels and making \(Y\) easier to estimate from
\(Z=Y+N\); that geometry is what the computations of \S\ref{sec:numerics}
show, not something Theorem~\ref{thm:ot-reformulation} establishes. The qualifier matters: as \(k\downarrow0\) this trend reverses
and the optimal law collapses onto \(\delta_0\)
(Proposition~\ref{prop:asymptotics}(b)). The curve shown is a schematic, not
the exact optimizer.}
\label{fig:map}
\end{figure}

\section{Related Work}
\label{sec:related}

\paragraph{The counterexample and decentralized control.}
Witsenhausen~\cite{witsenhausen1968counterexample} exhibited a nonlinear
policy strictly better than every affine one, overturning the expectation
that Gaussian primitives and quadratic costs force linear optimality; the
example became foundational for team decision theory and information
structures~\cite{ho1980team,yuksel2024stochastic}. The role of the
information pattern was clarified by Bansal and
Ba\c{s}ar~\cite{bansal1987stochastic}, who characterized when affine laws
are optimal; Ba\c{s}ar~\cite{basar2008variations} later placed the
counterexample within a broad family of nonclassical LQG teams and zero-sum
games, delineating when optimal (or saddle-point) policies are linear versus
nonlinear, and the monograph of Y\"uksel and
Ba\c{s}ar~\cite{yuksel2013stochastic} develops the decentralized
stochastic-control and networked-systems theory in which such
information-constrained problems sit. Complexity results, meanwhile,
established that decentralized LQG design is in general
intractable~\cite{papadimitriou1986intractable}. This is the backdrop for
the nonconvexity we confront in \S\ref{sec:finite-program-sec}.

\paragraph{Computing the counterexample.}
As no closed form is available, much effort has gone into numerical and
approximate solutions: approximating (neural) networks for the two
policies~\cite{baglietto2001numerical}, hierarchical and global search over
the nonconvex landscape~\cite{lee2001witsenhausen}, and
information-theoretic constructions with provable approximation guarantees,
especially in the vector case~\cite{grover2013approximately}. These works
produce increasingly good controllers; our aim is complementary, to expose
the \emph{structure} of the optimizer (absolute continuity,
Euler--Lagrange conditions, a quantization form) that explains why good
controllers take the shape they do, and to organize the computation around
it.

\paragraph{Optimal transport and MMSE.}
The transport viewpoint we adopt is due to Wu and
Verd\'u~\cite{wu2011witsenhausen}, who recast the first controller as a
transport map; it rests on functional properties of the
MMSE~\cite{wu2012functional} and the I-MMSE relation between estimation and
information~\cite{guo2005mutual}. We use standard one-dimensional
optimal-transport machinery, monotone rearrangement and Wasserstein
gradients~\cite{villani2009optimal,santambrogio2015optimal}. Relative
to~\cite{wu2011witsenhausen}, we go beyond the reformulation itself to a
full variational characterization of \(Q^\star\), its Fisher-information
form, and a finite-dimensional quantization program with a matching solver.

\paragraph{Quantization.}
Our finite-level program is, structurally, an optimal quantizer with an
MMSE regularizer. When control is expensive it reduces to the classical
scalar quantizer of Lloyd and Max~\cite{lloyd1982least}, whose theory and
vector extensions are surveyed in~\cite{gray1998quantization} and treated
measure-theoretically in~\cite{graf2000foundations}. The quantizer structure
makes available the standard remedies for its nonconvexity, the
Linde--Buzo--Gray splitting heuristic~\cite{linde1980algorithm} and
deterministic annealing~\cite{rose1998deterministic} with the control penalty
\(k\) as a temperature; the Voronoi/centroid stationarity conditions of
\S\ref{sec:finite-program-sec} are the estimation-regularized analogues of
the Lloyd--Max conditions.

\section{Characterization of the Optimal Distribution \(Q^\star\)}
\label{sec:characterization}

We consider the variational problem
\begin{equation}
J^\star
=
\inf_{Q\in\mathcal{P}_2(\mathbb{R})}
\Bigl\{
k^2 W_2^2(P,Q)
+
\operatorname{mmse}(Q)
\Bigr\},
\label{eq:OT-variational}
\end{equation}
where \(\mathcal{P}_2(\mathbb{R})\) denotes the set of Borel probability
measures on \(\mathbb{R}\) with finite second moment, \(P\) is the law of
the initial state \(X_0\), and
\[
\operatorname{mmse}(Q)
=
\mathbb{E}\!\left[\operatorname{Var}(Y\mid Y+N)\right],
\qquad
Y\sim Q,\quad N\sim\mathcal{N}(0,1)\ \text{independent}.
\]

Throughout we use the following elementary bounds on the estimation
term, which also make the objective coercive.

\begin{lemma}[Bounds on the MMSE functional]
\label{lem:mmse-bounds}
Let \(Y\sim Q\in\mathcal{P}_2(\mathbb{R})\) with variance
\(v:=\operatorname{Var}(Y)\), and \(N\sim\mathcal{N}(0,1)\) independent.
Then
\[
0\ \le\ \operatorname{mmse}(Q)\ \le\ \frac{v}{1+v}\ <\ 1 .
\]
In particular \(\operatorname{mmse}\) is uniformly bounded by \(1\) on all
of \(\mathcal{P}_2(\mathbb{R})\).
\end{lemma}

\begin{proof}
Nonnegativity is immediate since \(\operatorname{mmse}(Q)=\mathbb{E}
[\operatorname{Var}(Y\mid Y+N)]\ge 0\). For the upper bound, the MMSE is
the error of the \emph{optimal} estimator of \(Y\) from \(Z=Y+N\), hence
it is no larger than the error of the best \emph{linear} estimator. As
\(\operatorname{Var}(Y\mid Z)\) is invariant under deterministic shifts
of \(Y\), we may assume \(\mathbb{E}[Y]=0\); the linear MMSE is then
\[
\operatorname{mmse}(Q)\ \le\
\operatorname{Var}(Y)-\frac{\operatorname{Cov}(Y,Z)^2}{\operatorname{Var}(Z)}
=
v-\frac{v^2}{v+1}
=
\frac{v(v+1)-v^2}{v+1}
=
\frac{v}{1+v}.
\]
Here we used, since \(Z=Y+N\) with \(N\) independent of \(Y\) and
\(\operatorname{Var}(N)=1\),
\[
\operatorname{Cov}(Y,Z)=\operatorname{Cov}(Y,Y)+\operatorname{Cov}(Y,N)
=\operatorname{Var}(Y)=v,
\qquad
\operatorname{Var}(Z)=\operatorname{Var}(Y)+\operatorname{Var}(N)=v+1.
\]
Since \(v<\infty\), we have \(v/(1+v)<1\).
\end{proof}

\subsection{Existence of a Minimizer}

We first state a basic existence result.

\begin{proposition}[Existence]
\label{prop:existence}
Assume \(P\in\mathcal{P}_2(\mathbb{R})\) and \(k>0\). Then the functional
\[
J(Q)
:=
k^2 W_2^2(P,Q)
+
\operatorname{mmse}(Q),
\qquad Q\in\mathcal{P}_2(\mathbb{R}),
\]
admits at least one minimizer \(Q^\star\in\mathcal{P}_2(\mathbb{R})\).
\end{proposition}

\begin{proof}
This is the \emph{direct method} of the calculus of variations, which needs
three ingredients: the infimum is finite; minimizing sequences stay inside a
weakly compact set (\emph{coercivity}); and the functional does not jump
downward in the limit (\emph{weak lower semicontinuity}). We supply the three
in turn. Throughout, \(\rightharpoonup\) denotes weak (narrow) convergence of
probability measures, i.e.\ \(\int f\,d\nu_j\to\int f\,d\nu\) for every
bounded continuous \(f\).

\emph{Step 1: the infimum is finite.} Both terms of \(J\) are nonnegative, so
\(J\ge0\) and \(J^\star:=\inf_{Q\in\mathcal P_2}J(Q)\ge0\). The infimum is
over a nonempty set, since \(Q=P\) is admissible, and it is finite because
\[
J^\star\ \le\ J(P)\ =\ k^2W_2^2(P,P)+\operatorname{mmse}(P)
\ =\ \operatorname{mmse}(P)\ \le\ \frac{\operatorname{Var}(X_0)}
{1+\operatorname{Var}(X_0)}\ <\ 1
\]
by Lemma~\ref{lem:mmse-bounds}. Hence \(J^\star\in[0,1)\). Fix a minimizing
sequence \((\nu_j)\subset\mathcal P_2(\mathbb R)\), chosen so that
\begin{equation}
J^\star\ \le\ J(\nu_j)\ \le\ J^\star+1\ =:\ C\ <\ \infty
\qquad\text{for every }j .
\label{eq:minseq}
\end{equation}

\emph{Step 2: coercivity, i.e.\ uniformly bounded second moments.} Discarding
the (nonnegative) estimation term in~\eqref{eq:minseq},
\[
k^2\,W_2^2(P,\nu_j)\ \le\ J(\nu_j)\ \le\ C,
\qquad\text{so}\qquad
W_2(P,\nu_j)\ \le\ \frac{\sqrt C}{k}\ =:\ R
\]
for every \(j\); this is the only place \(k>0\) is used, and it is what
prevents a minimizing sequence from escaping to infinity. Now observe that
the second moment is itself a Wasserstein distance: the only coupling of
\(\nu\) with \(\delta_0\) is the product one, so
\(W_2^2(\nu,\delta_0)=\int y^2\,d\nu\). Since \(W_2\) is a metric on
\(\mathcal P_2(\mathbb R)\) \cite{santambrogio2015optimal}, the triangle
inequality gives
\begin{equation}
\Bigl(\int y^2\,d\nu_j\Bigr)^{1/2}
=W_2(\nu_j,\delta_0)
\le W_2(\nu_j,P)+W_2(P,\delta_0)
\le R+\Bigl(\int x^2\,dP\Bigr)^{1/2}
=:M<\infty ,
\label{eq:secondmoment}
\end{equation}
uniformly in \(j\).

\emph{Step 3: compactness.} Bounded second moments give tightness through
Markov's inequality: for every \(r>0\),
\[
\nu_j\bigl(\{|y|>r\}\bigr)\ \le\ \frac1{r^2}\int y^2\,d\nu_j\ \le\
\frac{M^2}{r^2},
\]
so given \(\varepsilon>0\) the \emph{single} compact set
\(K:=[-M/\sqrt\varepsilon,\,M/\sqrt\varepsilon]\) satisfies
\(\nu_j(K)\ge1-\varepsilon\) for all \(j\). Thus \((\nu_j)\) is tight, and
Prokhorov's theorem yields a subsequence \((\nu_{j_l})\) and a Borel
probability measure \(Q^\star\) with \(\nu_{j_l}\rightharpoonup Q^\star\).
That \(Q^\star\) is \emph{admissible}, i.e.\ lies in
\(\mathcal P_2(\mathbb R)\), follows from~\eqref{eq:secondmoment}: the map
\(y\mapsto y^2\) is nonnegative and continuous, so truncating at level \(m\)
and using that \(y^2\wedge m\) is bounded continuous,
\[
\int (y^2\wedge m)\,dQ^\star=\lim_l\int(y^2\wedge m)\,d\nu_{j_l}\le M^2
\quad\text{for every }m,
\]
whence \(\int y^2\,dQ^\star\le M^2<\infty\) by monotone convergence.

\emph{Step 4: the transport term is weakly lower semicontinuous.} We claim
\begin{equation}
W_2^2(P,Q^\star)\ \le\ \liminf_{l}W_2^2(P,\nu_{j_l}).
\label{eq:W2-lsc}
\end{equation}
Let \(\pi_l\in\Pi(P,\nu_{j_l})\) be optimal, so
\(\int(x-y)^2d\pi_l=W_2^2(P,\nu_{j_l})\). The family \((\pi_l)\) is tight,
because its marginals are: if \(P(K_1^c)\le\varepsilon\) and
\(\nu_{j_l}(K_2^c)\le\varepsilon\) then
\(\pi_l\bigl((K_1\times K_2)^c\bigr)\le
\pi_l(K_1^c\times\mathbb R)+\pi_l(\mathbb R\times K_2^c)\le2\varepsilon\),
and \(K_1\times K_2\) is compact. Passing to a further subsequence (not
relabelled) along which the \(\liminf\) in~\eqref{eq:W2-lsc} is attained as a
limit, Prokhorov gives \(\pi_l\rightharpoonup\pi\) for some probability
measure \(\pi\) on \(\mathbb R^2\). Its marginals are the limits of the
marginals, since \((x,y)\mapsto f(x)\) and \((x,y)\mapsto f(y)\) are bounded
continuous whenever \(f\) is; hence \(\pi\in\Pi(P,Q^\star)\). Finally the
cost \(c(x,y)=(x-y)^2\) is nonnegative and continuous, so by the same
truncation as in Step~3, for every \(m\),
\[
\int (c\wedge m)\,d\pi=\lim_l\int (c\wedge m)\,d\pi_l
\le\liminf_l\int c\,d\pi_l ,
\]
and letting \(m\to\infty\), \(\int c\,d\pi\le\liminf_l W_2^2(P,\nu_{j_l})\).
As \(\pi\) is \emph{some} coupling of \(P\) and \(Q^\star\), while
\(W_2^2(P,Q^\star)\) is the infimum over all of them,
\(W_2^2(P,Q^\star)\le\int c\,d\pi\), which is~\eqref{eq:W2-lsc}.

\emph{Step 5: the estimation term passes to the limit, and conclusion.} By
Fact~\ref{fact:mmse} the map \(Q\mapsto\operatorname{mmse}(Q)\) is continuous
along weakly convergent sequences, so
\(\operatorname{mmse}(\nu_{j_l})\to\operatorname{mmse}(Q^\star)\). Combining
this with~\eqref{eq:W2-lsc}, and using that \(\liminf(a_l+b_l)=\liminf
a_l+\lim b_l\) when \((b_l)\) converges,
\[
J(Q^\star)
=k^2W_2^2(P,Q^\star)+\operatorname{mmse}(Q^\star)
\le\liminf_l\Bigl[k^2W_2^2(P,\nu_{j_l})+\operatorname{mmse}(\nu_{j_l})\Bigr]
=\liminf_l J(\nu_{j_l})=J^\star .
\]
Since \(Q^\star\in\mathcal P_2(\mathbb R)\) by Step~3, the reverse inequality
\(J(Q^\star)\ge J^\star\) holds by definition of the infimum. Hence
\(J(Q^\star)=J^\star\) and \(Q^\star\) is a minimizer.
\end{proof}

\begin{remark}[Three comments on the proof]
\label{rem:existence-comments}
\emph{(i) Only lower semicontinuity is needed, and continuity is where the
smoothing works.} Step~5 used the full continuity supplied by
Fact~\ref{fact:mmse}, but weak \emph{lower} semicontinuity of
\(\operatorname{mmse}\) would have sufficed, since a sum of lower
semicontinuous terms is lower semicontinuous. We record this because
continuity is the stronger and less obvious property, and it is exactly here
that the Gaussian smoothing does the work: the observation law \(Q*\phi\)
depends on \(Q\) far more regularly than \(Q\) does on itself. For the
reader's reassurance we indicate why: writing \(\rho_j=\nu_j*\phi\) and
\(\rho=Q^\star*\phi\), the identity \(\operatorname{mmse}=1-\mathcal I(\rho)\)
of Proposition~\ref{prop:mmse-fisher} reduces the question to
\(\mathcal I(\rho_j)\to\mathcal I(\rho)\). Now \(\rho_j\to\rho\) and
\(\rho_j'\to\rho'\) pointwise (dominated convergence, \(\phi\) and \(\phi'\)
being bounded and continuous), so the integrands converge pointwise; and by
Cauchy--Schwarz applied to
\(\rho_j'(z)=-\int(z-y)\phi(z-y)\,d\nu_j\),
\[
\frac{\rho_j'(z)^2}{\rho_j(z)}\ \le\ h_j(z):=\int(z-y)^2\phi(z-y)\,d\nu_j(y),
\]
where \(h_j\,dz=\nu_j*(u^2\phi(u)\,du)\) is a probability density with
\(h_j\to h\) pointwise and \(\int h_j=\int h=1\). The generalized dominated
convergence theorem then gives \(\mathcal I(\rho_j)\to\mathcal I(\rho)\).
(Incidentally \(\int h_j=1\) re-proves the Stam bound
\(\mathcal I(\rho_j)\le1\) of Corollary~\ref{cor:fisher-bounds}.)

\emph{(ii) Where each hypothesis enters.} Finiteness of the second moment of
\(P\) is used twice, in Step~1 (through Lemma~\ref{lem:mmse-bounds}) and in
Step~2; and \(k>0\) is used only for the coercivity bound
\(W_2(P,\nu_j)\le\sqrt C/k\). Without a strictly positive transport penalty
the argument breaks down --- as it must, since at \(k=0\) the structure of
the problem changes completely (Remark~\ref{rem:k-zero}).

\emph{(iii) Uniqueness is not claimed, and does not hold in general.} The
functional \(J\) is a sum of a convex term and the \emph{concave} functional
\(\operatorname{mmse}\) (Fact~\ref{fact:mmse}), so it is not convex and the
direct method gives existence only. Proposition~\ref{prop:gauss-foc} shows that, already inside the Gaussian
family, the first-order condition can have several roots when
\(k^2<\tfrac14\) (three at \(k=0.2\), \(\sigma=5\)). This is why the statements that follow are phrased for
\emph{every} minimizer (as in Proposition~\ref{prop:ac}) rather than for
``the'' minimizer.
\end{remark}

\subsection{Absolute Continuity of the Optimizer}

For the remainder we record the mild symmetry and regularity properties
of the prior.

\begin{assumption}
\label{as:symmetric-P}
The distribution \(P\) is symmetric about zero, has mean zero, and admits
a density \(p\) that is strictly positive and smooth, as holds for
\(P=\mathcal{N}(0,\sigma^2)\).
\end{assumption}

When \(P\) is Gaussian, Fact~\ref{fact:regularity} additionally yields
symmetry of \(Q^\star\) about zero. We do not assert this under
Assumption~\ref{as:symmetric-P} alone: symmetry of the objective does not by
itself force every minimizer to be symmetric, absent uniqueness or a further
argument, and Fact~\ref{fact:regularity} is stated only for
\(P=\mathcal N(0,\sigma^2)\). Nothing in
Proposition~\ref{prop:ac} uses symmetry. The structural property we shall
need in the sequel is that
\(Q^\star\) cannot be purely discrete.

\begin{proposition}[No atoms]
\label{prop:ac}
Suppose \(k>0\) and \(P\) is absolutely continuous with a strictly
positive density and finite variance (as in
Assumption~\ref{as:symmetric-P}). Then no minimizer \(Q^\star\)
of \eqref{eq:OT-variational} has an atom.
\end{proposition}

\begin{proposition}[Absolute continuity in the Gaussian Witsenhausen case]
\label{prop:ac-gauss}
Suppose in addition that \(P=\mathcal N(0,\sigma^2)\), so that
Fact~\ref{fact:regularity} applies. Then every minimizer of
\eqref{eq:OT-variational} is absolutely continuous with respect to Lebesgue
measure, with interval support and density positive almost everywhere on that
support.
\end{proposition}

The split is deliberate. The spreading argument below needs only the general
hypotheses of Assumption~\ref{as:symmetric-P}, and gives the structural fact
we use repeatedly (no atoms). Absolute continuity, by contrast, we take from
Fact~\ref{fact:regularity}, which is stated in \cite{wu2011witsenhausen} for
the Gaussian prior of the Witsenhausen problem specifically; we do not claim
it for the general \(P\) of Proposition~\ref{prop:ac}. We also state the
density as positive \emph{almost everywhere} on the support rather than at
every point: Fact~\ref{fact:regularity} gives a strictly increasing optimal
map with a real-analytic left inverse, and an increasing real-analytic
function may still have isolated critical points, at which the induced density
vanishes. Nothing below uses pointwise positivity.

\begin{proof}[Proof of Proposition~\ref{prop:ac-gauss}]
Immediate from Fact~\ref{fact:regularity}: the optimal map \(f^\star\) is
strictly increasing with a real-analytic left inverse \(g\), so
\(Q^\star=f^\star_\#P\) has distribution function \(F_P\circ g\) on the
interval \(f^\star(\mathbb R)\), which is locally absolutely continuous with
density \((p\circ g)\,g'\ge0\); this vanishes only where \(g'\) does, a
Lebesgue-null set since \(g\) is real-analytic and nonconstant.
\end{proof}

\begin{proof}[Proof of Proposition~\ref{prop:ac}]
We give a self-contained
argument, which already rules out the discrete
solutions one might naively expect.

The mechanism is that the transport cost penalizes atoms at \emph{first}
order in the amount of spreading, whereas the estimation cost responds
only at \emph{second} order. We make this precise.

Suppose, for contradiction, that a minimizer \(Q^\star\) has an atom of
mass \(\alpha:=Q^\star(\{y_0\})>0\). Let \(S^\star\) be the monotone
optimal transport map with \(S^\star_\#P=Q^\star\), which exists and is
nondecreasing by Fact~\ref{fact:1d-ot} because \(P\) is absolutely
continuous. Consider the set of states that \(S^\star\) sends to the atom,
\[
I:=(S^\star)^{-1}(\{y_0\}).
\]
Three of its properties are what drive the argument. First, \(I\) is an
\emph{interval}: the preimage of a single point under a nondecreasing map is
order-convex, since \(x_1<x<x_2\) with \(S^\star(x_1)=S^\star(x_2)=y_0\)
forces \(y_0\le S^\star(x)\le y_0\). Second, \(P(I)=\alpha\): the
pushforward identity \(S^\star_\#P=Q^\star\) evaluated on the singleton
\(\{y_0\}\) reads \(P((S^\star)^{-1}(\{y_0\}))=Q^\star(\{y_0\})=\alpha>0\).
Third, \(I\) therefore has positive length --- a single point, or any
Lebesgue-null set, would carry no \(P\)-mass since \(P\) has a density ---
and so the conditional law
\[
P_I:=P(\,\cdot\mid I)=\tfrac1\alpha P|_I
\]
is a genuine, non-degenerate probability measure on \(I\), with mean
\(\mu_I\) and variance \(\sigma_I^2>0\) (both finite because \(P\) has finite
variance, even when \(I\) is a half-line). Non-degeneracy, i.e.\
\(\sigma_I^2>0\), is what will make the first-order gain strictly positive
below.

The plan is now to construct a competitor \(Q_s\) that replaces the atom by
a narrow cluster of width \(s\), and to show that the transport cost falls
like \(-s\) while the estimation cost rises only like \(+s^2\), so that
\(Q_s\) beats \(Q^\star\) for small \(s\).

\emph{Competitor.} Fix \(U\sim\mathrm{Unif}[-\sqrt3,\sqrt3]\) (mean \(0\),
variance \(1\), bounded), and for \(s\in(0,1]\) let
\(\mu_s:=\mathrm{Law}(y_0+sU)\). Define
\[
Q_s:=Q^\star-\alpha\,\delta_{y_0}+\alpha\,\mu_s\ \in\ \mathcal{P}_2(\mathbb{R}),
\]
which replaces the atom by a cluster of width \(O(s)\) centered at
\(y_0\) and leaves the rest of \(Q^\star\) untouched.

\emph{Transport estimate (first order).} Couple \(P\) and \(Q_s\) by the
plan equal to \(S^\star\) off \(I\) and, on \(I\), given by the
comonotone coupling of \(P_I\) with \(\mu_s\). Off \(I\) the cost equals
the corresponding part of \(W_2^2(P,Q^\star)\); on \(I\) it changes from
\(\int_I(x-y_0)^2\,P(dx)=\alpha\bigl(\sigma_I^2+(\mu_I-y_0)^2\bigr)\) to
\(\alpha\,W_2^2(P_I,\mu_s)\). By the one-dimensional formula
\(W_2^2(\mu,\nu)=\int_0^1(F_\mu^{-1}-F_\nu^{-1})^2\,dt\) and
\(F_{\mu_s}^{-1}(t)=y_0+s\,F_U^{-1}(t)\) (since \(\mu_s=\mathrm{Law}(y_0+sU)\)),
\[
W_2^2(P_I,\mu_s)
=\int_0^1\!\Bigl(F_{P_I}^{-1}(t)-y_0-s\,F_U^{-1}(t)\Bigr)^2 dt
=\int_0^1\!\Bigl[\bigl(\mu_I-y_0\bigr)+\bigl(F_{P_I}^{-1}(t)-\mu_I\bigr)
-s\,F_U^{-1}(t)\Bigr]^2 dt.
\]
Expanding the square and integrating term by term, the two centered
pieces are orthogonal to the constant \(\mu_I-y_0\) because
\(\int_0^1(F_{P_I}^{-1}-\mu_I)\,dt=0\) and \(\int_0^1 F_U^{-1}\,dt
=\mathbb E[U]=0\); using \(\int_0^1(F_{P_I}^{-1}-\mu_I)^2\,dt=\sigma_I^2\)
and \(\int_0^1(F_U^{-1})^2\,dt=\operatorname{Var}(U)=1\),
\[
W_2^2(P_I,\mu_s)
=(\mu_I-y_0)^2+\sigma_I^2-2\kappa s+s^2,
\qquad
\kappa:=\int_0^1 \bigl(F_{P_I}^{-1}(t)-\mu_I\bigr)F_U^{-1}(t)\,dt>0,
\]
where \(\kappa>0\) because both quantile functions are (strictly)
increasing, so their comonotone rearrangements are positively correlated
(equivalently, \(\kappa=\operatorname{Cov}\bigl(F_{P_I}^{-1}(T),F_U^{-1}(T)\bigr)\)
for \(T\sim\mathrm{Unif}[0,1]\), a covariance of two increasing functions of
\(T\), hence positive).
Since this coupling only upper-bounds \(W_2^2(P,Q_s)\),
\[
W_2^2(P,Q_s)-W_2^2(P,Q^\star)\ \le\ \alpha\,(s^2-2\kappa s).
\tag{$\ast$}
\]

\emph{Estimation estimate (second order).} Recall
\(\operatorname{mmse}(Q)=\mathbb{E}_Q[Y^2]-\Phi(Q)\) with
\(\Phi(Q)=\mathbb{E}\bigl[(\mathbb{E}[Y\mid Z])^2\bigr]=\int
g_Q(z)^2/p_Z^Q(z)\,dz\), where \(p_Z^Q=Q*\phi\) is the observation density and
\(g_Q(z)=\int y\,\phi(z-y)\,Q(dy)\) (so \(\mathbb E[Y\mid Z=z]=g_Q/p_Z^Q\)).

We first record that \(\Phi\) is \emph{convex} in \(Q\), which is what lets
us bound its increment from below by a tangent line, and hence bound the
increment of \(\operatorname{mmse}=\mathbb E_Q[Y^2]-\Phi\) from
\emph{above}. The reason is that \((g,p)\mapsto g^2/p\) is the
\emph{perspective} of the square function, jointly convex on
\(\mathbb R\times(0,\infty)\); convexity is preserved by precomposition with
the linear maps \(Q\mapsto g_Q(z)\) and \(Q\mapsto p_Z^Q(z)\), which are
linear because both are integrals of a fixed kernel against \(Q\); and it is
preserved by the integration over \(z\), a nonnegative combination. Hence
\(\Phi\) is convex on \(\mathcal P_2(\mathbb R)\). Its first variation at
\(Q^\star\) is
represented by \(\Psi^\star(y):=2y\,a^\star(y)-b^\star(y)\); this
representation is valid not only for density perturbations but for any
finite, compactly supported signed perturbation of mass zero such as
\(h=\alpha(\mu_s-\delta_{y_0})\), because along \(Q_t=Q^\star+th\) both
\(g_{Q_t}\) and \(p_Z^{Q_t}\) are affine in \(t\) with
\(p_Z^{Q^\star}=Q^\star*\phi>0\), so
\(\tfrac{d}{dt}\big|_0\Phi(Q_t)=\int\Psi^\star\,dh\)
(cf.~\eqref{eq:delta-mmse}). Moreover \(a^\star,b^\star\) are Gaussian
convolutions of \(m^\star\) and \((m^\star)^2\); since
\(m^\star=\mathbb{E}[Y\mid Z=\cdot]\) has at most linear growth (as
\(Q^\star\in\mathcal{P}_2\)) and \(p_Z^{Q^\star}\) is smooth and strictly
positive \emph{even when \(Q^\star\) has an atom}, differentiation under
the integral is justified and \(\Psi^\star\in C^\infty\) with locally
bounded derivatives. Convexity gives the tangent-line bound
\[
\Phi(Q_s)-\Phi(Q^\star)
\ \ge\
\int \Psi^\star\,d(Q_s-Q^\star)
=\alpha\Bigl(\int \Psi^\star\,d\mu_s-\Psi^\star(y_0)\Bigr).
\]
Now expand \(\Psi^\star\) about \(y_0\). Since \(\mu_s\) is the law of
\(y_0+sU\) with \(U\) bounded, Taylor's theorem with remainder gives, for
\(y\) in the (shrinking, bounded) support of \(\mu_s\),
\[
\Psi^\star(y)=\Psi^\star(y_0)+(\Psi^\star)'(y_0)(y-y_0)
+\tfrac12(\Psi^\star)''(y_0)(y-y_0)^2+o\bigl((y-y_0)^2\bigr).
\]
Integrating against \(\mu_s\) and using its first two moments,
\(\int(y-y_0)\,d\mu_s=s\,\mathbb E[U]=0\) and
\(\int(y-y_0)^2d\mu_s=s^2\operatorname{Var}(U)=s^2\), the linear term drops
out and
\[
\int \Psi^\star\,d\mu_s-\Psi^\star(y_0)
=\tfrac12(\Psi^\star)''(y_0)\,s^2+o(s^2).
\]
So the tangent-line bound gives \(\Phi(Q_s)-\Phi(Q^\star)\ge -C_1\alpha s^2\)
for \(C_1:=\tfrac12|(\Psi^\star)''(y_0)|+1<\infty\) and all small \(s\).
Meanwhile the second moment moves by exactly
\[
\mathbb{E}_{Q_s}[Y^2]-\mathbb{E}_{Q^\star}[Y^2]
=\alpha\Bigl(\mathbb E\bigl[(y_0+sU)^2\bigr]-y_0^2\Bigr)
=\alpha\bigl(2y_0s\,\mathbb E[U]+s^2\mathbb E[U^2]\bigr)=\alpha s^2 .
\]
Therefore
\[
\operatorname{mmse}(Q_s)-\operatorname{mmse}(Q^\star)
=\alpha s^2-\bigl(\Phi(Q_s)-\Phi(Q^\star)\bigr)
\ \le\ (1+C_1)\,\alpha s^2.
\tag{$\ast\ast$}
\]

\emph{Conclusion.} The competitor \(Q_s\) is admissible: it is a probability
measure (we removed mass \(\alpha\) and added mass \(\alpha\)) with finite
second moment, since \(\mu_s\) has bounded support. Adding \(k^2\) times
\((\ast)\) to \((\ast\ast)\),
\[
J(Q_s)-J(Q^\star)
\ \le\
k^2\alpha\bigl(s^2-2\kappa s\bigr)+(1+C_1)\alpha s^2
\ =\
-2k^2\alpha\kappa\, s+\bigl(k^2+1+C_1\bigr)\alpha s^2 .
\]
The right-hand side is \(\alpha s\bigl[-2k^2\kappa+(k^2+1+C_1)s\bigr]\),
which is \emph{strictly negative} for every
\[
0<s<\frac{2k^2\kappa}{k^2+1+C_1},
\]
a nonempty range because \(k>0\) and \(\kappa>0\). For such \(s\) we have
\(J(Q_s)<J(Q^\star)\), contradicting the optimality of \(Q^\star\). Hence no
minimizer has an atom. Observe how the two hypotheses enter: \(k>0\) makes
the linear gain present at all, and \(\sigma_I^2>0\) (from the strict
positivity of the density of \(P\)) makes \(\kappa>0\), so the gain does not
degenerate.
\end{proof}

\begin{remark}
\label{rem:k-zero}
The hypothesis \(k>0\) is essential: the entire contradiction rests on the
first-order transport gain \(-2k^2\alpha\kappa\,s\). At \(k=0\) the
objective reduces to \(J=\operatorname{mmse}\), which is minimized by the
atom \(Q^\star=\delta_0\) (with \(\operatorname{mmse}(\delta_0)=0\)); so
the conclusion genuinely fails without a strictly positive transport
penalty.
\end{remark}

Thus, in the Witsenhausen setting (e.g.\ Gaussian \(P\)),
Proposition~\ref{prop:ac} guarantees that \(Q^\star\) has no atoms, and by
Fact~\ref{fact:regularity} it has a smooth density \(q^\star\).
Correspondingly, the optimal first-stage controller (the optimal
transport map from \(P\) to \(Q^\star\)) is strictly increasing. As
\(k\to\infty\) it converges to the identity, and the optimal law converges to
\(P\), so the controllers approach the classical affine pair
(Proposition~\ref{prop:asymptotics}(a) and
Corollary~\ref{cor:limit-controllers}(a)). We stress that this is an
\emph{asymptotic} statement: at no finite \(k\) do we prove that \(Q^\star\)
is Gaussian or that the optimal map is affine. Whether affine control is
globally optimal above the stability threshold \(k_{\mathrm{stab}}\) of
Proposition~\ref{prop:kstab} is open (Remark~\ref{rem:hermite}).

\subsection{The MMSE--Fisher Identity}
\label{sec:fisher}

Before deriving the optimality condition, we record an exact rewriting of
the estimation term that replaces its posterior-mean functionals by a
single classical quantity, the Fisher information of the noisy
observation. This identity clarifies the structure of the problem, yields
sharp bounds, and, as we show in \S\ref{sec:EL}, gives the optimality
condition its cleanest, \emph{score} form. It is an interpretive lens: the
results of \S\ref{sec:EL} onward can all be phrased through
\(\operatorname{mmse}\) directly, but the Fisher form is more transparent.

\begin{definition}[Fisher information]
\label{def:fisher}
For an almost-everywhere positive, differentiable density \(\rho\) on
\(\mathbb{R}\), its (location) Fisher information is
\[
\mathcal{I}(\rho):=\int_{\mathbb{R}}\frac{\rho'(z)^2}{\rho(z)}\,dz
=\int_{\mathbb{R}}\bigl((\log\rho)'(z)\bigr)^2\rho(z)\,dz .
\]
\end{definition}

\begin{proposition}[MMSE--Fisher identity]
\label{prop:mmse-fisher}
Let \(Y\sim Q\in\mathcal{P}_2(\mathbb{R})\) and \(N\sim\mathcal{N}(0,1)\)
be independent, and let \(p_Z=Q*\phi\) be the density of \(Z=Y+N\). Then
\begin{equation}
\operatorname{mmse}(Q)=1-\mathcal{I}(p_Z).
\label{eq:mmse-fisher}
\end{equation}
\end{proposition}

\begin{proof}
\emph{Step 1 (Tweedie's formula).} Since \(\phi'(u)=-u\phi(u)\),
differentiating \(p_Z(z)=\int\phi(z-y)\,Q(dy)\) under the integral gives
\[
p_Z'(z)=\int\phi'(z-y)\,Q(dy)=-\int(z-y)\phi(z-y)\,Q(dy)
=-z\,p_Z(z)+\int y\,\phi(z-y)\,Q(dy).
\]
Solving for \(\int y\,\phi(z-y)Q(dy)=p_Z'(z)+z\,p_Z(z)\) and dividing by
\(p_Z(z)\), the posterior mean
\(m(z):=\mathbb{E}[Y\mid Z=z]=\int y\,\phi(z-y)Q(dy)/p_Z(z)\) satisfies
\emph{Tweedie's formula} \cite{guo2005mutual}
\[
m(z)=z+\frac{p_Z'(z)}{p_Z(z)}=z+(\log p_Z)'(z).
\]
\emph{Step 2 (posterior variance).} Let
\(g(z):=p_Z(z)m(z)=\int y\,\phi(z-y)Q(dy)\). The same differentiation gives
\[
g'(z)=\int y\,\phi'(z-y)\,Q(dy)=-\int y(z-y)\phi(z-y)Q(dy)
=-z\,g(z)+\int y^2\phi(z-y)Q(dy),
\]
i.e.\ \(g'(z)=p_Z(z)\bigl(\mathbb{E}[Y^2\mid Z=z]-z\,m(z)\bigr)\). From
\(m=g/p_Z\), the quotient rule gives
\(m'=g'/p_Z-m\,(p_Z'/p_Z)=g'/p_Z-m\,(\log p_Z)'\); substituting
\(g'/p_Z=\mathbb{E}[Y^2\mid Z]-z\,m\) and \((\log p_Z)'=m-z\),
\[
m'(z)=\bigl(\mathbb{E}[Y^2\mid Z=z]-z\,m\bigr)-m(m-z)
=\mathbb{E}[Y^2\mid Z=z]-m(z)^2=\operatorname{Var}(Y\mid Z=z).
\]
\emph{Step 3 (assemble).} Differentiating Tweedie's formula gives
\(m'=1+(\log p_Z)''\), so \(\operatorname{Var}(Y\mid Z=z)=1+(\log p_Z)''(z)\).
Integrating against \(p_Z\),
\[
\operatorname{mmse}(Q)=\int\operatorname{Var}(Y\mid Z=z)\,p_Z(z)\,dz
=1+\int(\log p_Z)''\,p_Z\,dz .
\]
Finally, \((\log p_Z)''=p_Z''/p_Z-(p_Z'/p_Z)^2\), so
\(\int(\log p_Z)''p_Z=\int p_Z''-\int(p_Z')^2/p_Z=0-\mathcal{I}(p_Z)\)
(the term \(\int p_Z''=0\) as \(p_Z'\to0\) at \(\pm\infty\)); this
gives~\eqref{eq:mmse-fisher}.
\end{proof}

Substituting~\eqref{eq:mmse-fisher} into the
reformulation~\eqref{eq:OT-variational} recasts the whole problem as a
competition between transport and Fisher information.

\begin{corollary}[Fisher-information variational principle]
\label{cor:fisher-var}
The Witsenhausen optimal value satisfies
\begin{equation}
J^\star=1-\sup_{Q\in\mathcal{P}_2(\mathbb{R})}
\Bigl\{\mathcal{I}(Q*\phi)-k^2 W_2^2(P,Q)\Bigr\}.
\label{eq:fisher-var}
\end{equation}
Equivalently, the first controller seeks a law \(Q\), close to \(P\) in
\(W_2\), whose Gaussian smoothing \(Q*\phi\) carries \emph{large} Fisher
information.
\end{corollary}

The Fisher form immediately re-derives the bounds of
Lemma~\ref{lem:mmse-bounds} and pinpoints the Gaussian as the worst case
for estimation.

\begin{corollary}[Estimation bounds via Cram\'er--Rao and Stam]
\label{cor:fisher-bounds}
For \(Q\in\mathcal{P}_2(\mathbb{R})\) with variance \(v\),
\[
\frac{1}{v+1}\ \le\ \mathcal{I}(Q*\phi)\ \le\ 1,
\qquad\text{hence}\qquad
0\ \le\ \operatorname{mmse}(Q)\ \le\ \frac{v}{v+1},
\]
recovering Lemma~\ref{lem:mmse-bounds}. The lower bound on
\(\mathcal{I}\) is the Cram\'er--Rao inequality
\(\mathcal{I}(p_Z)\ge 1/\operatorname{Var}(Z)\) with
\(\operatorname{Var}(Z)=v+1\), with equality precisely when \(Z\) is Gaussian --- and then \(Q\) is Gaussian
too, since \(Z=Y+N\) with \(Y\perp N\) and \(N\) Gaussian, by Cram\'er's
decomposition theorem \cite{lukacs1970characteristic}; the upper bound is the Stam convolution inequality
\(\mathcal{I}(Y+N)\le\mathcal{I}(N)=1\). Consequently, among laws of a
fixed variance the Gaussian \emph{minimizes} \(\mathcal{I}(Q*\phi)\) and
therefore \emph{maximizes} the estimation cost: the Gaussian ansatz of
\S\ref{sec:gaussian-class} is the worst case for estimation, and every
improvement must come from a non-Gaussian \(Q\).
\end{corollary}

\subsection{Euler--Lagrange Characterization}
\label{sec:EL}

Throughout this subsection we assume \(P\) has a smooth, strictly
positive density \(p\), and, by Proposition~\ref{prop:ac-gauss} and
Fact~\ref{fact:regularity}, that every minimizer \(Q^\star\) is
absolutely continuous with interval support and density \(q^\star>0\)
\emph{almost everywhere} on that support. We do not assume pointwise
positivity, which Proposition~\ref{prop:ac-gauss} deliberately does not claim;
what the arguments below actually use is the consequence that
\begin{equation}
Q^\star(I')>0\quad\text{for every nonempty open subinterval }
I'\subseteq\operatorname{supp}Q^\star ,
\label{eq:charges-subintervals}
\end{equation}
which follows since \(I'\) has positive Lebesgue measure and
\(q^\star>0\) a.e.\ on it. Let \(Y\sim Q^\star\),
\(N\sim\mathcal{N}(0,1)\), \(Z=Y+N\), and let \(\phi\) be the standard
normal density. We compute the first variation of
\(J(Q)=k^2 W_2^2(P,Q)+\operatorname{mmse}(Q)\) term by term.

\begin{lemma}[First variation of the transport term]
\label{lem:wass-derivative}
Let \(\varphi_Q\) be a Kantorovich potential for the quadratic-cost
optimal transport from \(Q\) to \(P\), normalized by
\[
\tfrac12 W_2^2(P,Q)=\int\varphi_Q\,dQ+\int\varphi_Q^{c}\,dP,
\qquad c(x,y)=\tfrac12(x-y)^2 .
\]
Then the linear first variation of \(Q\mapsto\tfrac12 W_2^2(P,Q)\) is
represented, up to an additive constant, by \(\varphi_Q\); consequently
\begin{equation}
\frac{\delta}{\delta Q(y)}\bigl[k^2 W_2^2(P,Q)\bigr]=2k^2\,\varphi_Q(y).
\label{eq:transport-firstvar}
\end{equation}
Moreover, in one dimension \(\varphi_Q\) is differentiable \(Q\)-a.e.\
with
\begin{equation}
\varphi_Q'(y)=y-T_Q(y),
\qquad T_Q:=F_P^{-1}\circ F_Q\quad(\,T_{Q\,\#}Q=P\,).
\label{eq:kantorovich-deriv}
\end{equation}
\end{lemma}

\begin{proof}
By Fact~\ref{fact:wass-var},
\(\tfrac{\delta}{\delta Q}[\tfrac12 W_2^2(P,Q)]=\varphi_Q\) and
\(\varphi_Q'(y)=y-T_Q(y)\), which is~\eqref{eq:kantorovich-deriv}. Since
\(k^2 W_2^2=2k^2\cdot\tfrac12 W_2^2\) and the first variation is linear in
the functional, the functional derivative of the transport term is
\(2k^2\varphi_Q\), which is~\eqref{eq:transport-firstvar}.
\end{proof}

Next we compute the estimation term's first variation explicitly.

\begin{lemma}[First variation of the MMSE term]
\label{lem:mmse-derivative}
For \(Q\) with density \(q\), define the mixture density, its first-moment
density, and the posterior mean
\[
p_Z^q(z)=\int q(u)\,\phi(z-u)\,du,\quad
g_q(z)=\int u\,q(u)\,\phi(z-u)\,du,\quad
m_q(z)=\frac{g_q(z)}{p_Z^q(z)}=\mathbb{E}_q[Y\mid Z=z],
\]
together with the Gaussian-smoothed posterior functionals
\begin{equation}
a_q(y)=\int_{\mathbb{R}} m_q(z)\,\phi(z-y)\,dz,
\qquad
b_q(y)=\int_{\mathbb{R}} m_q(z)^2\,\phi(z-y)\,dz.
\label{eq:aq-bq-def}
\end{equation}
Then
\begin{equation}
\frac{\delta}{\delta q(y)}\operatorname{mmse}(q)
=
y^2-2y\,a_q(y)+b_q(y).
\label{eq:delta-mmse}
\end{equation}
\end{lemma}

\begin{proof}
Write \(\operatorname{mmse}(q)=\int y^2 q(y)\,dy-\Phi(q)\) with
\(\Phi(q):=\int m_q(z)^2 p_Z^q(z)\,dz=\int g_q(z)^2/p_Z^q(z)\,dz\), where
\(g_q,p_Z^q\) are linear in \(q\). For a variation
\(q_\varepsilon=q+\varepsilon h\) with \(\int h=0\), one has
\(\partial_\varepsilon g_{q_\varepsilon}(z)=\int u\,h(u)\phi(z-u)\,du\)
and \(\partial_\varepsilon p_Z^{q_\varepsilon}(z)=\int
h(u)\phi(z-u)\,du\) (both linear in \(q\)). Since
\(\partial_\varepsilon(g^2/p)=2(g/p)\,\partial_\varepsilon g
-(g/p)^2\,\partial_\varepsilon p=2m_q\,\partial_\varepsilon g
-m_q^2\,\partial_\varepsilon p\),
\[
\delta\Phi(q)[h]
=\int\bigl(2m_q(z)\,\partial_\varepsilon g(z)-m_q(z)^2\,\partial_\varepsilon
p(z)\bigr)\,dz .
\]
Substitute the two variations and swap the order of the \(z\)- and
\(u\)-integrals (Fubini; the integrands are absolutely integrable as
\(m_q\) has at most linear growth and \(\phi\) is Schwartz). The first
term becomes
\(\int h(u)\bigl(2u\!\int m_q(z)\phi(z-u)\,dz\bigr)du\) and the second
\(\int h(u)\bigl(\int m_q(z)^2\phi(z-u)\,dz\bigr)du\); renaming \(u\to y\),
\[
\delta\Phi(q)[h]
=\int h(y)\Bigl(2y\!\int m_q(z)\phi(z-y)\,dz-\!\int
m_q(z)^2\phi(z-y)\,dz\Bigr)dy .
\]
By the definitions~\eqref{eq:aq-bq-def} this is
\(\int h(y)\,(2y\,a_q(y)-b_q(y))\,dy\), so
\(\delta\Phi/\delta q=2y\,a_q-b_q\). Since
\(\delta(\int y^2 q)/\delta q=y^2\), subtracting
yields~\eqref{eq:delta-mmse}.
\end{proof}

Combining the two lemmas gives the optimality condition.

\begin{theorem}[Euler--Lagrange characterization; score form]
\label{thm:EL}
Let \(q^\star\) be the density of a minimizer \(Q^\star\), let
\(p_Z^\star=Q^\star*\phi\) with \emph{score} \(s^\star:=(\log p_Z^\star)'\),
and write \(a^\star=a_{q^\star}\), \(b^\star=b_{q^\star}\),
\(m^\star=m_{q^\star}\), \(\varphi^\star=\varphi_{Q^\star}\), and
\(T^\star=F_P^{-1}\circ F_{Q^\star}\). There is a constant
\(\lambda\in\mathbb{R}\) such that the cost sensitivity
\[
\Lambda^\star(y):=2k^2\varphi^\star(y)+y^2-2y\,a^\star(y)+b^\star(y)
\]
obeys the complementary-slackness conditions
\begin{equation}
\Lambda^\star(y)=\lambda\quad(Q^\star\text{-a.e.}),
\qquad
\Lambda^\star(y)\ge\lambda\quad(\text{Lebesgue-a.e.}).
\label{eq:EL}
\end{equation}
On the interval \(\operatorname{supp}Q^\star\) equality holds throughout,
and differentiating yields the optimality condition in \emph{score form}
\begin{equation}
2k^2\bigl(y-T^\star(y)\bigr)
+\frac{d}{dy}\Bigl[\bigl(2\,s^{\star\prime}+(s^\star)^2\bigr)*\phi\Bigr](y)=0
\qquad\text{for }y\in\operatorname{int}\operatorname{supp}Q^\star ,
\label{eq:score-EL}
\end{equation}
or equivalently, through the posterior functionals \(a^\star,b^\star\),
\begin{equation}
2k^2\bigl(y-T^\star(y)\bigr)
+\frac{d}{dy}\Bigl(y^2-2y\,a^\star(y)+b^\star(y)\Bigr)=0 .
\label{eq:EL-differential}
\end{equation}
\end{theorem}

\begin{proof}
By Lemmas~\ref{lem:wass-derivative} and~\ref{lem:mmse-derivative}, the
first variation of \(J\) at \(Q^\star\) along a signed measure \(h\) of total
mass zero is
\begin{equation}
\delta J(Q^\star)[h]
:=\frac{d}{dt}\Big|_{t=0^+}J(Q^\star+th)
=\int \Lambda^\star\,dh,
\qquad
\Lambda^\star=2k^2\varphi^\star+y^2-2y\,a^\star+b^\star ,
\label{eq:firstvar-J}
\end{equation}
valid for finite, compactly supported signed measures \(h\) --- exactly as in the proof of
Proposition~\ref{prop:ac}, where the same differentiation under the integral
sign was justified using that \(p_Z^{Q^\star}=Q^\star*\phi\) is smooth and
strictly positive. Note \(\Lambda^\star\) is continuous.

\emph{Step 1: the exchange inequality.} Rather than invoke a multiplier rule
abstractly, we read the conditions off a single family of admissible
perturbations, which \emph{move mass} from where \(Q^\star\) lives to an
arbitrary test point. Fix a \emph{bounded} Borel set \(A\) with
\(\alpha:=Q^\star(A)>0\) (boundedness costs nothing: intersect any \(A\) of
positive measure with a large enough interval) and
a point \(y_1\in\mathbb R\), and let \(\mu_1\) be any probability measure
supported in a small ball around \(y_1\). Put
\[
h:=\mu_1-\frac{1}{\alpha}\,Q^\star|_A ,
\qquad
Q_t:=Q^\star+t\,h\quad(0\le t\le\alpha).
\]
Then \(h\) has total mass \(1-1=0\), and \(Q_t\ge0\) for \(t\in[0,\alpha]\)
because the only mass removed is \((t/\alpha)Q^\star|_A\le Q^\star|_A\).
So each \(Q_t\) is an admissible competitor, and since \(Q^\star\) minimizes
\(J\) we have \(J(Q_t)\ge J(Q^\star)\) for all such \(t\); dividing by
\(t>0\) and letting \(t\downarrow0\),
\[
0\ \le\ \delta J(Q^\star)[h]=\int\Lambda^\star\,dh
=\int\Lambda^\star\,d\mu_1-\frac1\alpha\int_A\Lambda^\star\,dQ^\star .
\]
Shrinking the ball and using continuity of \(\Lambda^\star\) gives the
\emph{exchange inequality}
\begin{equation}
\Lambda^\star(y_1)\ \ge\ \frac1{Q^\star(A)}\int_A\Lambda^\star\,dQ^\star
\qquad\text{for every }y_1\in\mathbb R\text{ and every }A\text{ with }
Q^\star(A)>0 .
\label{eq:exchange}
\end{equation}
In words: no point of the line is a strictly cheaper place to put mass than
the average cost of the mass \(Q^\star\) has already placed on \(A\). This is
the whole content of optimality at first order.

\emph{Step 2: the multiplier.} Let
\(\lambda:=\operatorname*{ess\,inf}_{Q^\star}\Lambda^\star\), the essential
infimum of \(\Lambda^\star\) with respect to \(Q^\star\); this is the
promised multiplier. Since \(\Lambda^\star\ge\lambda\) holds
\(Q^\star\)-a.e.\ by definition, the right-hand side of~\eqref{eq:exchange}
is an average of quantities \(\ge\lambda\) and so is itself \(\ge\lambda\),
for \emph{every} admissible \(A\). Hence \(\Lambda^\star(y_1)\ge\lambda\) for
every \(y_1\in\mathbb R\), which is the second
half of~\eqref{eq:EL}. For the first half, suppose
\(Q^\star(B)>0\) for \(B:=\{\Lambda^\star>\lambda+\varepsilon\}\) and some
\(\varepsilon>0\). Applying~\eqref{eq:exchange} with \(A=B\) gives
\(\Lambda^\star(y_1)\ge\lambda+\varepsilon\) for \emph{every} \(y_1\), hence
also \(Q^\star\)-essentially, contradicting the definition of \(\lambda\) as
the \(Q^\star\)-essential infimum. So \(Q^\star(B)=0\) for every
\(\varepsilon>0\), i.e.\ \(\Lambda^\star\le\lambda\) holds
\(Q^\star\)-a.e.; combined with \(\Lambda^\star\ge\lambda\) everywhere this
gives \(\Lambda^\star=\lambda\) \(Q^\star\)-a.e. This
proves~\eqref{eq:EL}, and shows the multiplier is not an extra unknown but
simply the common value of the cost sensitivity on the support.

\emph{Step 3: the differential form.} By Proposition~\ref{prop:ac-gauss} the
minimizer is absolutely continuous with support an interval \(I\) on which
\(q^\star>0\) almost everywhere; since \(\Lambda^\star\) is continuous and
equals \(\lambda\)
\(Q^\star\)-a.e., and \(Q^\star\) charges every nonempty open subinterval of
\(I\) by~\eqref{eq:charges-subintervals}, we get
\(\Lambda^\star\equiv\lambda\) pointwise on \(I\). (Explicitly: if
\(\Lambda^\star(y_0)\ne\lambda\) at some \(y_0\in I\), continuity would give a
subinterval around \(y_0\) on which \(\Lambda^\star\ne\lambda\), and that
subinterval carries positive \(Q^\star\)-mass, contradicting
\(\Lambda^\star=\lambda\) \(Q^\star\)-a.e. Note that only positivity
\emph{a.e.} is needed, not at every point.) A function that is
constant on an interval has vanishing derivative there, so
\(\frac{d}{dy}\Lambda^\star=0\) on \(I\).

One step needs care before substituting, because
\eqref{eq:kantorovich-deriv} supplies \(\varphi^{\star\prime}(y)=y-T^\star(y)\)
only \(Q^\star\)-a.e., whereas \eqref{eq:score-EL}--\eqref{eq:EL-differential}
are asserted at \emph{every} interior point of the support. The gap closes by
continuity. On \(\operatorname{int}I\) the distribution function
\(F_{Q^\star}\) is continuous (\(Q^\star\) is atomless) and strictly increasing
(\(q^\star>0\) a.e.\ there, so \(Q^\star\) charges every subinterval
by~\eqref{eq:charges-subintervals}), with values in \((0,1)\); and \(F_P^{-1}\)
is continuous on \((0,1)\) since \(P\) has a strictly positive density. Hence
\(T^\star=F_P^{-1}\circ F_{Q^\star}\) is \emph{continuous} on
\(\operatorname{int}I\). The locally Lipschitz function \(\varphi^\star\) thus
has an a.e.\ derivative agreeing with the continuous function
\(y\mapsto y-T^\star(y)\), so \(\varphi^\star\in C^1(\operatorname{int}I)\)
with \(\varphi^{\star\prime}(y)=y-T^\star(y)\) at \emph{every} interior point.
Substituting this into
\(\frac{d}{dy}\bigl[2k^2\varphi^\star+y^2-2y\,a^\star+b^\star\bigr]=0\)
gives~\eqref{eq:EL-differential} pointwise on \(\operatorname{int}I\), the
remaining terms \(a^\star,b^\star\) being smooth
(Lemma~\ref{lem:mmse-derivative}).
For the score form, the MMSE--Fisher identity~\eqref{eq:mmse-fisher} gives
\(\frac{\delta}{\delta Q}\operatorname{mmse}
=-\frac{\delta}{\delta Q}\mathcal I(Q*\phi)\); the first variation of
Fisher information is, with \(s=(\log\rho)'\) and
\(s'=\rho''/\rho-(\rho'/\rho)^2\),
\[
\frac{\delta\mathcal I}{\delta\rho}
=-2\frac{\rho''}{\rho}+\Bigl(\frac{\rho'}{\rho}\Bigr)^2
=-\bigl(2s'+s^2\bigr),
\]
and since \(\rho=Q*\phi\) is linear in \(Q\) with \(\phi\) even, the chain
rule gives \(\frac{\delta}{\delta Q}\operatorname{mmse}=(2s'+s^2)*\phi\).
This agrees with the representative \(y^2-2y\,a^\star+b^\star\) of
Lemma~\ref{lem:mmse-derivative} up to an additive constant. Indeed
\(\phi(z-y)\) is the density of \(Z\) given \(Y=y\), so
\(a^\star(y)=\mathbb E[m^\star(Z)\mid Y=y]\) and
\(b^\star(y)=\mathbb E[m^\star(Z)^2\mid Y=y]\)
(Remark~\ref{rem:psi-closed}); hence, writing \(Z=y+N\) and using Stein's
identity \(\mathbb E[N\,s^\star(y+N)]=\mathbb E[s^{\star\prime}(y+N)]\),
\begin{equation}
y^2-2y\,a^\star(y)+b^\star(y)
=\mathbb E\bigl[(y-m^\star(Z))^2\bigm|Y=y\bigr]
=1+\bigl[(2s^{\star\prime}+(s^\star)^2)*\phi\bigr](y),
\label{eq:psi-score}
\end{equation}
using \(m^\star(z)=z+s^\star(z)\), so that \(y-m^\star(y+N)=-(N+s^\star(y+N))\).
Two representatives of a first variation that differ by a constant define the
same variation along mass-zero perturbations, so either may be used
in~\eqref{eq:EL}; substituting the second one
in~\eqref{eq:EL-differential} (the constant drops on differentiation)
yields~\eqref{eq:score-EL}.
\end{proof}

\begin{remark}
The score form~\eqref{eq:score-EL} says the transport pull
\(2k^2(y-T^\star(y))\) is balanced by the gradient of the Fisher
sensitivity \((2s^{\star\prime}+(s^\star)^2)*\phi\), the first variation
of \(\mathcal I(Q^\star*\phi)\). Expanding the equivalent
form~\eqref{eq:EL-differential},
\(\frac{d}{dy}(y^2-2y\,a^\star+b^\star)
=2y-2(a^\star+y\,a^{\star\prime})+b^{\star\prime}\), exhibits it as a
nonlinear, nonlocal integral equation coupling the transport map
\(T^\star\) with the posterior structure \(m^\star\). It admits no
closed-form solution but supports asymptotic analysis (small or large
\(k\)) and fixed-point numerical schemes in the space of probability
measures.
\end{remark}

\subsection{A Semi-Closed-Form Solution in the Gaussian Class}
\label{sec:gaussian-class}

To obtain an analytically tractable benchmark, we restrict the
minimization problem
\begin{equation}
J^\star
=
\inf_{Q}
\Bigl\{
k^2 W_2^2(P,Q)
+
\operatorname{mmse}(Q)
\Bigr\}
\label{eq:OT-variational-semi}
\end{equation}
to a Gaussian parametric family for $Q$.

We restrict $Q$ to the centered Gaussian family
$Q_\tau:=\mathcal{N}(0,\tau^2)$, $\tau>0$, and set
\[
J_{\mathrm{G}}^\star:=\inf_{\tau>0}J(\tau),
\qquad
J(\tau):=k^2 W_2^2\bigl(\mathcal{N}(0,\sigma^2),\mathcal{N}(0,\tau^2)\bigr)
+\operatorname{mmse}\bigl(\mathcal{N}(0,\tau^2)\bigr).
\]

\begin{proposition}[Gaussian-class objective]
\label{prop:gauss-obj}
For $P=\mathcal{N}(0,\sigma^2)$ and $Q_\tau=\mathcal{N}(0,\tau^2)$,
\begin{equation}
J(\tau)=k^2(\sigma-\tau)^2+\frac{\tau^2}{1+\tau^2},\qquad\tau>0.
\label{eq:J-tau}
\end{equation}
\end{proposition}

\begin{proof}
\emph{Transport.} For centered one-dimensional Gaussians the monotone
optimal map from $\mathcal N(0,\sigma^2)$ to $\mathcal N(0,\tau^2)$ is the
linear scaling $x\mapsto(\tau/\sigma)x$, so
\[
W_2^2\bigl(\mathcal N(0,\sigma^2),\mathcal N(0,\tau^2)\bigr)
=\int\Bigl(x-\tfrac{\tau}{\sigma}x\Bigr)^2 dP
=\Bigl(1-\tfrac{\tau}{\sigma}\Bigr)^2\!\int x^2\,dP
=\Bigl(1-\tfrac{\tau}{\sigma}\Bigr)^2\sigma^2=(\sigma-\tau)^2 .
\]
\emph{Estimation.} With $Y\sim\mathcal N(0,\tau^2)$ and $Z=Y+N$,
$N\sim\mathcal N(0,1)$ independent, the pair $(Y,Z)$ is jointly Gaussian
with $\operatorname{Cov}(Y,Z)=\tau^2$ and
$\operatorname{Var}(Z)=\tau^2+1$; the Gaussian conditional-variance formula
gives, constant in $z$,
\[
\operatorname{Var}(Y\mid Z=z)
=\operatorname{Var}(Y)-\frac{\operatorname{Cov}(Y,Z)^2}{\operatorname{Var}(Z)}
=\tau^2-\frac{\tau^4}{\tau^2+1}
=\frac{\tau^2(\tau^2+1)-\tau^4}{\tau^2+1}
=\frac{\tau^2}{1+\tau^2},
\]
so $\operatorname{mmse}(Q_\tau)=\mathbb E[\operatorname{Var}(Y\mid Z)]
=\tau^2/(1+\tau^2)$. Adding the two terms gives~\eqref{eq:J-tau}.
\end{proof}

\begin{figure}[t]
\centering
\begin{tikzpicture}
\begin{axis}[width=0.7\linewidth,height=5.4cm,
  xlabel={\(\tau\)},ylabel={cost},
  domain=0:3.5,samples=200,xmin=0,xmax=3.5,ymin=0,ymax=1.6,
  legend style={at={(0.03,0.97)},anchor=north west,font=\small,draw=none,fill=none},
  clip=true]
\addplot[densely dotted,line width=1pt,black,mark=none] {0.25*(2-x)^2};
\addlegendentry{transport \(k^2(\sigma-\tau)^2\)}
\addplot[densely dashed,line width=1pt,black,mark=none] {x^2/(1+x^2)};
\addlegendentry{estimation \(\operatorname{mmse}\)}
\addplot[line width=1.4pt,black,mark=none] {0.25*(2-x)^2 + x^2/(1+x^2)};
\addlegendentry{total \(J(\tau)\)}
\addplot[densely dashed,gray,mark=none] coordinates {(1,0) (1,0.75)};
\addplot[only marks,mark=*,mark size=1.7pt,black] coordinates {(1,0.75)};
\node[font=\scriptsize,anchor=west] at (axis cs:1.1,0.08){\(\tau^\star\)};
\node[font=\scriptsize,anchor=south west] at (axis cs:1.1,0.79){\(J^\star_{\mathrm G}\)};
\end{axis}
\end{tikzpicture}
\caption{The transport--estimation trade-off in the Gaussian class
(\(\sigma=2\), \(k=0.5\)). Shrinking \(\tau\) toward \(0\) cheapens
estimation but raises the transport cost of cancelling \(X_0\); enlarging
\(\tau\) toward \(\sigma\) does the reverse. The optimal \(\tau^\star\)
balances the two (Proposition~\ref{prop:gauss-obj}).}
\label{fig:tradeoff}
\end{figure}

Writing $s:=\tau$ for the standard deviation, we characterize the
stationary points and the convexity of $J$.

\begin{proposition}[Stationarity, convexity, and non-uniqueness]
\label{prop:gauss-foc}
The infimum $J_{\mathrm{G}}^\star$ is attained at an interior point
$s^\star>0$, and every stationary point satisfies the scalar equation
\begin{equation}
k^2(\sigma-s)=\frac{s}{(1+s^2)^2},
\label{eq:FOC-s-alt}
\end{equation}
equivalently the quintic
\begin{equation}
k^2(\sigma-s)(1+2s^2+s^4)-s=0.
\label{eq:quintic-s}
\end{equation}
Moreover $J$ is strictly convex on $(0,\infty)$ when $k^2>\tfrac14$, in
which case \eqref{eq:FOC-s-alt} has a unique root and $s^\star$ is unique.
For $k^2<\tfrac14$, $J$ may be nonconvex and \eqref{eq:FOC-s-alt} may
possess several positive roots, for instance three when $k=0.2$,
$\sigma=5$ (two local minima separated by a local maximum), and the
global minimizer is selected by comparing the values $J(s)$.
\end{proposition}

\begin{proof}
With $s=\tau$, $J(s)=k^2(\sigma-s)^2+s^2/(1+s^2)$ is continuous on
$(0,\infty)$, with $J(s)\to\infty$ as $s\to\infty$ and, since
\[
J'(s)=2k^2(s-\sigma)+\frac{2s}{(1+s^2)^2},
\]
$J'(0^+)=-2k^2\sigma<0$; hence the infimum is attained at an interior
stationary point. Setting $J'(s)=0$ gives~\eqref{eq:FOC-s-alt}, and
multiplying by $(1+s^2)^2$ gives the quintic~\eqref{eq:quintic-s}.
Differentiating $J'(s)=2k^2(s-\sigma)+2s(1+s^2)^{-2}$ once more, the
quotient rule gives
\[
\frac{d}{ds}\,\frac{2s}{(1+s^2)^2}
=\frac{2(1+s^2)^2-2s\cdot2(1+s^2)\cdot2s}{(1+s^2)^4}
=\frac{2(1+s^2)-8s^2}{(1+s^2)^3}
=\frac{2-6s^2}{(1+s^2)^3},
\]
so $J''(s)=2k^2+\dfrac{2-6s^2}{(1+s^2)^3}$. The rational term
$\psi(s):=(2-6s^2)/(1+s^2)^3$ attains its minimum where
$\psi'(s)=0$; a short computation gives
$\psi'(s)=24s(s^2-1)/(1+s^2)^4$, which vanishes at $s=1$ (the only
positive root), where $\psi(1)=(2-6)/2^3=-\tfrac12$, and $\psi(0^+)=2$, $\psi\to0^{-}$
as $s\to\infty$, so $\min_{s>0}\psi=-\tfrac12$. Hence
$J''(s)\ge2k^2-\tfrac12$ for all $s$, and $J''>0$ everywhere iff
$k^2>\tfrac14$, giving strict convexity and a unique minimizer (at
$k^2=\tfrac14$ one has $J''\ge0$ with equality only at $s=1$, so $J$ is still
convex). For $k^2<\tfrac14$ the term $2-6s^2$ renders $J''$
negative near $s=1$, so $J$ can be nonconvex; a direct evaluation at
$k=0.2$, $\sigma=5$ exhibits three positive roots
of~\eqref{eq:FOC-s-alt}. This residual nonconvexity within the Gaussian
class already foreshadows the genuine nonconvexity underlying
Witsenhausen's counterexample. Each stationary point is a root of the
explicit algebraic equation~\eqref{eq:quintic-s}, so $s^\star$ is
\emph{semi-closed}.
\end{proof}

\begin{figure}[t]
\centering
\begin{tikzpicture}
\begin{axis}[width=0.7\linewidth,height=5.4cm,
  xlabel={\(s\)},ylabel={},domain=0:5.2,samples=200,
  xmin=0,xmax=5.2,ymin=0,ymax=1.05,restrict y to domain=0:1.05,
  legend style={at={(0.98,0.98)},anchor=north east,font=\small,draw=none,fill=none},
  clip=true]
\addplot[line width=1.4pt,black,mark=none] {x/(1+x^2)^2};
\addlegendentry{\(s/(1+s^2)^2\)}
\addplot[densely dashed,line width=1pt,black,mark=none] {0.04*(5-x)};
\addlegendentry{\(k^2(\sigma-s)\), \(k=0.2\)}
\addplot[densely dotted,line width=1pt,black,mark=none] {0.25*(5-x)};
\addlegendentry{\(k^2(\sigma-s)\), \(k=0.5\)}
\addplot[densely dotted,gray,mark=none] coordinates {(0.22,0) (0.22,0.185)};
\addplot[densely dotted,gray,mark=none] coordinates {(1.5,0) (1.5,0.14)};
\addplot[densely dotted,gray,mark=none] coordinates {(4.8,0) (4.8,0.008)};
\addplot[only marks,mark=*,mark size=2.1pt,black] coordinates {(0.22,0.185) (1.5,0.14) (4.8,0.008)};
\node[font=\scriptsize,anchor=south west] at (axis cs:0.7,0.42){\(k=0.2\): three roots};
\end{axis}
\end{tikzpicture}
\caption{Stationarity condition~\eqref{eq:FOC-s-alt} as the intersection
of the bump \(s/(1+s^2)^2\) with the line \(k^2(\sigma-s)\)
(\(\sigma=5\)). For \(k=0.2\) (below the threshold \(k^2=\tfrac14\)) the
line meets the bump \emph{three} times, two local minima of \(J\)
separated by a local maximum, whereas for \(k=0.5\) there is a single
crossing (Proposition~\ref{prop:gauss-foc}). This is the Gaussian-class
signature of Witsenhausen nonconvexity.}
\label{fig:foc}
\end{figure}

The scalar equation~\eqref{eq:FOC-s-alt} makes the two extreme regimes
transparent.

\begin{corollary}[Asymptotics of the Gaussian-class solution]
\label{cor:gauss-asymp}
As $k^2\to\infty$, $s^\star=\sigma+O(k^{-2})$, so the optimal Gaussian
approaches $P$. As $k^2\to0$, $s^\star=k^2\sigma+O(k^4)\to0$ and
$J(s^\star)=k^2\sigma^2+O(k^4)\to0$: the minimizer collapses toward the
point mass $\delta_0$, \emph{not} toward large variance.
\end{corollary}

\begin{proof}
Rewrite~\eqref{eq:FOC-s-alt} by dividing through by $k^2$, isolating the
displacement of $s$ from $\sigma$:
\begin{equation}
\sigma-s=\frac{s}{k^2(1+s^2)^2}.
\label{eq:FOC-rearranged}
\end{equation}

\emph{Large $k$.} The stationary point stays in the bounded range
$(0,\sigma)$: the right-hand side of~\eqref{eq:FOC-rearranged} is positive,
so $s<\sigma$, whence $s/(1+s^2)^2\le\sigma$ and the right-hand side is at
most $\sigma/k^2\to0$. Thus $\sigma-s^\star\to0$, i.e.\ $s^\star\to\sigma$.
Feeding this back into the right-hand side --- one step of fixed-point
iteration, legitimate because $s\mapsto s/(1+s^2)^2$ is continuous and
$s^\star\to\sigma$ ---
\[
\sigma-s^\star=\frac{s^\star}{k^2(1+(s^\star)^2)^2}
=\frac{\sigma+o(1)}{k^2\bigl[(1+\sigma^2)^2+o(1)\bigr]}
=\frac{\sigma}{k^2(1+\sigma^2)^2}+o(k^{-2}),
\]
so $s^\star=\sigma-\Theta(k^{-2})$ and the optimal Gaussian approaches $P$.

\emph{Small $k$.} Now the balance runs the other way: if $s^\star$ stayed
bounded away from $0$ then the right-hand side
of~\eqref{eq:FOC-rearranged} would blow up as $k\to0$ while the left-hand
side stays below $\sigma$, a contradiction; so $s^\star\to0$. For small $s$
expand the denominator, $s/(1+s^2)^2=s(1-2s^2+O(s^4))=s+O(s^3)$, and
substitute into~\eqref{eq:FOC-s-alt}:
\[
k^2(\sigma-s^\star)=s^\star+O((s^\star)^3)
\quad\Longrightarrow\quad
s^\star(1+k^2)=k^2\sigma+O((s^\star)^3)
\quad\Longrightarrow\quad
s^\star=\frac{k^2\sigma}{1+k^2}+O(k^6),
\]
the last step using $s^\star=O(k^2)$ to convert $O((s^\star)^3)$ into
$O(k^6)$. Expanding the geometric factor, $s^\star=k^2\sigma(1-k^2+\cdots)
=k^2\sigma+O(k^4)$. Substituting into
$J(s)=k^2(\sigma-s)^2+s^2/(1+s^2)$, the transport term is
$k^2\bigl(\sigma-k^2\sigma+O(k^4)\bigr)^2
=k^2\sigma^2(1-k^2+O(k^4))^2=k^2\sigma^2+O(k^4)$, and the estimation term is
$(s^\star)^2/(1+(s^\star)^2)=(s^\star)^2+O((s^\star)^4)=O(k^4)$; adding them,
$J(s^\star)=k^2\sigma^2+O(k^4)\to0$. Since
$\operatorname{mmse}(\mathcal{N}(0,\tau^2))=\tau^2/(1+\tau^2)$ is
increasing in $\tau$, small variance is what reduces the estimation cost;
the signalling benefit of a large spread, central to Witsenhausen's
counterexample, requires a non-Gaussian $Q$ and is invisible in this
family.
\end{proof}

\begin{remark}[Interpretation and limitations]
Within the Gaussian family, the optimal \(Q_\tau\) solves the scalar
equation~\eqref{eq:FOC-s-alt}, and the induced first controller is
\emph{linear} because both \(P\) and \(Q\) are Gaussian. By
the Hermite-mode analysis of Proposition~\ref{prop:kstab}, the \emph{high-variance
(near-identity)} Gaussian critical point --- the one with
\(\tau\) close to \(\sigma\), which is the Gaussian-class minimizer once
\(k\gtrsim0.2\) at \(\sigma=5\) --- is
\emph{second-variation stable} in the full problem for
\(k\ge k_{\mathrm{stab}}\), where \(k_{\mathrm{stab}}\approx0.735\) at
\(\sigma=5\) (and \(k_{\mathrm{stab}}\to2/e\) as \(\sigma\to\infty\)). For
\(k<k_{\mathrm{stab}}\) that branch gives a strict upper bound on
\(J^\star\), because
a rippled
(high-Hermite-mode) perturbation of \(Q\) strictly decreases the cost. It
nonetheless furnishes an explicit
one-dimensional trade-off between transport cost and MMSE and a natural
reference point for perturbation analysis around the Gaussian benchmark.
\end{remark}

\subsection{The Second-Variation Stability Threshold of the Gaussian Branch}
\label{sec:hermite}

This subsection explains, in words and then in formulas, why the
second-variation stability threshold \(k_{\mathrm{stab}}\) of the
\emph{high-variance} Gaussian
branch differs from the value
\(k^2=\tfrac14\) of Proposition~\ref{prop:gauss-foc}, and how the Gaussian
prior lets one compute it.

\emph{Which critical point.} By Proposition~\ref{prop:gauss-foc} the
stationarity condition~\eqref{eq:FOC-s-alt} can have several positive roots,
and the analysis below tests \emph{one root at a time}. Throughout,
\(k_{\mathrm{stab}}\) refers to the branch with \(\tau\) near \(\sigma\), the
near-identity or \emph{high-variance} branch; this is the branch that carries
the classical linear solution as \(k\to\infty\)
(Corollary~\ref{cor:gauss-asymp}) and, at \(\sigma=5\), the branch attaining
\(J_{\mathrm G}^\star\) for \(k\gtrsim0.2\). It is \emph{not} a statement about
every Gaussian critical point, and the low-variance branch behaves
differently; see the caveats at the end.

\emph{The question.} Proposition~\ref{prop:gauss-foc} asked only whether the
best Gaussian \(Q_\tau=\mathcal N(0,\tau^2)\) is a local minimum \emph{among
Gaussians}, a one-variable question in \(\tau\), answered by
\(k^2>\tfrac14\). But \(J\) is minimized over \emph{all} laws \(Q\), so the
sharper question is how the \emph{second variation} of \(J\) at \(Q_\tau\)
behaves against
small perturbations that make \(Q\)
\emph{non-Gaussian}. The largest \(k\) at which some admissible perturbation
direction
first makes the second variation negative is what we call the
\emph{stability threshold} \(k_{\mathrm{stab}}\) of that branch: below it the
branch is a
saddle and an explicit nonlinear competitor does strictly better, while above
it no tested direction has negative curvature. What
\(k\ge k_{\mathrm{stab}}\) certifies is exactly that --- second-variation
stability of one branch, along the directions tested --- and not local,
still less global, optimality; see the caveats in
Remark~\ref{rem:hermite}.

\emph{Gaussian critical points are genuine critical points.} Before
perturbing we check that there is something to perturb around: for every
root \(\tau\) of~\eqref{eq:FOC-s-alt}, on either branch and in particular at
the best Gaussian \(\tau^\star\), the law \(Q_\tau\) is
stationary for \(J\) not merely within the Gaussian family but in
\emph{every} mass-zero direction. Indeed \(T_{Q_\tau}(y)=(\sigma/\tau)y\), so
\(\varphi_{Q_\tau}(y)=\tfrac12(1-\sigma/\tau)y^2\) up to a constant
by~\eqref{eq:kantorovich-deriv}; and with \(m(z)=\beta z\),
\(\beta=\tau^2/\rho^2\), \(\rho^2=\tau^2+1\), formula
\eqref{eq:psi-condexp} gives
\(\Psi(y)=\mathbb E[((1-\beta)y-\beta N)^2]=y^2/\rho^4+\beta^2\). Hence the
cost sensitivity of Theorem~\ref{thm:EL} is
\[
\Lambda(y)=2k^2\varphi_{Q_\tau}(y)+\Psi(y)
=\Bigl[k^2\bigl(1-\tfrac{\sigma}{\tau}\bigr)+\tfrac{1}{(1+\tau^2)^2}\Bigr]y^2
+\text{const},
\]
which is constant in \(y\) --- the Euler--Lagrange
condition~\eqref{eq:EL} --- precisely when
\(k^2(\sigma-\tau)=\tau/(1+\tau^2)^2\), i.e.\ exactly at the stationarity
condition~\eqref{eq:FOC-s-alt} of Proposition~\ref{prop:gauss-foc}. So the
in-family first-order condition and the full Euler--Lagrange condition agree
on the Gaussian family (a useful consistency check on
Theorem~\ref{thm:EL}), and the second variation is what decides.

\emph{Second variation: a tug-of-war.} Perturb by a signed measure \(h\) of
total mass zero, \(Q_\varepsilon=Q_\tau+\varepsilon h\). By the previous
paragraph \(Q_\tau\) is a stationary point, so
\(J(Q_\varepsilon)=J(Q_\tau)+\tfrac12\varepsilon^2\,
\delta^2 J[h,h]+o(\varepsilon^2)\), and the \emph{sign of the second
variation} \(\delta^2 J[h,h]\) decides stability in the direction \(h\):
negative in some direction rules a local minimum \emph{out}, while
nonnegative in every direction is consistent with, but does not by itself
establish, a local minimum. Split
\(J=k^2W_2^2(P,\cdot)+\mathrm{mmse}\). The transport term is convex, so it
\emph{stabilizes} (curvature \(>0\), pulling \(Q\) back toward \(P\)); the
estimation term is concave (\(\mathrm{mmse}=1-\mathcal I\) with \(\mathcal I\)
convex), so it \emph{destabilizes} (curvature \(<0\), pushing \(Q\) away).
Because \(k^2\) multiplies only the transport curvature, large \(k\) keeps
the Gaussian stable; the Gaussian ceases to be a local minimum once, in
some direction \(h\), the estimation push overpowers the transport pull.
Concretely, writing \(\mathrm{mmse}=\mathbb E_Q[Y^2]-\Phi\) with
\(\Phi(Q)=\int g_Q^2/p_Z\) (Lemma~\ref{lem:mmse-derivative}), the identity
\((g^2/p)''=2(g_h-m\,p_h)^2/p\) integrates to the estimation curvature
\[
\delta^2\Phi[h,h]=2\int_{\mathbb R}\frac{\bigl(g_h(z)-m(z)\,p_h(z)\bigr)^2}{p_Z(z)}\,dz\ \ge0,
\qquad
g_h(z)=\!\int\! u\,h(u)\phi(z-u)\,du,\ \ p_h=h*\phi,
\]
so \(\delta^2\mathrm{mmse}=-\delta^2\Phi\le0\), confirming the signs above.

\emph{Hermite modes: a basis that diagonalizes the tug-of-war.} Rather than
test every \(h\), we choose a basis of perturbation directions in which the
second variation is \emph{diagonal} (no direction couples to another), so
stability reduces to checking one scalar inequality per direction. Because
the reference law is Gaussian, that basis is the \emph{Hermite functions}
\(h_n(y)=\mathrm{He}_n(y/\tau)\,q_\tau(y)\) (a Hermite polynomial times the
Gaussian density), the eigenfunctions of the Ornstein--Uhlenbeck operator,
orthogonal with \(\|h_n\|^2:=\int h_n^2/q_\tau=n!\). Each \(h_n\) deforms
the bell curve in one characteristic way:
\(n=1\) shifts the \emph{mean} (a pure translation);
\(n=2\) is a \emph{breathing} mode that changes the \emph{variance} (indeed
\(\partial_\tau q_\tau=\tau^{-1}h_2\)) and keeps \(Q\) Gaussian;
\(n=3\) adds \emph{skewness};
\(n=4\) is the first mode that changes the \emph{shape}, flattening the
single peak into two shoulders, the onset of \emph{bimodality}; and large
\(n\) is a fine \emph{ripple}, \(h_n\) oscillating with local wavenumber
\(\sqrt n/\tau\) across the bulk of \(Q_\tau\). Every mode decouples from
every other in this basis --- that is what ``diagonal'' means --- so no
symmetry argument is needed to separate them. Parity enters only through
admissibility: \(q_\tau+\varepsilon h_n\ge0\) requires
\(1+\varepsilon\mathrm{He}_n(y/\tau)\ge0\) for all \(y\), which for
\emph{even} \(n\) holds for every small \(\varepsilon>0\) (then
\(\mathrm{He}_n\to+\infty\) at both ends), whereas for odd \(n\) the
polynomial \(\mathrm{He}_n\) tends to \(-\infty\) in one tail and
\(q_\tau+\varepsilon h_n\) is negative there for every \(\varepsilon\ne0\).

\emph{Every mode is nonetheless an admissible tangent direction.} That
failure is an artefact of the \emph{straight-line} path
\(\varepsilon\mapsto q_\tau+\varepsilon h_n\), not of the direction \(h_n\).
The set of probability measures is not a linear space, and \(h_n\) is a
legitimate tangent vector to it provided some curve of probability measures
has \(h_n\) as its first-order variation. Such a curve is available for every
\(n\), odd or even: take the \emph{push-forward} curve
\begin{equation}
Q^{(n)}_\varepsilon:=(\mathrm{id}+\varepsilon v_n)_{\#}Q_\tau,
\qquad
v_n(y):=\tau\,\mathrm{He}_{n-1}(y/\tau) .
\label{eq:pushforward-curve}
\end{equation}
Each \(Q^{(n)}_\varepsilon\) is a probability measure for \emph{every}
\(\varepsilon\in\mathbb R\), being the image of one under a measurable map ---
no positivity constraint can be violated. Its first-order density variation is
\(-\frac{d}{dy}\bigl[q_\tau v_n\bigr]\), and by the identity
\(\frac{d}{dy}[\mathrm{He}_{n-1}(y/\tau)q_\tau(y)]
=-\tau^{-1}\mathrm{He}_n(y/\tau)q_\tau(y)\) used below, this is exactly
\[
-\frac{d}{dy}\Bigl[\tau\,\mathrm{He}_{n-1}(y/\tau)\,q_\tau(y)\Bigr]
=\mathrm{He}_n(y/\tau)\,q_\tau(y)=h_n(y).
\]
So \(h_n\) is tangent to an admissible curve for every \(n\).

\emph{The chart, and why the expansion along these curves is legitimate.}
We must say in what sense \(\varepsilon\mapsto J(Q_\varepsilon)\) is twice
differentiable, since that is precisely what the argument uses. We work in the
\emph{displacement chart}: perturbations are parametrized by a displacement
field \(v\), the curve is \(\varepsilon\mapsto(\mathrm{id}+\varepsilon
v)_\#Q_\tau\), and the object expanded is the \emph{real} function
\(F_v(\varepsilon):=J\bigl((\mathrm{id}+\varepsilon v)_\#Q_\tau\bigr)\). Take
\(v\) bounded with bounded derivative, say \(\|v'\|_\infty\le L\); for
\(|\varepsilon|<1/L\) the map \(\mathrm{id}+\varepsilon v\) is then strictly
increasing, so it \emph{is} the monotone rearrangement onto its image and
\(F_{Q_\varepsilon}^{-1}=(\mathrm{id}+\varepsilon v)\circ F_{Q_\tau}^{-1}\).
Both halves of \(F_v\) are then manifestly smooth:
\begin{itemize}
\item[--] by Fact~\ref{fact:1d-ot},
\(W_2^2(P,Q_\varepsilon)=\int_0^1\bigl(F_P^{-1}-G-\varepsilon v\circ
G\bigr)^2dt\) with \(G:=F_{Q_\tau}^{-1}\), an \emph{exact quadratic polynomial}
in \(\varepsilon\);
\item[--] \(\varepsilon\mapsto\operatorname{mmse}(Q_\varepsilon)\) is \(C^2\).
Here local domination is not enough --- the integral runs over all of
\(\mathbb R\) and carries a factor \(1/p_\varepsilon\) --- so we give a global
Gaussian-tail bound, in the style of Step~2 of
Proposition~\ref{prop:lloydlimit}. Write
\(\operatorname{mmse}(Q_\varepsilon)=\mathbb E[Y_\varepsilon^2]
-\int g_\varepsilon^2/p_\varepsilon\,dz\) with \(Y_\varepsilon=Y+\varepsilon
v(Y)\),
\[
p_\varepsilon(z)=\!\int\!\phi\bigl(z-\theta_\varepsilon(y)\bigr)q_\tau(y)dy,
\quad
g_\varepsilon(z)=\!\int\!\theta_\varepsilon(y)\,
\phi\bigl(z-\theta_\varepsilon(y)\bigr)q_\tau(y)dy,
\quad
\theta_\varepsilon:=\mathrm{id}+\varepsilon v .
\]
The first term is an exact quadratic polynomial in \(\varepsilon\). For the
second, put \(m_\varepsilon=g_\varepsilon/p_\varepsilon\) and recall
\(\partial_\varepsilon(g^2/p)=2m\,\partial_\varepsilon g
-m^2\partial_\varepsilon p\) and
\(\partial^2_\varepsilon(g^2/p)
=2(\partial_\varepsilon g-m\,\partial_\varepsilon p)^2/p
+2m\,\partial^2_\varepsilon g-m^2\partial^2_\varepsilon p\). Every term other
than the first is of the form \(\int\Pi(y,z)\,
\phi(z-\theta_\varepsilon(y))q_\tau(y)\,dy\) times a bounded factor, with
\(\Pi\) polynomial of degree \(\le4\) (each \(\partial_\varepsilon\) brings
down one factor \(v(y)\) or \(v(y)(z-\theta_\varepsilon(y))\), and \(\|v\|_\infty
\le V\)); and the first term carries no surviving \(1/p\), because
Cauchy--Schwarz against the probability measure
\(\phi(z-\theta_\varepsilon(y))q_\tau(y)dy/p_\varepsilon(z)\) gives
\[
\frac{\bigl(\partial_\varepsilon g_\varepsilon
-m_\varepsilon\partial_\varepsilon p_\varepsilon\bigr)^2}{p_\varepsilon}
\ \le\ \int\Pi_1(y,z)^2\,\phi\bigl(z-\theta_\varepsilon(y)\bigr)q_\tau(y)\,dy ,
\]
again of the same form. Finally, for \(|\varepsilon|\le\varepsilon_0\) we have
\(|z-\theta_\varepsilon(y)|\ge|z-y|-\varepsilon_0V\), so each such integral is
bounded by a convolution of a polynomial-weighted Gaussian in \(z-y\) with the
Gaussian \(q_\tau\), giving
\[
\Bigl|\partial_\varepsilon^{\,j}\frac{g_\varepsilon(z)^2}{p_\varepsilon(z)}
\Bigr|
\ \le\ C\bigl(1+|z|^{6}\bigr)e^{-c z^2},
\qquad j=0,1,2,
\]
uniformly in \(|\varepsilon|\le\varepsilon_0\), with
\(c<\tfrac12(\tau^2+1)^{-1}\) and \(C=C(\tau,V,\varepsilon_0)\). This is
integrable, so dominated convergence licenses differentiating twice under the
\(z\)-integral.
\end{itemize}
The polynomial \(v_n\) is not bounded, so we apply this to the truncations
\(v_n^R:=\chi_R\,v_n\), with \(\chi_R\) a smooth cutoff equal to \(1\) on
\([-R,R]\); each \(v_n^R\) is bounded with bounded derivative, so the above
applies verbatim, and the corresponding tangent
\(h_n^R=-\frac{d}{dy}[q_\tau v_n^R]\) satisfies \(h_n^R\to h_n\) in
\(\mathcal H_\tau\) as \(R\to\infty\) (the difference is supported in
\(|y|>R\), where \(q_\tau\) and every polynomial factor decay Gaussian-fast).
Since the curvature forms \(\delta^2 W_2^2\) and \(\delta^2\Phi\) are bounded
quadratic forms on \(\mathcal H_\tau\), the eigenvalue ratios along
\(h_n^R\) converge to those along \(h_n\), so testing \(h_n\) is the limit of
a family of fully admissible tests. No parity restriction survives.

\emph{And the second variation does not depend on which curve realizes it.}
This is what lets us keep the eigenvalues computed along the straight line.
Writing
\(Q_\varepsilon=Q_\tau+\varepsilon h+\tfrac12\varepsilon^2g+o(\varepsilon^2)\)
for such a curve (the expansion above supplies \(h\) and \(g\)), one has
\(\int h=\int g=0\), and
\[
J(Q_\varepsilon)=J(Q_\tau)+\varepsilon\,\delta J[h]
+\tfrac12\varepsilon^2\bigl(\delta^2J[h,h]+\delta J[g]\bigr)+o(\varepsilon^2).
\]
At the stationary point \(Q_\tau\) the cost sensitivity is
\emph{constant on all of \(\mathbb R\)} (the display above:
\(\Lambda\equiv\lambda\), the Gaussian having full support), so
\(\delta J[g]=\int\Lambda\,dg=\lambda\int dg=0\) for any mass-zero \(g\).
Hence the \(\varepsilon^2\) coefficient is \(\tfrac12\delta^2J[h,h]\)
regardless of the curve, and in particular the eigenvalues
\(\widehat\beta_n,\widehat\omega_n\) computed along the linear path are the
correct second variations along the admissible
curves~\eqref{eq:pushforward-curve}. We therefore test \emph{all} \(n\ge1\),
with no parity restriction.

\emph{The eigenvalues and the stability test.} A short computation
diagonalizes both terms in this basis. Using the Gaussian-convolution
scaling identity
\(\int\mathrm{He}_n(u/\tau)q_\tau(u)\phi(z-u)\,du=\beta^{n/2}
\mathrm{He}_n(z/\rho)p_Z(z)\) (with \(\rho^2=\tau^2+1\),
\(\beta=\tau^2/\rho^2\); proved by matching Hermite generating functions)
and the recurrence
\(x\,\mathrm{He}_n=\mathrm{He}_{n+1}+n\,\mathrm{He}_{n-1}\), the
\(\mathrm{He}_{n+1}\) terms in \(g_{h_n}-m\,p_{h_n}\) cancel and one is left
with \(g_{h_n}-m\,p_{h_n}=n\,\beta^{(n-1)/2}(\tau/\rho^2)
\mathrm{He}_{n-1}(z/\rho)\,p_Z\). Hence the estimation curvature is diagonal
with eigenvalues
\[
\widehat\beta_n:=\frac{\delta^2\Phi[h_n,h_n]}{\|h_n\|^2}
=2n\,\frac{\tau^{2n}}{(\tau^2+1)^{\,n+1}}\ >0
\qquad(\text{the ``push'' in mode }n).
\]
The transport curvature is diagonal in the same basis and, contrary to what
one might fear, is also elementary. Perturbing the quantile function of
\(Q_t=Q_\tau+t\,h\) through \(F_{Q_t}(G_t(u))=u\) gives
\(G_1=-H/q_\tau\) and \(G_2=-[\tfrac12q_\tau'G_1^2+hG_1]/q_\tau\), where
\(H(y)=\int_{-\infty}^yh\); substituting into
\(W_2^2=\int_0^1(F_P^{-1}-G_t)^2\) and changing variables \(u=F_{Q_\tau}(y)\),
\[
\begin{aligned}
\delta^2W_2^2[h,h]
&=2\!\int\!\frac{H^2}{q_\tau}
+2\!\int\!(T_{Q_\tau}(y)-y)\frac{q_\tau'}{q_\tau}\frac{H^2}{q_\tau}\\
&\quad-4\!\int\!(T_{Q_\tau}(y)-y)\frac{hH}{q_\tau}.
\end{aligned}
\]
For \(h=h_n\) the antiderivative is again a Hermite function: since
\(\frac{d}{dy}\bigl[\mathrm{He}_{n-1}(y/\tau)q_\tau(y)\bigr]
=-\tau^{-1}\mathrm{He}_n(y/\tau)q_\tau(y)\) (by
\(\mathrm{He}_{n-1}'=(n-1)\mathrm{He}_{n-2}\), \(q_\tau'=-(y/\tau^2)q_\tau\)
and the recurrence \(x\mathrm{He}_{n-1}=\mathrm{He}_n+(n-1)\mathrm{He}_{n-2}\)),
\[
H_n(y)=-\tau\,\mathrm{He}_{n-1}(y/\tau)\,q_\tau(y),
\qquad\text{so}\qquad
\begin{aligned}
\frac{H_n^2}{q_\tau}&=\tau^2\,\mathrm{He}_{n-1}(y/\tau)^2q_\tau(y),\\
\frac{h_nH_n}{q_\tau}&=-\tau\,\mathrm{He}_n\mathrm{He}_{n-1}\,q_\tau .
\end{aligned}
\]
With \(X\sim\mathcal N(0,1)\) and the orthogonality relations
\(\mathbb E[\mathrm{He}_m^2]=m!\),
\(\mathbb E[X^2\mathrm{He}_m^2]=m!\,(2m+1)\) and
\(\mathbb E[X\mathrm{He}_n\mathrm{He}_{n-1}]=n!\) (the last two follow from
\(x\mathrm{He}_m=\mathrm{He}_{m+1}+m\mathrm{He}_{m-1}\)), the three integrals
are
\[
\int\frac{H_n^2}{q_\tau}=\tau^2(n-1)!,
\qquad
\int y^2\frac{H_n^2}{q_\tau}=\tau^4(n-1)!\,(2n-1),
\qquad
\int y\,\frac{h_nH_n}{q_\tau}=-\tau^2\,n! .
\]
Substituting \(T_{Q_\tau}(y)-y=cy\) with \(c:=\sigma/\tau-1\) and
\(q_\tau'/q_\tau=-y/\tau^2\), they combine to
\[
\begin{aligned}
&2\tau^2(n-1)!\;-\;\frac{2c}{\tau^2}\,\tau^4(n-1)!(2n-1)\;+\;4c\,\tau^2n!\\
&\qquad=2\tau^2(n-1)!\bigl[1-c(2n-1)+2cn\bigr]
=2\tau^2(n-1)!\,(1+c),
\end{aligned}
\]
using \(n!=n\,(n-1)!\); and \(1+c=\sigma/\tau\), so
\(\delta^2W_2^2[h_n,h_n]=2\sigma\tau\,(n-1)!\). Dividing by
\(\|h_n\|^2=n!\), the transport eigenvalues are, in closed form,
\begin{equation}
\widehat\omega_n:=\frac{\delta^2W_2^2[h_n,h_n]}{\|h_n\|^2}
=\frac{2\sigma\tau}{n}\ >0
\qquad(\text{the ``pull'' in mode }n).
\label{eq:omega-hat}
\end{equation}
The \(1/n\) is the familiar weakness of \(W_2\) against high frequencies: a
fine ripple moves each particle only a short distance, so transport resists
it only feebly. The second variation is nonnegative \emph{in every Hermite
mode} exactly when the pull beats the push modewise:
\begin{equation}
k^2\,\widehat\omega_n\ \ge\ \widehat\beta_n
\quad\Longleftrightarrow\quad
k^2\ \ge\ \frac{\widehat\beta_n}{\widehat\omega_n}
=\frac{n^2\,\tau^{2n-1}}{\sigma\,(\tau^2+1)^{n+1}},
\qquad\text{for every }n .
\label{eq:mode-test}
\end{equation}
Modewise nonnegativity is not yet nonnegativity in every direction; the next
lemma supplies the two facts --- vanishing mixed terms and completeness ---
that promote it, and only afterwards do we speak of ``every direction''.

\begin{lemma}[Hermite diagonalization and completeness]
\label{lem:hermite-diag}
Let \(\mathcal H_\tau:=\{h:\int h=0,\ \|h\|^2:=\int h^2/q_\tau<\infty\}\).
Then:
\begin{enumerate}
\item[\textup{(i)}] \(\{h_n\}_{n\ge1}\) is a complete orthogonal system in
\(\mathcal H_\tau\), with \(\|h_n\|^2=n!\);
\item[\textup{(ii)}] both curvature forms are diagonal in it:
\(\delta^2\Phi[h_m,h_n]=\delta^2W_2^2[h_m,h_n]=0\) for \(m\ne n\).
\end{enumerate}
Consequently, for \(h=\sum_{n\ge1}c_nh_n\in\mathcal H_\tau\),
\begin{equation}
\delta^2J[h,h]=\sum_{n\ge1}c_n^2\,\bigl(k^2\widehat\omega_n
-\widehat\beta_n\bigr)\,n! ,
\label{eq:diag-sum}
\end{equation}
so \eqref{eq:mode-test} holds for every \(n\) \emph{if and only if}
\(\delta^2J[h,h]\ge0\) for every \(h\in\mathcal H_\tau\).
\end{lemma}

\begin{proof}
\emph{(i)} The map \(h\mapsto u:=h/q_\tau\) is an isometry from
\(\mathcal H_\tau\) onto \(\{u\in L^2(q_\tau):\int u\,q_\tau=0\}\), and sends
\(h_n\) to \(\mathrm{He}_n(\cdot/\tau)\). The Hermite polynomials
\(\{\mathrm{He}_n(\cdot/\tau)\}_{n\ge0}\) form a complete orthogonal system in
\(L^2(q_\tau)\) with \(\int\mathrm{He}_m\mathrm{He}_n\,q_\tau=n!\,\delta_{mn}\)~\cite{szego1939orthogonal}; the constraint \(\int u\,q_\tau=0\) removes exactly
the constant \(n=0\), leaving \(\{h_n\}_{n\ge1}\) complete in
\(\mathcal H_\tau\).

\emph{(ii)} For the estimation form, the computation above gives
\[
g_{h_n}-m_\tau\,p_{h_n}=n\beta^{(n-1)/2}(\tau/\rho^2)
\mathrm{He}_{n-1}(z/\rho)\,p_Z,
\]
where \(m_\tau(z)=\beta z\) is the conditional mean (subscripted here to keep
it apart from the mode index \(m\)), so
\[
\delta^2\Phi[h_m,h_n]
=2\!\int\!\frac{(g_{h_m}-m_\tau p_{h_m})(g_{h_n}-m_\tau p_{h_n})}{p_Z}
=2mn\,\beta^{\frac{m+n}{2}-1}\frac{\tau^2}{\rho^4}
\!\int\!\mathrm{He}_{m-1}\!\Bigl(\tfrac z\rho\Bigr)
\mathrm{He}_{n-1}\!\Bigl(\tfrac z\rho\Bigr)p_Z(z)\,dz,
\]
and \(p_Z\) is the \(\mathcal N(0,\rho^2)\) density, so the integral is
\((m-1)!\,\delta_{mn}\).

For the transport form, polarize the three integrals of the displayed formula
for \(\delta^2W_2^2\) and substitute \(H_n=-\tau\mathrm{He}_{n-1}(\cdot/\tau)
q_\tau\), \(T_{Q_\tau}(y)-y=cy\), \(q_\tau'/q_\tau=-y/\tau^2\). Writing
\(x=y/\tau\) and using \(x\mathrm{He}_j=\mathrm{He}_{j+1}+j\mathrm{He}_{j-1}\)
twice,
\[
\begin{aligned}
\int\frac{H_mH_n}{q_\tau}&=\tau^2(m-1)!\,\delta_{mn},\\
\int y^2\frac{H_mH_n}{q_\tau}
&=\tau^4(n-1)!\bigl[\delta_{n,m+2}+(2m-1)\delta_{nm}
+(m-1)(m-2)\delta_{n,m-2}\bigr],
\end{aligned}
\]
\[
\int y\,\frac{h_mH_n}{q_\tau}=-\tau^2(n-1)!\bigl[\delta_{n,m+2}
+m\,\delta_{nm}\bigr].
\]
Only \(|m-n|\in\{0,2\}\) can contribute. For \(n=m\) the three combine to
\(2\tau^2(m-1)!(1+c)=2\sigma\tau(m-1)!\), which is~\eqref{eq:omega-hat}. For
\(n=m+2\) the first integral vanishes and the other two contribute
\(-2c\tau^2(m+1)!\) and \(+2c\tau^2(m+1)!\), which cancel; the case
\(n=m-2\) is symmetric. Hence the form is diagonal.

Given (i) and (ii), \eqref{eq:diag-sum} is Parseval's identity for the
bounded quadratic form \(\delta^2J\) on \(\mathcal H_\tau\), and the stated
equivalence is immediate, the coefficients \(c_n\) being arbitrary.
\end{proof}

\emph{Consistency: the variance mode.} The mode \(n=2\) is the
\emph{in-family} test, since \(\partial_\tau q_\tau=\tau^{-1}h_2\). Its
inequality~\eqref{eq:mode-test} reads \(k^2\sigma\ge4\tau^3/(1+\tau^2)^3\),
and eliminating \(\sigma\) with the stationarity
condition~\eqref{eq:FOC-s-alt} turns it into
\[
k^2\ \ge\ \frac{3\tau^2-1}{(1+\tau^2)^3},
\]
which is exactly \(J''(\tau)\ge0\) for the scalar function
\(J(\tau)=k^2(\sigma-\tau)^2+\tau^2/(1+\tau^2)\) of
Proposition~\ref{prop:gauss-foc}. This is a sharp check
on~\eqref{eq:omega-hat}. Note the threshold \(k^2\ge\tfrac14\) of
Proposition~\ref{prop:gauss-foc} is the \emph{worst case over \(\tau\)} of
this pointwise test (the right-hand side is maximized at \(\tau=1\), where it
equals \(\tfrac14\)), i.e.\ the condition for \(J\) to be convex on all of
\((0,\infty)\); at a given \(\tau^\star\) the in-family test is weaker.

\emph{Which mode binds.} It is tempting to expect the quartic \(n=4\) --- the
first \emph{bimodal} mode --- to be the first to fail as \(k\) decreases,
splitting the Gaussian into two humps in a pitchfork bifurcation. The closed
form~\eqref{eq:mode-test} shows otherwise. Writing
\(\beta:=\tau^2/(\tau^2+1)\), the right-hand side of~\eqref{eq:mode-test} is
\(n^2\beta^{\,n+1}/(\sigma\tau^3)\). Treating \(n\) as a continuous variable,
this is maximized at
\[
n_c=\frac{2}{\log(1/\beta)}=\frac{2}{\log(1+\tau^{-2})}\approx2\tau^2 ,
\]
and the binding Hermite mode --- the index is of course an integer --- is the
maximizing integer nearest \(n_c\), determined numerically
in~\eqref{eq:kstab-value}. Either way
the binding perturbation is not bimodal at all but a \emph{fine ripple},
of local wavenumber \(\sqrt{n_c}/\tau\to\sqrt2\), i.e.\ of wavelength
\(2\pi/\sqrt2=\pi\sqrt2\approx4.44\) in \(y\). The mechanism is transparent
in Fourier terms: a ripple of wavenumber \(\omega\) and amplitude \(A\) in
\(Q\) survives the channel damped to \(Ae^{-\omega^2/2}\) in \(p_Z=Q*\phi\),
so it raises the Fisher information \(\mathcal I(p_Z)\) by
\(\sim\omega^2A^2e^{-\omega^2}\), while displacing mass only \(\sim A/\omega\)
and so costing only \(\sim A^2/\omega^2\) in transport; the benefit-to-cost
ratio \(\omega^4e^{-\omega^2}\) is maximized at \(\omega=\sqrt2\), as found.
The bimodal mode \(n=4\) is far from binding: at \(\sigma=5\) its own
threshold is \(k\approx0.15\), against \(k\approx0.73\) at the maximizing
mode.

We can now state the result of this subsection.

\begin{proposition}[Second-variation stability of the high-variance Gaussian
branch]
\label{prop:kstab}
Let \(P=\mathcal N(0,\sigma^2)\) and let \(Q_{\tau}\) be the Gaussian critical
point of \eqref{eq:FOC-s-alt} on the high-variance (near-identity) branch.
Define
\[
k_{\mathrm{stab}}(\sigma):=\inf\Bigl\{k>0:\ k^2\ \ge\
\frac{n^2\tau^{2n-1}}{\sigma(\tau^2+1)^{n+1}}\ \text{ for every }n\ge1,
\ \tau=\tau(k)\Bigr\} .
\]
Then, for \(k\ge k_{\mathrm{stab}}\), \(\delta^2J[h,h]\ge0\) for every
\(h\in\mathcal H_\tau\); and for \(k<k_{\mathrm{stab}}\) there is an
admissible perturbation, realized by the push-forward
curve~\eqref{eq:pushforward-curve}, along which \(J\) strictly decreases, so
that branch is not a local minimizer. Numerically, at \(\sigma=5\),
\begin{equation}
k_{\mathrm{stab}}=0.734779\qquad(\tau^\star=4.98619,\ \text{binding integer
mode }n=51),
\label{eq:kstab-value}
\end{equation}
and, since \(\beta^{\,n_c+1}\to e^{-2}\) with \(n_c\approx2\tau^2\), the
large-\(\sigma\) limit is explicit,
\[
k_{\mathrm{stab}}^2\ \longrightarrow\ \frac{4}{e^2},
\qquad\text{i.e.}\qquad k_{\mathrm{stab}}\ \longrightarrow\ \frac2e\approx0.7358 .
\]
\end{proposition}

\begin{proof}
The eigenvalue computation above gives \(\widehat\omega_n=2\sigma\tau/n\) and
\(\widehat\beta_n=2n\tau^{2n}/(\tau^2+1)^{n+1}\), so \eqref{eq:mode-test} is
the displayed inequality. Lemma~\ref{lem:hermite-diag} converts ``for every
\(n\)'' into ``for every \(h\in\mathcal H_\tau\)'' via~\eqref{eq:diag-sum}.
For \(k<k_{\mathrm{stab}}\) some mode has
\(k^2\widehat\omega_n<\widehat\beta_n\); the
curve~\eqref{eq:pushforward-curve} (truncated as described there) realizes
that direction admissibly and has \(\varepsilon^2\)-coefficient
\(\tfrac12\delta^2J[h_n,h_n]<0\) at the stationary point, so \(J\) strictly
decreases along it. The numerical values follow by solving
\eqref{eq:mode-test} at the maximizing integer mode simultaneously
with~\eqref{eq:FOC-s-alt}.
\end{proof}

The binding mode at \(\sigma=5\) is odd --- which is why the admissibility
discussion above matters. (Had we restricted to even modes the value would be
\(0.734715\) at \(n=50\), a difference of \(6\times10^{-5}\); so the
conclusions are insensitive to the point, but the statement is now about the
full mode family rather than a symmetric subfamily.) We quote
\(k_{\mathrm{stab}}\approx0.735\)
throughout.
Thus for \(k<k_{\mathrm{stab}}\) the \emph{high-variance} Gaussian critical
point is a saddle: an admissible rippled,
non-Gaussian perturbation of it strictly decreases \(J\). For
\(k\ge k_{\mathrm{stab}}\) no admissible Hermite direction at that point has
negative curvature.

\begin{remark}[What the threshold does and does not establish]
\label{rem:hermite}
The conclusion for
\(k<k_{\mathrm{stab}}\) is a genuine negative result about that branch: a
direction of strict
descent is exhibited, so the near-identity Gaussian is not a local minimizer.
The conclusion for
\(k\ge k_{\mathrm{stab}}\) is weaker than it may look, and the statement as a
whole says less about \emph{affine optimality} than the name might suggest.
We state \(k_{\mathrm{stab}}\) as a \emph{second-variation
stability threshold of one branch} rather than a linear-optimality threshold
for the following four reasons.

\emph{(i) One branch, not all Gaussians.} When \eqref{eq:FOC-s-alt} has
several roots the test must be applied to each separately, and the answers
differ. At \(\sigma=5\) the low-variance root is
\emph{stable} well below \(k_{\mathrm{stab}}\): for instance at \(k=0.1\) the
unique root is \(\tau=0.04975\) with
\(\max_{n\ge1}\widehat\beta_n/\widehat\omega_n
=9.90\times10^{-3}<k^2=10^{-2}\), and at \(k=0.15\) and \(k=0.2\) the low
roots \(\tau=0.1128,0.2087\) are likewise stable, while the high roots are
not. (At \(k=0.2\) the two Gaussian local minima even tie in value,
\(J_{\mathrm G}=0.960\) at both \(\tau=0.2087\) and \(\tau=4.7913\).) So
\(k<k_{\mathrm{stab}}\) does \emph{not} say that every affine policy is
second-variation unstable; it says the near-identity one is. That affine
control is genuinely suboptimal at such \(k\) is established instead by
explicit non-Gaussian competitors --- at \(k=0.1\) the computed
\(\widehat J=0.053\) against \(J_{\mathrm G}=0.248\) --- not by this
calculation.

(The margin at \(k=0.1\) is thin, \(9.90\times10^{-3}\) against
\(10^{-2}\), but it does not close. On any Gaussian branch the binding mode
at small \(\tau\) is \(n=1\), and the stationarity
condition~\eqref{eq:FOC-s-alt} gives exactly
\(\widehat\beta_1/\widehat\omega_1=\tau/(\sigma(1+\tau^2)^2)
=k^2(1-\tau/\sigma)<k^2\), while the higher modes are smaller by the factor
\(n^2\beta^{\,n-1}\). The absolute margin is therefore \(k^2\tau/\sigma>0\);
it is small only because \(\tau\) is small, and the low branch remains
second-variation stable as \(k\) decreases.)

\emph{(ii) Modewise nonnegativity is not yet nonnegativity in every direction,
and nonnegativity is not yet a local minimum.} Two gaps sit here, and we flag
both. \emph{First}, the passage from the family of scalar
inequalities~\eqref{eq:mode-test} to a statement about \emph{every}
admissible \(h\) is exactly Lemma~\ref{lem:hermite-diag}, which is why we
proved it rather than assumed it. \emph{Second}, even granted that,
concluding local minimality from \(\delta^2J\ge0\) requires
a uniform third-order remainder estimate and a
coercivity bound \(\delta^2J[h,h]\ge c\|h\|^2\) with \(c>0\) --- neither of
which
we supply. Coercivity in fact \emph{fails} in this norm: by
\eqref{eq:omega-hat} and the eigenvalues \(\widehat\beta_n\),
\[
\frac{\delta^2J[h_n,h_n]}{\|h_n\|^2}
=k^2\widehat\omega_n-\widehat\beta_n
=\frac{2\sigma\tau k^2}{n}-\frac{2n\,\tau^{2n}}{(\tau^2+1)^{n+1}}
\ \xrightarrow[n\to\infty]{}\ 0^+ ,
\]
so for \(k>k_{\mathrm{stab}}\) the quadratic form is strictly positive in
every mode but degenerates like \(1/n\) along fine ripples. A local-minimum
claim would have to control the third variation uniformly over that
degenerating family. We do not, and so we claim only stability of the second
variation.

\emph{(iii) Local, not global.} Even a proved local minimum would leave open
whether some \(Q\) far from Gaussian beats the affine policy; the
\emph{global} optimizer may depart from the Gaussian at a larger \(k\) still,
so \(k_{\mathrm{stab}}\) is at best a lower bound on any global
linear-optimality threshold, and only for the branch tested. The global
question remains open.

\emph{(iv) Numerically invisible.} The destabilizing
direction is a small-amplitude high-frequency ripple, and near
\(k_{\mathrm{stab}}\) the
gain over the affine policy is only \(O(\varepsilon^2)\); a solver that
searches over few-level signalling constellations will not find it and will
report the affine policy as optimal well below \(k_{\mathrm{stab}}\)
(cf.\ \S\ref{sec:numerics} and Remark~\ref{rem:kc-numerics}).
\end{remark}

\section{A Finite-Dimensional Program}
\label{sec:finite-program-sec}

We now reduce the variational problem~\eqref{eq:OT-variational} to a
finite-dimensional program that can be solved numerically, by restricting
to finitely supported laws. We emphasize that nothing in this section or
the next uses the MMSE--Fisher identity of \S\ref{sec:fisher}: the
estimation term enters only through \(\operatorname{mmse}\) and its first
variation (Lemma~\ref{lem:mmse-derivative}), so the computational theory
stands on the reformulation of \S\ref{sec:ot-reform}
(Theorem~\ref{thm:ot-reformulation}) alone. The Fisher form
is an interpretive lens, not a prerequisite.

\subsection{The finite-level program}

We first show that restricting to finitely supported laws loses nothing in
the limit.

\begin{proposition}[Finite-level approximation, and attainment]
\label{prop:finite-level}
For \(n\ge1\) set
\(J_n^\star:=\inf\{J(Q):Q\text{ supported on at most }n\text{ points}\}\).
Then:
\begin{enumerate}
\item[\textup{(i)}] the infimum is \emph{attained}, by some
\(Q_n^\star\) supported on at most \(n\) points; and
\item[\textup{(ii)}] \(J_n^\star\) is nonincreasing in \(n\) and
\(J_n^\star\downarrow J^\star\) as \(n\to\infty\).
\end{enumerate}
Thus~\eqref{eq:OT-variational} is the limit
of the finite-dimensional programs over atom locations and weights, even
though every minimizer \(Q^\star\) is itself atomless
(Proposition~\ref{prop:ac}). Part~(i) is what licenses speaking of exact
finite-level minimizers \(Q_n^\star\) --- as \S\ref{sec:prelim} and
Remark~\ref{rem:atomic-vs-ac} do --- rather than only of the values
\(J_n^\star\).
\end{proposition}

\begin{proof}
For each \(n\ge 1\) let
\(F_n := \{Q : |\operatorname{supp} Q| \le n\}\) denote the feasible set
defining \(J_n^\star := \inf_{Q\in F_n} J(Q)\).

\emph{(i) Attainment.} Run the direct method of
Proposition~\ref{prop:existence} inside \(F_n\). Take a minimizing sequence
\((\nu_j)\subset F_n\) with \(J(\nu_j)\le J_n^\star+1\); the coercivity of
Step~2 there (which used only \(k>0\) and a bound on \(J\)) gives uniformly
bounded second moments, hence tightness, so
\(\nu_{j_l}\rightharpoonup\nu\in\mathcal P_2(\mathbb R)\)
along a subsequence (the limit has finite second moment by the truncation
argument of Step~3 there).

It remains to check \(\nu\in F_n\), i.e.\ that \(F_n\) is closed under weak
convergence. It is tempting to argue coordinatewise --- extract limits of the
atom locations and weights --- but that is not available: bounded second
moments do \emph{not} confine the atoms to a compact set, as
\(\nu_j=(1-j^{-2})\delta_0+j^{-2}\delta_j\) shows, whose second moment is
\(1\) for every \(j\) while one atom escapes to infinity. The correct argument
uses no location control at all. Suppose \(|\operatorname{supp}\nu|\ge n+1\),
and pick \(n+1\) distinct points of \(\operatorname{supp}\nu\) together with
pairwise disjoint open neighbourhoods \(U_0,\dots,U_n\) of them; by definition
of the support, \(\nu(U_i)>0\) for each \(i\). The portmanteau theorem gives,
for open sets,
\[
\liminf_{l}\nu_{j_l}(U_i)\ \ge\ \nu(U_i)\ >\ 0,
\qquad i=0,\dots,n ,
\]
so for all large \(l\) every one of the \(n+1\) disjoint sets \(U_i\) carries
positive \(\nu_{j_l}\)-mass; each therefore contains a support point of
\(\nu_{j_l}\), giving \(|\operatorname{supp}\nu_{j_l}|\ge n+1\) and
contradicting \(\nu_{j_l}\in F_n\). Hence \(|\operatorname{supp}\nu|\le n\)
and \(\nu\in F_n\). Lower semicontinuity of the
transport term (Step~4) and weak continuity of \(\operatorname{mmse}\)
(Fact~\ref{fact:mmse}) give \(J(\nu)\le\liminf_lJ(\nu_{j_l})=J_n^\star\), so
\(Q_n^\star:=\nu\) attains the infimum.

\emph{(ii) Monotonicity.} Two elementary facts give
the claim.

\emph{(ii.a) The feasible sets are nested: \(F_n\subseteq F_{n+1}\).}
If \(|\operatorname{supp} Q|\le n\) then, since \(n\le n+1\), also
\(|\operatorname{supp} Q|\le n+1\); the same measure \(Q\), with its atoms
and weights unchanged, satisfies the looser constraint (no atom is moved or
added).

\emph{(ii.b) An infimum over a larger set is no larger.}
For any real-valued \(f\) and sets \(A\subseteq B\), \(\inf_B f\le \inf_A f\),
since \(\{f(x):x\in A\}\subseteq\{f(x):x\in B\}\). With \(A=F_n\),
\(B=F_{n+1}\), \(f=J\),
\[
  J_{n+1}^\star=\inf_{Q\in F_{n+1}}J(Q)\le\inf_{Q\in F_n}J(Q)=J_n^\star .
\]
Hence \((J_n^\star)_{n\ge1}\) is nonincreasing.

\emph{Lower bound.} Every \(Q\in F_n\) is finitely supported with finite
second moment, so \(Q\in\mathcal P_2(\mathbb R)\); therefore
\(J_n^\star\ge\inf_{Q\in\mathcal P_2(\mathbb R)}J(Q)=:J^\star\) for every
\(n\), and the sequence is bounded below by \(J^\star\).

\emph{Convergence.} A nonincreasing sequence of reals bounded below
converges to its infimum (greatest-lower-bound property of \(\mathbb R\));
hence \(J_n^\star\downarrow\inf_n J_n^\star\ge J^\star\).

\emph{Upper bound \(\inf_n J_n^\star\le J^\star\).} Fix any
\(Q\in\mathcal P_2(\mathbb R)\). Finitely supported measures are
\(W_2\)-dense in \(\mathcal P_2(\mathbb R)\); concretely, the \(n\)-point
quantile discretization \(\nu_n:=\tfrac1n\sum_{i=1}^n\delta_{F_Q^{-1}((i-1/2)/n)}\)
satisfies \(\nu_n\to Q\) in \(W_2\). Then
\(W_2^2(P,\nu_n)\to W_2^2(P,Q)\) (continuity of \(W_2\) under
\(W_2\)-convergence) and, since \(W_2\)-convergence implies weak
convergence with bounded second moments, Fact~\ref{fact:mmse} gives
\(\operatorname{mmse}(\nu_n)\to\operatorname{mmse}(Q)\); hence
\(J(\nu_n)\to J(Q)\). As \(\nu_n\) has at most \(n\) atoms,
\(J_n^\star\le J(\nu_n)\), so \(\inf_n J_n^\star\le\lim_n J(\nu_n)=J(Q)\).
Minimizing over \(Q\) gives \(\inf_n J_n^\star\le J^\star\).

Combining the two bounds, \(J_n^\star\downarrow J^\star\).
\end{proof}

Two families of quantities specify an \(n\)-level controller: the \(n\)
output \emph{levels} \(y_1<\dots<y_n\) (the values that \(Y\) may take),
and the \emph{partition} of the state space into \(n\) \emph{decision
regions} (which level each realization of \(X_0\) is mapped to). In one
dimension the \(W_2\)-optimal transport is monotone
(Fact~\ref{fact:1d-ot}), so each decision region is an interval
\(I_i=(\zeta_{i-1},\zeta_i]\); the partition is therefore fixed by the
\(n-1\) \emph{thresholds} \(\zeta_1<\dots<\zeta_{n-1}\), or equivalently by
the atom weights \(w_i=P(I_i)\) through \(\zeta_i=F_P^{-1}(W_i)\),
\(W_i=\sum_{j\le i}w_j\). \textbf{The finite-level program optimizes over
both families, the \(n\) levels and the \(n-1\) thresholds
(equivalently, weights), \(2n-1\) parameters in all.} We now make the
objective explicit for Gaussian \(P\). Order the atoms \(y_1<\dots<y_n\),
take weights \(w_i\ge0\) with \(\sum_i w_i=1\), and form
\(W_i=\sum_{j\le i}w_j\) (\(W_0=0\)) and the quantile thresholds
\[
\zeta_i:=F_P^{-1}(W_i)=\sigma\,\Phi^{-1}(W_i),
\qquad \zeta_0=-\infty,\ \zeta_n=+\infty,
\]
with \(\Phi\) the standard normal CDF and \(\Phi_\sigma,\phi_\sigma\) the
CDF and density of \(P=\mathcal{N}(0,\sigma^2)\). On the bin
\(I_i:=(\zeta_{i-1},\zeta_i]\) the mass and (unnormalized) mean of \(P\)
are, in closed form,
\[
w_i=\Phi_\sigma(\zeta_i)-\Phi_\sigma(\zeta_{i-1}),
\qquad
\mu_i:=\int_{I_i}x\,dP=\sigma^2\bigl(\phi_\sigma(\zeta_{i-1})-\phi_\sigma(\zeta_i)\bigr),
\qquad
\bar x_i:=\frac{\mu_i}{w_i}.
\]

\begin{proposition}[The finite-level program]
\label{prop:finite-program}
For \(P=\mathcal{N}(0,\sigma^2)\) and \(Q_n=\sum_{i=1}^n w_i\delta_{y_i}\),
\begin{equation}
J(Q_n)=
k^2\underbrace{\Bigl[\bigl(\sigma^2-\textstyle\sum_i w_i\bar x_i^2\bigr)
+\sum_i w_i\,(y_i-\bar x_i)^2\Bigr]}_{W_2^2(P,Q_n)}
+\underbrace{\Bigl[\sum_i w_i y_i^2-\int_{\mathbb{R}}\frac{g(z)^2}{p_Z(z)}\,dz\Bigr]}_{\operatorname{mmse}(Q_n)},
\label{eq:finite-program}
\end{equation}
where \(p_Z(z)=\sum_i w_i\phi(z-y_i)\) and
\(g(z)=\sum_i w_i y_i\phi(z-y_i)\). Hence
\begin{equation}
J_n^\star=\min_{1\le m\le n}\ \min\Bigl\{J(Q_m):
y_1<\dots<y_m,\ \zeta_1<\dots<\zeta_{m-1}\Bigr\},
\label{eq:finite-min}
\end{equation}
the inner minimization being over the \(m\) levels and the \(m-1\)
\emph{interior} thresholds, equivalently over strictly positive weights
\(w\in\operatorname{relint}\Delta_{m-1}\): for each \(m\) a \emph{smooth,
nonconvex}
program in \(2m-1\) variables, whose transport part is closed-form and
whose estimation part is a single one-dimensional Gaussian-mixture
integral (smooth, with an explicit gradient evaluated by quadrature). The
outer minimization over \(m\le n\) is what handles laws with fewer than \(n\)
effective atoms; it is needed because a vanishing weight is not an interior
point of the parameterization --- \(\bar x_i=\mu_i/w_i\) is undefined at
\(w_i=0\) --- so the boundary \(w_i=0\) is treated by dropping the atom and
solving the smaller program rather than by evaluating \eqref{eq:finite-program}
there. By Proposition~\ref{prop:finite-level}(i) the outer minimum is attained.
The levels \(y_i\) are the
controller's outputs and the bins \(I_i\) its decision regions, giving the
\(n\)-level first controller
\[
f_n(x)=\sum_{i=1}^n y_i\,\mathbf 1_{I_i}(x).
\]
\end{proposition}

\begin{proof}
\emph{Transport term.} Since \(P\) is absolutely continuous and the atoms
are ordered, Fact~\ref{fact:1d-ot} makes the monotone coupling optimal: it
sends the entire bin \(I_i\) (of \(P\)-mass \(w_i\)) to \(y_i\). Expanding
the square,
\[
W_2^2(P,Q_n)=\sum_i\int_{I_i}(x-y_i)^2\,dP
=\sum_i\Bigl(\underbrace{\textstyle\int_{I_i}x^2\,dP}_{m_2^i}
-2y_i\underbrace{\textstyle\int_{I_i}x\,dP}_{\mu_i}
+y_i^2\underbrace{\textstyle\int_{I_i}dP}_{w_i}\Bigr).
\]
Since the bins partition \(\mathbb{R}\), \(\sum_i m_2^i=\int x^2\,dP
=\sigma^2\). Writing \(\mu_i=w_i\bar x_i\) and completing the square in
each summand,
\[
m_2^i-2y_i\mu_i+y_i^2 w_i
=m_2^i-w_i\bar x_i^2+w_i\bigl(y_i^2-2y_i\bar x_i+\bar x_i^2\bigr)
=m_2^i-w_i\bar x_i^2+w_i(y_i-\bar x_i)^2 ;
\]
summing over \(i\) and using \(\sum_i m_2^i=\sigma^2\) gives
\(W_2^2(P,Q_n)=\sigma^2-\sum_i w_i\bar x_i^2+\sum_i w_i(y_i-\bar x_i)^2\).
Finally \(\mu_i=\int_{\zeta_{i-1}}^{\zeta_i}x\,\phi_\sigma(x)\,dx
=-\sigma^2\!\int_{\zeta_{i-1}}^{\zeta_i}\phi_\sigma'(x)\,dx
=\sigma^2\bigl(\phi_\sigma(\zeta_{i-1})-\phi_\sigma(\zeta_i)\bigr)\),
using \(x\,\phi_\sigma(x)=-\sigma^2\phi_\sigma'(x)\).

\emph{Estimation term.} For \(Y\sim Q_n\), the observation \(Z=Y+N\) has
the Gaussian-mixture density \(p_Z(z)=\sum_i w_i\phi(z-y_i)\). By Bayes'
rule the posterior weights are
\(\Pr(Y=y_i\mid Z=z)=w_i\phi(z-y_i)/p_Z(z)\), so the posterior moments are
\[
\mathbb{E}[Y\mid Z=z]=\frac{\sum_i w_i y_i\phi(z-y_i)}{p_Z(z)}=\frac{g(z)}{p_Z(z)},
\qquad
\mathbb{E}[Y^2\mid Z=z]=\frac{\sum_i w_i y_i^2\phi(z-y_i)}{p_Z(z)}.
\]
Using \(\operatorname{Var}(Y\mid Z)=\mathbb E[Y^2\mid Z]-\mathbb E[Y\mid Z]^2\)
and integrating against \(p_Z\),
\[
\operatorname{mmse}(Q_n)
=\int\mathbb E[Y^2\mid Z=z]\,p_Z(z)\,dz-\int\Bigl(\tfrac{g(z)}{p_Z(z)}\Bigr)^2 p_Z(z)\,dz .
\]
The first integral is \(\int\sum_i w_i y_i^2\phi(z-y_i)\,dz=\sum_i w_i y_i^2\)
(since \(\int\phi=1\); equivalently \(\mathbb E[\mathbb E[Y^2\mid Z]]=\mathbb E[Y^2]\)),
and the second is \(\int g^2/p_Z\). Hence
\(\operatorname{mmse}(Q_n)=\sum_i w_i y_i^2-\int g^2/p_Z\). Adding the two
terms gives~\eqref{eq:finite-program}. The program has the \(2n-1\) free
parameters \((y,w)\) noted above; on the interior \(\{0<W_1<\dots<W_{n-1}<1\}\)
the thresholds \(\zeta_i\) and moments \(w_i,\mu_i\) are smooth in \(w\),
and the mixture integral \(\int g^2/p_Z\) is smooth in \((y,w)\) because
\(p_Z>0\); and \(f_n\) is exactly the monotone transport map realizing
\(Q_n\).
\end{proof}

\begin{remark}[The second controller as a soft nearest-level rule]
\label{rem:softmax}
Cancelling the common factor \((2\pi)^{-1/2}e^{-z^2/2}\) between numerator
and denominator of \(m=g/p_Z\) puts the second controller in \emph{softmax}
(Boltzmann) form,
\begin{equation}
m(z)=\sum_j p_j(z)\,y_j,
\qquad
p_j(z)=\frac{w_j\,e^{\,z y_j-\frac12 y_j^2}}{\sum_k w_k\,e^{\,z y_k-\frac12 y_k^2}}
=\Pr(Y=y_j\mid Z=z),
\label{eq:softmax}
\end{equation}
a convex combination of the levels by the posterior \emph{responsibilities}
of the Gaussian mixture, with logits \(\ell_j(z)=z\,y_j-\tfrac12 y_j^2+\log w_j\)
\emph{linear} in \(z\). Thus \(m\) is a smooth, monotone \emph{soft
nearest-level} interpolation: nearly flat at \(y_j\) on the region of
\emph{observation} space where the \(j\)th mixture component dominates (one
responsibility \(\approx1\)), with smooth transitions where two
responsibilities trade off. (Those regions live in \(z\)-space and should not
be identified with the decision cells \(I_i\subset\mathbb R\) of the first
controller, which live in \(x\)-space, though the two correspond loosely when
the levels are well separated.) It is the smooth counterpart of
the staircase \(f_n\). Equivalently, by Tweedie's formula
(Proposition~\ref{prop:mmse-fisher}), \(m(z)=z+(\log p_Z)'(z)\). In
computation, we form the \(p_j(z)\) using log-sum-exp and subtract
\(\max_j\ell_j(z)\) before exponentiating to avoid overflow when the levels
are far apart.
\end{remark}

\subsection{Equal-mass programs and the discrete Euler--Lagrange equation}

The most transparent special case fixes equal weights; its analysis
mirrors, term by term, the continuous Euler--Lagrange theory
(Theorem~\ref{thm:EL}).

\begin{corollary}[Equal-mass reduction]
\label{cor:equalmass}
Fix \(w_i\equiv1/n\). Then the bins \(\zeta_i=\sigma\Phi^{-1}(i/n)\) and
the barycenters \(\bar x_i\) are constants, and
the program in Proposition~\ref{prop:finite-program} reduces to the
minimization over the
\emph{ordered cone}
\begin{equation}
\mathcal C_n:=\{y\in\mathbb R^n:\ y_1\le y_2\le\dots\le y_n\}
\label{eq:ordered-cone}
\end{equation}
of the smooth function
\begin{equation}
J_n^{\mathrm{eq}}(y)
=k^2\Bigl(c_n+\tfrac1n\textstyle\sum_i(y_i-\bar x_i)^2\Bigr)
+\operatorname{mmse}(Q_n^{y}),
\qquad c_n:=\sigma^2-\tfrac1n\textstyle\sum_i\bar x_i^2,
\label{eq:eqmass-obj}
\end{equation}
where \(Q_n^y=\tfrac1n\sum_i\delta_{y_i}\). Writing
\(J_n^{\mathrm{eq},\star}:=\min_{y\in\mathcal C_n}J_n^{\mathrm{eq}}(y)\), one
has
\(J_n^{\mathrm{eq},\star}\ge J_n^\star\),
\(J_{2n}^{\mathrm{eq},\star}\le J_n^{\mathrm{eq},\star}\), and
\(J_n^{\mathrm{eq},\star}\to J^\star\) as \(n\to\infty\).
\end{corollary}

The ordering constraint is not decoration. The law \(Q_n^y\) is invariant
under permutations of \(y\), whereas the transport term
\(\tfrac1n\sum_i(y_i-\bar x_i)^2\) is not, so
\(J_n^{\mathrm{eq}}(y)=J(Q_n^y)\) holds \emph{only} when the levels are paired
with the quantile bins in increasing order: off \(\mathcal C_n\) the formula
\eqref{eq:eqmass-obj} evaluates the cost of a non-monotone assignment and
therefore strictly overstates \(k^2W_2^2(P,Q_n^y)+\operatorname{mmse}(Q_n^y)\)
(Fact~\ref{fact:1d-ot}). Since minimizers may have coincident levels
(Remark~\ref{rem:merging}), the constraint is active in exactly the cases of
interest, and Proposition~\ref{prop:discreteEL} is stated accordingly.

\begin{proof}
\emph{The objective.} Put \(w_i=1/n\) in~\eqref{eq:finite-program}. The
thresholds become \(\zeta_i=F_P^{-1}(i/n)=\sigma\Phi^{-1}(i/n)\), fixed
numbers depending only on \(n\) and \(\sigma\); hence the bins \(I_i\), their
barycenters \(\bar x_i\), and the quantity
\(c_n=\sigma^2-\tfrac1n\sum_i\bar x_i^2\) are all constants, not variables.
What remains of~\eqref{eq:finite-program} is
\(k^2\bigl(c_n+\tfrac1n\sum_i(y_i-\bar x_i)^2\bigr)
+\operatorname{mmse}(Q_n^y)\), which is~\eqref{eq:eqmass-obj}: a function of
the level vector \(y\) alone. Proposition~\ref{prop:finite-program} derived
the transport term from the \emph{monotone} coupling, which sends bin \(I_i\)
to \(y_i\); this is the optimal coupling precisely when \(y\in\mathcal C_n\),
which is why the minimization is posed over \(\mathcal C_n\). Restricting to
\(\mathcal C_n\) loses nothing, since every \(y\in\mathbb R^n\) has a sorted
rearrangement with the same law and no larger transport cost.

\emph{The bound \(J_n^{\mathrm{eq},\star}\ge J_n^\star\).} Every equal-mass
law \(Q_n^y\) is a law with at most \(n\) atoms, so it is feasible for the
full program of Proposition~\ref{prop:finite-program}; minimizing over a
subset cannot go below the minimum over the whole set.

\emph{The bound \(J_{2n}^{\mathrm{eq},\star}\le J_n^{\mathrm{eq},\star}\).}
The equal-mass values are indexed by \(n\), and the constraint sets for
different \(n\) are \emph{not} nested --- masses \(1/n\) are not masses
\(1/(n+1)\) --- so monotonicity is available only along doublings. Given
levels \(y\in\mathbb R^n\), define \(y'\in\mathbb R^{2n}\) by duplicating,
\(y'_{2i-1}=y'_{2i}=y_i\). The induced law is
\(\tfrac1{2n}\sum_{j}\delta_{y'_j}=\tfrac1n\sum_i\delta_{y_i}=Q_n^y\), the
\emph{same} measure, so the estimation terms agree. The transport terms agree
too, because both equal \(W_2^2(P,Q_n^y)\): each is the cost of the monotone
coupling of \(P\) with the same target law
(Proposition~\ref{prop:finite-program}), and the monotone coupling does not
care whether we describe it with \(n\) or \(2n\) bins. Hence
\(J_{2n}^{\mathrm{eq}}(y')=J_n^{\mathrm{eq}}(y)\), and taking the minimum
over \(y\) gives \(J_{2n}^{\mathrm{eq},\star}\le J_n^{\mathrm{eq},\star}\).

\emph{Convergence.} Fix \(Q\in\mathcal P_2(\mathbb R)\) and let
\(\nu_n:=\tfrac1n\sum_i\delta_{F_Q^{-1}((i-1/2)/n)}\) be its midpoint
quantile discretization, which is an \emph{equal-mass} law and therefore
feasible for~\eqref{eq:eqmass-obj}. As in the proof of
Proposition~\ref{prop:finite-level}, \(\nu_n\to Q\) in \(W_2\), hence
\(W_2^2(P,\nu_n)\to W_2^2(P,Q)\) and, by Fact~\ref{fact:mmse},
\(\operatorname{mmse}(\nu_n)\to\operatorname{mmse}(Q)\); so \(J(\nu_n)\to
J(Q)\) and \(\limsup_nJ_n^{\mathrm{eq},\star}\le\lim_nJ(\nu_n)=J(Q)\).
Taking the infimum over \(Q\) gives \(\limsup_nJ_n^{\mathrm{eq},\star}\le
J^\star\), while \(J_n^{\mathrm{eq},\star}\ge J_n^\star\ge J^\star\) for
every \(n\) from the first bound. The two combine to
\(J_n^{\mathrm{eq},\star}\to J^\star\).
\end{proof}

\begin{figure}[t]
\centering
\begin{tikzpicture}
\begin{axis}[width=0.80\linewidth,height=4.6cm,axis lines=left,
  xlabel={\(x\) (initial state \(X_0\))},
  xmin=-13,xmax=13,ymin=0,ymax=0.094,ytick=\empty,clip=false]
\addplot[draw=none,fill=black!12,domain=-5.338:-2.830,samples=25]{exp(-x^2/50)/12.533}\closedcycle;
\addplot[draw=none,fill=black!12,domain=-0.9:0.9,samples=25]{exp(-x^2/50)/12.533}\closedcycle;
\addplot[draw=none,fill=black!12,domain=2.830:5.338,samples=25]{exp(-x^2/50)/12.533}\closedcycle;
\addplot[very thick,domain=-13:13,samples=140]{exp(-x^2/50)/12.533};
\addplot[densely dashed,gray] coordinates {(-5.338,0) (-5.338,0.088)};
\addplot[densely dashed,gray] coordinates {(-2.830,0) (-2.830,0.088)};
\addplot[densely dashed,gray] coordinates {(-0.900,0) (-0.900,0.088)};
\addplot[densely dashed,gray] coordinates {(0.900,0) (0.900,0.088)};
\addplot[densely dashed,gray] coordinates {(2.830,0) (2.830,0.088)};
\addplot[densely dashed,gray] coordinates {(5.338,0) (5.338,0.088)};
\addplot[only marks,mark=*,mark size=1.4pt] coordinates
  {(-7.898,0) (-3.999,0) (-1.842,0) (0,0) (1.842,0) (3.999,0) (7.898,0)};
\node[font=\scriptsize] at (axis cs:0,0.045){\(\tfrac1n\)};
\node[font=\scriptsize,anchor=south] at (axis cs:-3.999,0.006){\(\bar x_i\)};
\node[font=\scriptsize,gray,anchor=south] at (axis cs:-2.830,0.089){\(\zeta_i\)};
\end{axis}
\end{tikzpicture}
\caption{The equal-mass reduction (Corollary~\ref{cor:equalmass}) for
\(n=7\), \(\sigma=5\). The prior \(P=\mathcal{N}(0,\sigma^2)\) is split
into \(n\) bins of equal probability \(1/n\) by the fixed quantile
thresholds \(\zeta_i=\sigma\Phi^{-1}(i/n)\) (dashed), narrow where \(P\)
is dense, wide in the tails. Only the \(n\) levels \(y_i\) (one per bin)
remain free; the transport pull of~\eqref{eq:discrete-EL} draws each level
toward its bin barycenter \(\bar x_i\) (dots, the centroid choice for this
\emph{fixed} partition --- not the Lloyd--Max quantizer, whose cells would
also have to be the Voronoi cells of the levels,
Remark~\ref{rem:not-lloydmax}),
while the estimation push spreads them outward
(Proposition~\ref{prop:discreteEL}).}
\label{fig:eqmass}
\end{figure}

\begin{proposition}[Discrete Euler--Lagrange equation and its KKT form]
\label{prop:discreteEL}
Write
\begin{equation}
G_i(y):=
\underbrace{\tfrac{2k^2}{n}\,(y_i-\bar x_i)}_{\text{transport pull}}
+\underbrace{\partial_{y_i}\!\operatorname{mmse}(Q_n^y)}_{\text{estimation push}},
\qquad
\partial_{y_i}\!\operatorname{mmse}(Q_n^y)
=\tfrac1n\,\mathbb{E}_N\!\big[\,2\,e_i(N)+N\,e_i(N)^2\,\big],
\label{eq:discrete-EL}
\end{equation}
where \(N\sim\mathcal N(0,1)\), \(m(z)=\mathbb E[Y\mid Z=z]\) is the
posterior mean of the mixture \(Q_n^y*\phi\), and
\(e_i(N):=y_i-m(y_i+N)\) is the estimation residual at level \(y_i\). Let
\(y\in\mathcal C_n\) be a local minimizer of~\eqref{eq:eqmass-obj} over the
ordered cone~\eqref{eq:ordered-cone}, and partition
\(\{1,\dots,n\}\) into the maximal blocks
\(B=\{i,i+1,\dots,j\}\) of indices carrying a common level
\(y_i=\dots=y_j\). Then:
\begin{enumerate}
\item[\textup{(i)}] \emph{(interior, distinct levels)} if \(B=\{i\}\) is a
singleton, \(G_i(y)=0\), the discrete counterpart of the differential
condition~\eqref{eq:EL-differential} with the displacement \(y_i-\bar x_i\)
playing the role of \(y-T^\star(y)\);
\item[\textup{(ii)}] \emph{(merged block)} for a block \(B=\{i,\dots,j\}\) with
\(|B|\ge2\) the first-order conditions are the KKT conditions of the ordered
cone: the \emph{block sum} vanishes and the partial sums have a sign,
\begin{equation}
\sum_{l\in B}G_l(y)=0,
\qquad
\sum_{l=i}^{r}G_l(y)\ \le\ 0
\quad\text{for every }r=i,\dots,j-1 .
\label{eq:kkt-block}
\end{equation}
\end{enumerate}
In particular \(G_l(y)=0\) generally \emph{fails} at a merged block: the
levels in \(B\) share one value and hence one estimation push
\(\partial_{y_l}\operatorname{mmse}\), while their centroid pulls
\(y_l-\bar x_l\) differ because the \(\bar x_l\) are distinct. What holds is
the pooled condition~\eqref{eq:kkt-block}, exactly as in isotonic regression.
Interpreting~(i): the transport pull
draws each level toward its centroid \(\bar x_i\), while the estimation
push spreads the levels apart so that \(Q_n^y*\phi\) has larger Fisher
information (Corollary~\ref{cor:fisher-var}).
\end{proposition}

\begin{proof}
The transport part of~\eqref{eq:eqmass-obj} is
\(k^2(c_n+\tfrac1n\sum_j(y_j-\bar x_j)^2)\) with \(c_n\) and the \(\bar x_j\)
independent of \(y\), so \(\partial_{y_i}\) of it is
\(\tfrac{2k^2}{n}(y_i-\bar x_i)\); hence \(\nabla J_n^{\mathrm{eq}}=G\).

\emph{From gradient to KKT.} The cone \(\mathcal C_n\) is
\(\{y:A y\le0\}\) with \(A\) the \((n-1)\times n\) difference matrix
\((Ay)_r=y_r-y_{r+1}\). At a local minimizer the first-order condition is
\(\langle G(y),d\rangle\ge0\) for every feasible direction \(d\), i.e.\ every
\(d\) with \(d_r\le d_{r+1}\) whenever \(y_r=y_{r+1}\). Feasible directions
are unconstrained across block boundaries and within a singleton block, which
gives \(G_i=0\) in case~(i) (both \(\pm e_i\) are feasible). Inside a block
\(B=\{i,\dots,j\}\), the feasible directions are those with
\(d_i\le\dots\le d_j\) and \(d_l\) free for \(l\notin B\); taking
\(d=\pm\mathbf 1_B\) (both feasible, since a common shift preserves the
equalities) gives \(\sum_{l\in B}G_l=0\), and taking
\(d=\mathbf 1_{\{r+1,\dots,j\}}\) for \(i\le r<j\) (feasible, being
nondecreasing on \(B\)) gives \(\sum_{l=r+1}^{j}G_l\ge0\); subtracting this
from the vanishing block sum gives
\(\sum_{l=i}^{r}G_l=-\sum_{l=r+1}^{j}G_l\le0\), the second half
of~\eqref{eq:kkt-block}. Equivalently these are
the KKT conditions with multipliers \(\lambda_r\ge0\) on the active
constraints \(y_r-y_{r+1}\le0\): stationarity of the Lagrangian reads
\(G_l+\lambda_l-\lambda_{l-1}=0\) with \(\lambda_{i-1}=\lambda_j=0\), so
\(\sum_{l=i}^{r}G_l=\lambda_{i-1}-\lambda_r=-\lambda_r\le0\), confirming the
sign.

It remains to compute the estimation gradient.

\emph{Step 1: differentiate the explicit formula.} We do not invoke an
envelope theorem. Doing so would be circular here: the usual argument
concludes from the nonnegativity of
\(y'\mapsto F(y',m_y)-\operatorname{mmse}(Q_n^{y'})\) at its zero that the
gradient vanishes, which already presumes the differentiability of
\(y\mapsto\operatorname{mmse}(Q_n^y)\) --- precisely what is at issue. Instead
we differentiate the closed form directly. By
Proposition~\ref{prop:finite-program} with \(w_j\equiv1/n\),
\[
\operatorname{mmse}(Q_n^y)=\frac1n\sum_jy_j^2-\int\frac{g(z)^2}{p_Z(z)}\,dz,
\qquad
p_Z=\frac1n\sum_j\phi(\cdot-y_j),
\qquad
g=\frac1n\sum_jy_j\phi(\cdot-y_j),
\]
and both \(g\) and \(p_Z\) are smooth in \((y,z)\) with \(p_Z>0\), so
\(g^2/p_Z\) is smooth; the \(y\)-derivatives up to second order are dominated
by an integrable function of \(z\), uniformly on compact \(y\)-sets, by
the estimate~\eqref{eq:dom-bounds} established in Step~2 of the proof of
Proposition~\ref{prop:lloydlimit} (that step is self-contained and does not
use the present proposition, so there is no circularity). Hence
\(\operatorname{mmse}(Q_n^{\cdot})\) is \(C^2\) and we may differentiate under
the integral sign.

\emph{Step 2: the explicit derivative.} Since \(\phi'(u)=-u\phi(u)\),
\[
\partial_{y_i}p_Z=\tfrac1n(z-y_i)\phi(z-y_i),
\qquad
\partial_{y_i}g=\tfrac1n\phi(z-y_i)\bigl[1+y_i(z-y_i)\bigr].
\]
Writing \(m=g/p_Z\) and using
\(\partial_{y_i}(g^2/p_Z)=2m\,\partial_{y_i}g-m^2\,\partial_{y_i}p_Z\),
\[
\partial_{y_i}\operatorname{mmse}
=\frac{2y_i}{n}
-\frac1n\!\int\!\phi(z-y_i)
\Bigl[2m\bigl(1+y_i(z-y_i)\bigr)-m^2(z-y_i)\Bigr]dz .
\]
Using \(\int\phi(z-y_i)\,dz=1\) to absorb \(2y_i/n\) into the integral and
collecting terms,
\[
\partial_{y_i}\operatorname{mmse}
=\frac1n\!\int\!\phi(z-y_i)
\Bigl[2\bigl(y_i-m\bigr)-(z-y_i)\bigl(2my_i-m^2\bigr)\Bigr]dz .
\]

\emph{Step 3: probabilistic form.} Substitute \(z=y_i+N\) with
\(N\sim\mathcal N(0,1)\), so that \(\phi(z-y_i)\,dz\) becomes the law of
\(N\) and \(z-y_i\) becomes \(N\); writing \(e_i(N)=y_i-m(y_i+N)\) for the
estimation residual at level \(y_i\), and using
\(e_i^2=y_i^2-2y_im+m^2\) together with the vanishing first moment
\(\mathbb E_N[N]=0\) (which lets us add \(\tfrac1n\mathbb E_N[N\,y_i^2]=0\)),
the bracket becomes \(2e_i+N e_i^2\):
\[
\partial_{y_i}\operatorname{mmse}
=\tfrac1n\,\mathbb E_N\big[2e_i(N)+N\,e_i(N)^2\big],
\]
which is~\eqref{eq:discrete-EL}. (The two terms have distinct roles: the
first pulls \(y_i\) toward the current posterior mean, the second is the
score-weighted correction accounting for the fact that moving a level also
reshapes the likelihood, and hence the estimate, at every \(z\).)
\end{proof}

\begin{proposition}[Fixed-partition centroid limit]
\label{prop:lloydlimit}
There is \(k_0=k_0(n,\sigma)\) such that for every \(k>k_0\) the
objective~\eqref{eq:eqmass-obj} has a unique minimizer \(y^\star(k)\), and
\[
y^\star(k)=\bar x+O(k^{-2})\qquad(k\to\infty).
\]
Hence the optimal equal-mass controller converges to the
\emph{fixed-partition centroid} quantizer
\(f_n^{\mathrm{cent}}=\sum_i\bar x_i\mathbf 1_{I_i}\),
in agreement with \(Q^\star_k\to P\)
(Proposition~\ref{prop:asymptotics}(a)).
\end{proposition}

\begin{remark}[This is not yet Lloyd--Max]
\label{rem:not-lloydmax}
The limit object of Proposition~\ref{prop:lloydlimit} satisfies the
\emph{centroid} condition, each level being the barycenter of its cell, but
not the \emph{nearest-neighbour} condition: the cells here are the fixed
equal-probability quantile bins \(I_i=(\zeta_{i-1},\zeta_i]\) with
\(\zeta_i=\sigma\Phi^{-1}(i/n)\), which are \emph{not} the Voronoi cells of
the resulting centroids \(\bar x_i\) (already at \(n=3\), \(\sigma=5\) the
equal-mass boundary \(\zeta_1=\sigma\Phi^{-1}(1/3)\approx-2.15\) differs from
the midpoint \(\tfrac12(\bar x_1+\bar x_2)\approx-2.73\), the centroids being
\(\{0,\pm5.45\}\)). A Lloyd--Max
quantizer requires \emph{both} conditions, so
\(f_n^{\mathrm{cent}}\) is in general strictly suboptimal as a quantizer of
\(P\). The genuine Lloyd--Max conditions are recovered only in the full
free-threshold program, where the thresholds are variables too
(Corollary~\ref{cor:lloydmax-limit}); we reserve the name ``Lloyd--Max limit''
for that statement.
\end{remark}

\begin{proof}
Throughout, the minimization is over the closed convex cone
\(\mathcal C_n\) of~\eqref{eq:ordered-cone}, as
Corollary~\ref{cor:equalmass} requires: off \(\mathcal C_n\) the
expression~\eqref{eq:eqmass-obj} is not \(J(Q_n^y)\)
(see the discussion after
Corollary~\ref{cor:equalmass}). We keep the constraint visible at every step.

\emph{Step 0: a minimizer exists.} \(J_n^{\mathrm{eq}}\) is continuous in
\(y\) and coercive, since the transport part \(\tfrac{k^2}{n}\|y-\bar x\|^2\)
tends to \(\infty\) while \(\operatorname{mmse}\ge0\); as
\(\mathcal C_n\) is closed and nonempty, \(J_n^{\mathrm{eq}}\) attains its
minimum \emph{on \(\mathcal C_n\)}. Write \(y^\star\) for any such
minimizer. (Note \(\bar x\in\mathcal C_n\), the bin barycenters being
increasing, so the comparison in Step~1 is with a feasible point.)

\emph{Step 1: minimizers are trapped near the centroids.} Comparing
\(y^\star\) with the particular choice \(y=\bar x\), and
cancelling the constant \(k^2c_n\) common to both,
\[
\frac{k^2}{n}\|y^\star-\bar x\|^2+\operatorname{mmse}(Q_n^{y^\star})
=J_n^{\mathrm{eq}}(y^\star)-k^2c_n
\ \le\ J_n^{\mathrm{eq}}(\bar x)-k^2c_n
=\operatorname{mmse}(Q_n^{\bar x}).
\]
Dropping the nonnegative \(\operatorname{mmse}(Q_n^{y^\star})\) on the left
and applying the uniform bound \(\operatorname{mmse}\le1\) of
Lemma~\ref{lem:mmse-bounds} on the right,
\[
\frac{k^2}{n}\|y^\star-\bar x\|^2\le1,
\qquad\text{i.e.}\qquad
\|y^\star-\bar x\|\le\frac{\sqrt n}{k}.
\]
In particular every minimizer lies in the fixed set
\(B:=\{y\in\mathcal C_n:\|y-\bar x\|\le1\}\) as soon as \(k\ge\sqrt n\);
\(B\) is compact and \emph{convex}, being the intersection of a ball with the
cone. Note it does
not shrink with \(k\); it merely confines the problem to a compact convex set,
which is what the next step needs.

\emph{Step 2: \(\operatorname{mmse}\) has bounded curvature on \(B\).} We
claim \(y\mapsto\operatorname{mmse}(Q_n^y)\) is \(C^2\) on a neighbourhood of
\(B\) with \(\|\nabla^2\operatorname{mmse}\|_{\mathrm{op}}\le C=C(n,\sigma)\)
there. Indeed
\(\operatorname{mmse}(Q_n^y)=\tfrac1n\sum_iy_i^2-\int g^2/p_Z\,dz\) with
\(p_Z(z)=\tfrac1n\sum_j\phi(z-y_j)\) and
\(g(z)=\tfrac1n\sum_jy_j\phi(z-y_j)\). The first sum is a quadratic
polynomial. In the second, \(g\) and \(p_Z\) are finite sums of Gaussians,
hence smooth in \((y,z)\), and \(p_Z(z)>0\) for every \(z\), so \(g^2/p_Z\) is
smooth. For the domination we exhibit the bound explicitly rather than
appeal to the boundedness of the ratios alone, which by itself would leave
polynomial factors in \(z\) uncontrolled. Let \(\Gamma:=\sup_{y\in B}\max_j
|y_j|<\infty\). Write \(g^2/p_Z=m^2p_Z\) with \(m=g/p_Z\), and note three
elementary facts, valid for all \(z\) and all \(y\in B\):
\begin{equation}
|m(z)|\le\Gamma,
\qquad
\frac{\phi(z-y_j)}{p_Z(z)}\le n,
\qquad
p_Z(z)\le\frac{1}{\sqrt{2\pi}}\,e^{-(|z|-\Gamma)^2/2}\ \ (|z|\ge\Gamma),
\label{eq:dom-bounds}
\end{equation}
the first because \(m\) is a convex combination of the levels, the second
because \(p_Z\ge\tfrac1n\phi(\cdot-y_j)\), and the third because
\(|z-y_j|\ge|z|-\Gamma\) for every \(j\). Now every \(y\)-derivative of
\(m^2p_Z\) of order \(\le2\) is a finite sum of terms of the form
\[
(\text{const})\times m^{a}\times
\prod_{l}\frac{\phi(z-y_{j_l})}{p_Z(z)}\times
\mathrm{poly}(z,y)\times p_Z(z),
\qquad a\le2,
\]
with \(\mathrm{poly}\) of degree at most \(4\) in \(z\) (each \(\partial_{y_j}\)
brings down at most one factor \((z-y_j)\), and at most two derivatives act).
Applying~\eqref{eq:dom-bounds} to each factor and keeping the last one intact,
every such term is bounded in modulus, uniformly for \(y\in B\), by
\[
C(n,\Gamma)\,\bigl(1+|z|^4\bigr)\,e^{-(|z|-\Gamma)^2/2},
\]
which is integrable in \(z\). Dominated
convergence therefore licenses differentiating twice under the integral sign,
so \(\operatorname{mmse}\) is \(C^2\), and its Hessian, being continuous on
the compact \(B\), is bounded there.

\emph{Step 3: strong convexity and uniqueness.} The transport part of
\eqref{eq:eqmass-obj} is the quadratic \(\tfrac{k^2}{n}\|y-\bar x\|^2\) plus
a constant, with Hessian \(\tfrac{2k^2}{n}I\). Hence on \(B\)
\[
\nabla^2J_n^{\mathrm{eq}}
=\frac{2k^2}{n}I+\nabla^2\operatorname{mmse}
\ \succeq\ \Bigl(\frac{2k^2}{n}-C\Bigr)I\ \succ\ 0
\qquad\text{as soon as }k^2>\frac{nC}{2},
\]
so \(J_n^{\mathrm{eq}}\) is strongly convex on the compact convex set
\(B\subseteq\mathcal C_n\). A
strongly convex function has at most one minimizer on a convex set, and by
Step~1 all minimizers over \(\mathcal C_n\) lie in \(B\); so taking
\(k_0:=\max\{\sqrt n,\sqrt{nC/2}\}\), the minimizer \(y^\star(k)\) is unique
for \(k>k_0\).

\emph{Step 4: the rate.} For \(k>k_0\) the minimizer has \emph{distinct}
levels: by Step~1 it lies within \(\sqrt n/k\) of \(\bar x\), whose entries
are distinct and separated by a fixed gap
\(\gamma_n:=\min_i(\bar x_{i+1}-\bar x_i)>0\), so once
\(2\sqrt n/k<\gamma_n\) no two coordinates can coincide and the ordering
constraint of \eqref{eq:ordered-cone} is inactive. Enlarging \(k_0\) to
\(\max\{\sqrt n,\sqrt{nC/2},2\sqrt n/\gamma_n\}\), case~(i) of
Proposition~\ref{prop:discreteEL} applies at every coordinate and
\(G_i(y^\star)=0\), which we solve for the
displacement:
\[
y_i^\star-\bar x_i=-\frac{n}{2k^2}\,\partial_{y_i}\operatorname{mmse}(Q_n^{y^\star}).
\]
The gradient is uniformly bounded on \(B\): by
Proposition~\ref{prop:discreteEL},
\(\partial_{y_i}\operatorname{mmse}=\tfrac1n\mathbb E_N[2e_i+Ne_i^2]\) with
\(e_i(N)=y_i-m(y_i+N)\), and since the posterior mean \(m\) is a convex
combination of the levels we have \(|m|\le\max_j|y_j|\), so
\(|e_i|\le2\max_j|y_j|\le2\Gamma=:\Gamma'\) uniformly on \(B\); hence
\(|\partial_{y_i}\operatorname{mmse}|
\le\tfrac1n(2\Gamma'+\Gamma'^2\mathbb E|N|)=:C'\). Therefore
\(|y_i^\star-\bar x_i|\le nC'/(2k^2)=O(k^{-2})\) for each \(i\), which is the
claim; letting \(k\to\infty\) recovers \(y^\star\to\bar x\), i.e.\ the
centroid quantizer of the \emph{fixed} equal-mass partition
(Remark~\ref{rem:not-lloydmax}).
\end{proof}

\begin{remark}[Level merging and non-uniqueness]
\label{rem:merging}
For finite \(k\), when the trade-off calls for fewer effective levels
than \(n\) the surplus levels \emph{merge}: a minimizer
of~\eqref{eq:eqmass-obj} has coincident \(y_i\), a discrete echo of the
non-uniqueness of Proposition~\ref{prop:gauss-foc}. Merged levels sit on the
boundary of the ordered cone~\eqref{eq:ordered-cone}, so the unconstrained
equation \(G_i=0\) does not hold there and the pooled
conditions~\eqref{eq:kkt-block} take its place; this is the discrete analogue
of the pool-adjacent-violators structure of isotonic regression. That
\eqref{eq:eqmass-obj} is genuinely nonconvex for small \(k\) is a separate
fact, and merging is not what shows it; what shows it is that the objective
has several distinct stationary points --- our runs from different
signalling-scaled seeds converge to constellations of different effective
level counts and different costs --- exactly as in
Proposition~\ref{prop:gauss-foc}, where the same happens already inside the
one-dimensional Gaussian slice once \(k^2<\tfrac14\).
The program in Proposition~\ref{prop:finite-program} coincides with the
formulation used in
the numerical Witsenhausen literature; what is new here is that it arises
as the exact finite-level restriction of the transport-estimation
functional~\eqref{eq:OT-variational}, with the convergence guarantee
\(J_n^\star\downarrow J^\star\) of Proposition~\ref{prop:finite-level}.
Algorithm~\ref{alg:lloyd} performs the projected descent
of~\eqref{eq:discrete-EL} on the ordered cone; run from several
signalling-scaled initializations
(and, in principle, warm-started across increasing \(n\) by level-splitting),
it reliably recovered the low-cost branch in our experiments.
\end{remark}

\begin{algorithm}[!b]
\caption{Estimation-aware Lloyd iteration (equal-mass finite-level program)}
\label{alg:lloyd}
\KwIn{penalty \(k>0\), prior std.\ \(\sigma\), levels \(n\), tolerance \(\varepsilon\)}
Fix bins \(\zeta_i=\sigma\,\Phi^{-1}(i/n)\), \(i=1,\dots,n-1\)
   (\(\zeta_0=-\infty\), \(\zeta_n=+\infty\))\;
Barycenters
   \(\bar x_i\leftarrow n\sigma^2\bigl(\phi_\sigma(\zeta_{i-1})-\phi_\sigma(\zeta_i)\bigr)\),
   \(i=1,\dots,n\)\;
\For{each start \(y^{(0)}\in\{\bar x,\ 2\bar x,\ 2.6\bar x,\ \dots\}\)}{
\tcp*[h]{centroid start plus signalling-scaled seeds; the program is
nonconvex, so a single start is not enough (Remark~\ref{rem:merging})}\;
Initialize \(y\leftarrow y^{(0)}\)\;
\Repeat{\(\|y-\Pi_{\mathcal C_n}(y-G)\|<\varepsilon\)
   \tcp*{projected-gradient residual}}{
  \emph{Estimation step:} form \(p_Z(z)=\tfrac1n\sum_j\phi(z-y_j)\) and
     \(m(z)=\sum_j y_j\phi(z-y_j)\big/\sum_j\phi(z-y_j)\)\;
  \(g_i\leftarrow\partial_{y_i}\!\operatorname{mmse}(Q_n)\) by
     one-dimensional quadrature\;
  \(G_i\leftarrow \tfrac{2k^2}{n}\,(y_i-\bar x_i)+g_i\)
     \tcp*{transport pull \(+\) estimation push}
  \emph{Quantization step:} \(y\leftarrow
     \Pi_{\mathcal C_n}\bigl(y-t\,G\bigr)\),
     \(t\) from backtracking line search on \(J\)\;
  \tcp*[h]{\(\Pi_{\mathcal C_n}\) is the isotonic projection onto
     \eqref{eq:ordered-cone} (pool-adjacent-violators); it is the identity
     while the levels stay distinct, and pools them into blocks when they
     merge, so the iterate always remains in \(\mathcal C_n\)}\;
}
Keep \(y\) if it improves on the best start so far\;
}
\KwOut{levels \(\widehat y\), controller
   \(\widehat f_n=\sum_i\widehat y_i\mathbf 1_{I_i}\), cost
   \(\widehat J_n=J(\widehat Q_n)\). Because the stopping test uses a finite
   tolerance \(\varepsilon\), the output is an \emph{approximate} KKT point
   (an \(\varepsilon\)-stationary point) of program~\eqref{eq:eqmass-obj}
   over the ordered cone~\eqref{eq:ordered-cone}, whose KKT conditions
   are~\eqref{eq:kkt-block}; it is an upper estimate of
   \(J_n^{\mathrm{eq},\star}\), not a certified minimum}
\end{algorithm}

\begin{figure}[t]
\centering
\includegraphics[width=\linewidth]{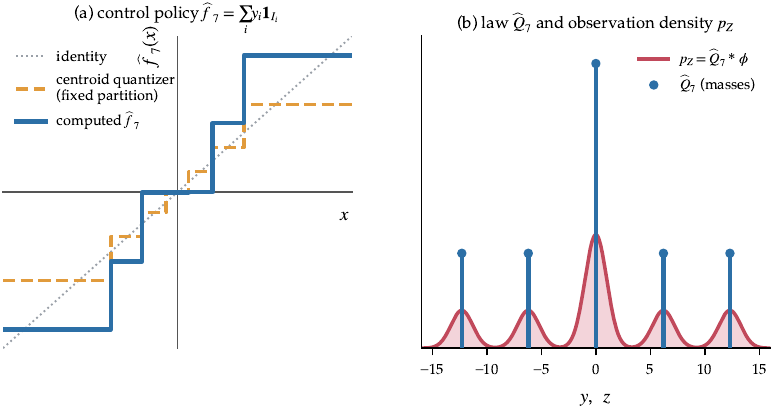}
\caption{The computed solution of the \emph{equal-mass} program
(Corollary~\ref{cor:equalmass}) for \(\sigma=5\), \(k=0.1\), \(n=7\); note
this is the restricted program, so its cost lies above the free-weight values
\(\widehat J_n\) tabulated in Example~\ref{ex:convergence}.
\emph{(a)} The control policy \(\widehat f_7\) (solid) is an \emph{output}
of Algorithm~\ref{alg:lloyd}. The dashed staircase is the
\emph{fixed-partition centroid} quantizer, which places each level at the
barycenter \(\bar x_i\) of its (fixed, equal-mass) bin: it minimizes the
transport (quantization) error alone over the levels, and is
both the algorithm's \emph{initialization} and its \(k\to\infty\) limit
(Proposition~\ref{prop:lloydlimit}), not itself a solution of the
Witsenhausen problem --- and not the classical Lloyd--Max quantizer either,
since its cells are quantiles rather than Voronoi cells
(Remark~\ref{rem:not-lloydmax}). Accounting for the estimation term pushes
the levels
outward, lowering the cost from \(J\approx0.74\) (centroid quantizer) to
\(J\approx0.14\) for the \(\widehat f_7\) shown. Both panels are for the
equal-mass
program; the free-weight computed value at the same \(k\) is lower still,
\(\widehat J_5\approx0.053\) (Example~\ref{ex:convergence}), and the
equal-mass problem is itself nonconvex, so the constellation shown is an
approximate KKT point of~\eqref{eq:kkt-block} at the solver tolerance rather
than a certified equal-mass optimum. The identity (dotted) is
\(U_1\equiv0\). \emph{(b)} The induced law \(Q_7\) as point masses
(stems) and the observation density \(p_Z=Q_7*\phi\) that Controller~2
sees. The seven levels merge into five well-separated atoms
\(\{0,\pm6.2,\pm12.3\}\) (masses
\(\tfrac37,\tfrac17,\tfrac17,\tfrac17,\tfrac17\)), whose Gaussian blurs
barely overlap, so \(Y\) is easily resolved from \(Z\) (small MMSE,
large Fisher information). This is the discrete realization of the
signalling staircase of Figure~\ref{fig:map}.}
\label{fig:fn}
\end{figure}

\subsection{The full program: power cells and generalized Lloyd--Max}
\label{sec:voronoi}

The equal-mass reduction froze the partition and moved only the levels. We
now let \emph{both} families of variables move, levels \(y_i\) and
thresholds \(\zeta_i\), and record the resulting stationarity conditions.
They reveal that the finite-level program is, structurally, an optimal
\emph{quantizer}: a nearest-neighbour partition paired with centroid
levels, both perturbed by the estimation term --- the partition becoming an
\emph{additively weighted} (power, or Laguerre) partition rather than an
ordinary Voronoi one. This characterization both reveals
the program's general structure and supplies a solver for the non-equal-mass
case.

\begin{proposition}[Optimal levels and thresholds]
\label{prop:voronoi}
Let \((y,\zeta)\) be an interior stationary point of the full
program~\eqref{eq:finite-program}, with \(w_i=P(I_i)>0\), \(\bar x_i\) the
barycenter of \(P\) on the bin \(I_i=(\zeta_{i-1},\zeta_i]\), and
\[
\Psi(y):=y^2-2y\,a(y)+b(y)
=\frac{\delta\operatorname{mmse}(Q_n)}{\delta q}(y)
\]
the estimation sensitivity of Lemma~\ref{lem:mmse-derivative}, where
\(a(y)=\int m(z)\phi(z-y)\,dz\) and \(b(y)=\int m(z)^2\phi(z-y)\,dz\) are
the smoothed posterior functionals~\eqref{eq:aq-bq-def} built from the
posterior mean \(m(z)=\mathbb E[Y\mid Z=z]\) of the mixture \(Q_n*\phi\)
(all well defined for atomic \(Q_n\), whose mixture
\(p_Z=\sum_i w_i\phi(\cdot-y_i)\) is smooth). Then
\begin{align}
\text{(levels)}\qquad
& 2k^2 w_i\,(y_i-\bar x_i)+\partial_{y_i}\!\operatorname{mmse}(Q_n)=0,
\label{eq:voronoi-level}\\[2pt]
\text{(thresholds)}\qquad
& \zeta_i=\underbrace{\frac{y_i+y_{i+1}}{2}}_{\text{Voronoi boundary}}
   +\underbrace{\frac{\Psi(y_{i+1})-\Psi(y_i)}{2k^2\,(y_{i+1}-y_i)}}_{\text{MMSE correction}}
\qquad(i=1,\dots,n-1).
\label{eq:voronoi-thresh}
\end{align}
Equivalently, \eqref{eq:voronoi-thresh} says that the stationary partition
assigns each state to the level minimizing a \emph{penalized} distance,
\begin{equation}
x\ \longmapsto\ \operatorname*{arg\,min}_i
\Bigl\{\,(x-y_i)^2+k^{-2}\Psi(y_i)\,\Bigr\},
\label{eq:power-cell}
\end{equation}
since the two sides of~\eqref{eq:voronoi-thresh} are exactly where two such
penalized distances agree. The cells are therefore \emph{additively weighted}
Voronoi cells --- power (Laguerre) cells with weights \(k^{-2}\Psi(y_i)\) ---
and not ordinary Voronoi cells: the \emph{midpoint} \(\tfrac12(y_i+y_{i+1})\)
is the nearest-neighbour boundary, and it is displaced by a term that
vanishes only as \(k\to\infty\). Ordinary Voronoi cells are recovered exactly
when the weights are equal, in particular in the limit
\(k\to\infty\) (Corollary~\ref{cor:lloydmax-limit}).
\end{proposition}

\begin{proof}
Parametrize the program by \((y,\zeta)\); the masses are
\(w_i=F_P(\zeta_i)-F_P(\zeta_{i-1})\) and
\[
J=k^2\sum_{j}\int_{\zeta_{j-1}}^{\zeta_j}(x-y_j)^2\,dP
+\operatorname{mmse}(Q_n).
\]
\emph{Levels.} The level \(y_i\) appears in exactly one transport summand,
the \(j=i\) one, and (through the atom locations of \(Q_n\)) in the
estimation term. Differentiating the transport summand under the integral
sign, the bin \(I_i\) being fixed,
\[
\partial_{y_i}\!\int_{I_i}(x-y_i)^2\,dP
=\int_{I_i}\!\partial_{y_i}(x-y_i)^2\,dP
=-2\int_{I_i}(x-y_i)\,dP
=-2\bigl(\mu_i-y_iw_i\bigr)
=2w_i\bigl(y_i-\bar x_i\bigr),
\]
the last step by \(\mu_i=w_i\bar x_i\). Multiplying by \(k^2\) and adding
\(\partial_{y_i}\!\operatorname{mmse}(Q_n)\)
gives~\eqref{eq:voronoi-level}. (The estimation gradient is the one computed
in Proposition~\ref{prop:discreteEL}, with the weight \(w_i\) in place of
\(1/n\): that computation used the equal weights only to write the prefactor.)

\emph{Thresholds.} Now fix the levels and move a single boundary
\(\zeta_i\). It is the right endpoint of \(I_i\) and the left endpoint of
\(I_{i+1}\), so exactly two summands of the transport term, and exactly two
masses, depend on it. The masses move in opposite directions at equal rate:
by the fundamental theorem of calculus applied to
\(w_i=F_P(\zeta_i)-F_P(\zeta_{i-1})\) and
\(w_{i+1}=F_P(\zeta_{i+1})-F_P(\zeta_i)\),
\[
\partial_{\zeta_i}w_i=p(\zeta_i)=-\partial_{\zeta_i}w_{i+1},
\]
which says simply that raising \(\zeta_i\) transfers \(P\)-mass at rate
\(p(\zeta_i)\) from bin \(i+1\) to bin \(i\).

For the transport term, the integrands do not depend on \(\zeta_i\), only the
limits do, so Leibniz's rule leaves only boundary evaluations, and the two
share the same boundary point \(\zeta_i\):
\[
\partial_{\zeta_i}\!\Bigl[\!\int_{\zeta_{i-1}}^{\zeta_i}\!(x-y_i)^2dP
+\!\int_{\zeta_i}^{\zeta_{i+1}}\!(x-y_{i+1})^2dP\Bigr]
=p(\zeta_i)(\zeta_i-y_i)^2-p(\zeta_i)(\zeta_i-y_{i+1})^2 ,
\]
with a plus sign for the upper limit of the first integral and a minus sign
for the lower limit of the second.

For the estimation term we differentiate \emph{along the simplex}, never off
it. With the atoms held fixed, moving \(\zeta_i\) moves the measure
\(Q_n=\sum_jw_j\delta_{y_j}\) in the single mass-preserving direction
\begin{equation}
\frac{d\,Q_n}{d\zeta_i}
=p(\zeta_i)\bigl(\delta_{y_i}-\delta_{y_{i+1}}\bigr),
\label{eq:dQ-dzeta}
\end{equation}
by the mass-transfer rates above: raising \(\zeta_i\) moves mass from atom
\(i+1\) to atom \(i\) at rate \(p(\zeta_i)\), and no other weight changes.
This direction is a signed measure of \emph{total mass zero}, which is exactly
the class along which the first variation of
Lemma~\ref{lem:mmse-derivative} is defined --- and along which its
representative \(\Psi\) is unambiguous, an additive constant in \(\Psi\)
integrating to zero. Hence
\[
\partial_{\zeta_i}\!\operatorname{mmse}
=\int\Psi\;d\Bigl(\frac{dQ_n}{d\zeta_i}\Bigr)
=p(\zeta_i)\bigl[\Psi(y_i)-\Psi(y_{i+1})\bigr],
\]
which is meaningful for atomic \(Q_n\) because \(\Psi\) is built from
\(m\) and \(p_Z=\sum_jw_j\phi(\cdot-y_j)\), both smooth with \(p_Z>0\). (One
could instead write \(\partial\operatorname{mmse}/\partial w_j=\Psi(y_j)\)
and apply the chain rule, but that differentiates in a direction leaving the
probability simplex; \eqref{eq:dQ-dzeta} avoids the issue entirely and gives
the same answer.)
Hence
\[
\partial_{\zeta_i}J
=p(\zeta_i)\Bigl\{k^2\bigl[(\zeta_i-y_i)^2-(\zeta_i-y_{i+1})^2\bigr]
+\Psi(y_i)-\Psi(y_{i+1})\Bigr\}.
\]
As \(p(\zeta_i)>0\) (the prior has a strictly positive density), the brace
must vanish at a stationary point. Factor the difference of squares, with
\(a=\zeta_i-y_i\) and \(b=\zeta_i-y_{i+1}\), so that \(a-b=y_{i+1}-y_i\) and
\(a+b=2\zeta_i-y_i-y_{i+1}\):
\[
(\zeta_i-y_i)^2-(\zeta_i-y_{i+1})^2=a^2-b^2=(a-b)(a+b)
=(y_{i+1}-y_i)\bigl(2\zeta_i-y_i-y_{i+1}\bigr).
\]
The vanishing brace therefore reads
\(k^2(y_{i+1}-y_i)(2\zeta_i-y_i-y_{i+1})=\Psi(y_{i+1})-\Psi(y_i)\). Since the
levels are distinct, \(y_{i+1}-y_i\ne0\), and we may divide:
\[
2\zeta_i-y_i-y_{i+1}=\frac{\Psi(y_{i+1})-\Psi(y_i)}{k^2\,(y_{i+1}-y_i)} ,
\]
and dividing by \(2\) and rearranging gives~\eqref{eq:voronoi-thresh}.
Observe also, for later use in Remark~\ref{rem:general-solver}, that
\(\zeta_i\mapsto\partial_{\zeta_i}J\) is \(p(\zeta_i)>0\) times a
\emph{strictly increasing affine} function of \(\zeta_i\) (slope
\(2k^2(y_{i+1}-y_i)>0\)) when \(\Psi\) is held fixed, so the stationary point
so obtained is a minimum in \(\zeta_i\), not a maximum or an inflection.
\end{proof}

\begin{remark}[The estimation sensitivity in closed form]
\label{rem:psi-closed}
The sensitivity \(\Psi\) has an exact representation that names its role.
Because \(\phi(z-y)\) is the density of \(Z=y+N\) given \(Y=y\), the
smoothed functionals of Lemma~\ref{lem:mmse-derivative} are the conditional
moments \(a(y)=\int m(z)\phi(z-y)\,dz=\mathbb E[m(Z)\mid Y=y]\) and
\(b(y)=\mathbb E[m(Z)^2\mid Y=y]\); completing the square,
\begin{equation}
\Psi(y)=y^2-2y\,a(y)+b(y)=\mathbb E\bigl[(y-m(Z))^2\bigm|Y=y\bigr].
\label{eq:psi-condexp}
\end{equation}
So \(\Psi(y)\) is the mean-squared gap between the level \(y\) and the MMSE
estimate \(m(Z)\), averaged over the channel noise given that the source
sits at \(y\), literally how badly level \(y\) is resolved after the noise.
This is why it acts as an \emph{estimation sensitivity} and drives both the
threshold correction~\eqref{eq:voronoi-thresh} and the level
condition~\eqref{eq:voronoi-level}. The representation is exact but not
elementary: for \(n=1\) (a single atom \(y_1\), \(m\equiv y_1\)) it is
\(\Psi(y)=(y-y_1)^2\); for \(n\ge2\), \(m\) is the softmax
ratio~\eqref{eq:softmax}, whose convolution against \(\phi\) has no
elementary antiderivative, so \(\Psi\) is evaluated by the one-dimensional
Gaussian quadrature of Algorithm~\ref{alg:lloyd}.

Two distinct mechanisms can switch the estimation correction off, and they
should not be confused. In the \emph{well-separated} regime --- the
signalling branch observed at moderately small \(k\), where the levels are
driven far apart (and not, by
Proposition~\ref{prop:asymptotics}(b), the \(k\downarrow0\) limit itself) ---
\(m(Z)\approx y\) with high probability given \(Y=y\), so
\(\Psi(y)\to0\): each level is resolved almost perfectly. In the
\emph{expensive}-control regime \(k\to\infty\) the levels do \emph{not}
separate, they converge to the fixed Lloyd--Max centroids of \(P\), so
\(\Psi\) does \emph{not} vanish; what vanishes is the correction term
in~\eqref{eq:voronoi-thresh}, and it does so because of the explicit factor
\(k^{-2}\) multiplying the (bounded) increment \(\Psi(y_{i+1})-\Psi(y_i)\).
It is this second mechanism that produces the Lloyd--Max limit
(Corollary~\ref{cor:lloydmax-limit}).
\end{remark}

\begin{figure}[!ht]
\centering
\begin{tikzpicture}
\begin{axis}[width=0.95\linewidth,height=4.8cm,axis lines=left,
  axis on top=true,
  xlabel={state space \(x\)},ylabel={decision level},
  xmin=-11,xmax=11,ymin=0,ymax=1.22,ytick=\empty,xtick=\empty,clip=false]
\addplot[draw=none,fill=black!4]  coordinates {(-11,0)(-5.75,0)(-5.75,1)(-11,1)}\closedcycle;
\addplot[draw=none,fill=black!11] coordinates {(-5.75,0)(-1.75,0)(-1.75,1)(-5.75,1)}\closedcycle;
\addplot[draw=none,fill=black!4]  coordinates {(-1.75,0)(1.75,0)(1.75,1)(-1.75,1)}\closedcycle;
\addplot[draw=none,fill=black!11] coordinates {(1.75,0)(5.75,0)(5.75,1)(1.75,1)}\closedcycle;
\addplot[draw=none,fill=black!4]  coordinates {(5.75,0)(11,0)(11,1)(5.75,1)}\closedcycle;
\addplot[thick] coordinates {(-5.75,0)(-5.75,1)};
\addplot[thick] coordinates {(-1.75,0)(-1.75,1)};
\addplot[thick] coordinates {(1.75,0)(1.75,1)};
\addplot[thick] coordinates {(5.75,0)(5.75,1)};
\addplot[only marks,mark=*,mark size=1.7pt] coordinates
  {(-8,0.5)(-3.5,0.5)(0,0.5)(3.5,0.5)(8,0.5)};
\node[font=\scriptsize,anchor=north] at (axis cs:-8,0.44){\(y_1\)};
\node[font=\scriptsize,anchor=north] at (axis cs:-3.5,0.44){\(y_2\)};
\node[font=\scriptsize,anchor=north] at (axis cs:0,0.44){\(y_3\)};
\node[font=\scriptsize,anchor=north] at (axis cs:3.5,0.44){\(y_4\)};
\node[font=\scriptsize,anchor=north] at (axis cs:8,0.44){\(y_5\)};
\node[font=\scriptsize] at (axis cs:-8.6,0.86){\(I_1\)};
\node[font=\scriptsize] at (axis cs:-3.75,0.86){\(I_2\)};
\node[font=\scriptsize] at (axis cs:0,0.86){\(I_3\)};
\node[font=\scriptsize] at (axis cs:3.75,0.86){\(I_4\)};
\node[font=\scriptsize] at (axis cs:8.6,0.86){\(I_5\)};
\draw[<->,gray] (axis cs:0,1.08) -- (axis cs:1.75,1.08)
  node[midway,above,font=\scriptsize,black]{\(d\)};
\draw[<->,gray] (axis cs:1.75,1.08) -- (axis cs:3.5,1.08)
  node[midway,above,font=\scriptsize,black]{\(d\)};
\addplot[densely dashed] coordinates {(-5.35,0)(-5.35,1.06)};
\draw[<->] (axis cs:-5.75,0.32) -- (axis cs:-5.35,0.32);
\node[font=\scriptsize,anchor=south] at (axis cs:-5.35,1.07){MMSE shift};
\node[font=\scriptsize,anchor=west] at (axis cs:-5.26,0.80){\(\zeta_1\)};
\end{axis}
\end{tikzpicture}
\caption{Decision cells of a five-level controller. The solid boundaries are
the midpoints \(\tfrac12(y_i+y_{i+1})\), equidistant (\(d\) each) from the two
adjacent levels: this is the one-dimensional \emph{Voronoi} partition of the
levels, which is the stationary partition only in the limit \(k\to\infty\)
(Corollary~\ref{cor:lloydmax-limit}). For finite
\(k\) the true threshold \(\zeta_i\) sits off the
midpoint (dashed), shifted toward the level of larger estimation sensitivity
by the correction in~\eqref{eq:voronoi-thresh}; the cells are then
\emph{power} (Laguerre) cells with weights \(k^{-2}\Psi(y_i)\), so each state
goes to the level minimizing the penalized distance~\eqref{eq:power-cell}
rather than to the nearest level. The shift vanishes as
\(k\to\infty\).}
\label{fig:voronoi-cells}
\end{figure}

\begin{corollary}[Classical Lloyd--Max as \(k\to\infty\)]
\label{cor:lloydmax-limit}
Let \((y(k),\zeta(k))_{k\ge k_1}\) be any family of interior stationary points
of the full program~\eqref{eq:finite-program} that is \emph{nondegenerate},
meaning that for some constants \(\Gamma,\underline w,\gamma>0\) and all
\(k\ge k_1\),
\begin{equation}
|y_i(k)|\le\Gamma,
\qquad
w_i(k)=P(I_i(k))\ge\underline w,
\qquad
y_{i+1}(k)-y_i(k)\ \ge\ \gamma
\label{eq:nondegenerate}
\end{equation}
(levels bounded, no cell losing its mass, adjacent levels uniformly
separated). Then along such a family the MMSE corrections
in~\eqref{eq:voronoi-level} and \eqref{eq:voronoi-thresh} are \(O(k^{-2})\)
and vanish, so the
stationarity conditions become
\[
y_i=\bar x_i\ \text{(centroid)},\qquad
\zeta_i=\tfrac12(y_i+y_{i+1})\ \text{(midpoint)} :
\]
in the limit: the bins are the \emph{Voronoi cells} of the levels and each
level is the
\emph{centroid} of its cell. These are exactly the Lloyd--Max optimality
conditions for the \(n\)-point quantization of \(P\). The finite-level
program is therefore an optimal quantizer \emph{regularized by} the
estimation cost \(\operatorname{mmse}\), reducing to the classical
quantizer when control is expensive.
\end{corollary}

The hypotheses~\eqref{eq:nondegenerate} cannot simply be dropped, and we do
not derive them from Proposition~\ref{prop:lloydlimit}: that proposition
concerns the \emph{equal-mass, fixed-partition} subproblem, in which the cells
are frozen quantile bins, and says nothing about stationary points of the
full free-threshold program. Degenerate families are not excluded a priori ---
a cell can lose its mass, or two levels merge, as they do at finite \(k\)
(Remark~\ref{rem:merging}) --- and along such a family the corrections need
not be \(O(k^{-2})\). The corollary is therefore a statement about
nondegenerate branches, which is what the numerics track. (For \emph{global}
\(n\)-level minimizers one can argue differently: dividing \(J\) by \(k^2\)
and using \(0\le\operatorname{mmse}\le1\), the objective converges uniformly
to the ordinary \(n\)-point quantization error of \(P\), whose minimizers are
the Lloyd--Max quantizers; we do not need that route here.)

\begin{proof}
Fix a family satisfying~\eqref{eq:nondegenerate}. Under those bounds both
estimation corrections are uniformly \(O(k^{-2})\).
For the levels, divide~\eqref{eq:voronoi-level} by \(2k^2w_i\):
\[
y_i-\bar x_i=-\frac{\partial_{y_i}\!\operatorname{mmse}}{2k^2w_i},
\qquad\text{and}\qquad
\bigl|\partial_{y_i}\!\operatorname{mmse}\bigr|
\le w_i\bigl(2\Gamma'+\Gamma'^2\,\mathbb E|N|\bigr)
\]
by Proposition~\ref{prop:discreteEL} with \(\Gamma':=2\Gamma\) bounding the
residual \(|e_i|\), exactly as in Step~4 of the proof of
Proposition~\ref{prop:lloydlimit}; the factors \(w_i\) cancel and
\(|y_i-\bar x_i|=O(k^{-2})\). For the thresholds, \(\Psi\) is bounded by
\(\Gamma'^2\) on the
compact set \(\{|y|\le\Gamma\}\) (it is the conditional second moment
\eqref{eq:psi-condexp} of a quantity bounded by \(\Gamma'\)), and by
\eqref{eq:nondegenerate} the levels stay separated by \(\gamma\), so the
correction
\((\Psi(y_{i+1})-\Psi(y_i))/(2k^2(y_{i+1}-y_i))\)
in~\eqref{eq:voronoi-thresh} is at most \(\Gamma'^2/(k^2\gamma)=O(k^{-2})\).
Both therefore vanish
as \(k\to\infty\), leaving \(y_i=\bar x_i\) and
\(\zeta_i=\tfrac12(y_i+y_{i+1})\): the centroid and nearest-neighbour
conditions defining a Lloyd--Max quantizer
\cite{lloyd1982least,gray1998quantization}.
\end{proof}

\subsection{Solver and continuation strategy}
\label{sec:solver-strategy}

\begin{remark}[A general solver, and higher dimensions]
\label{rem:general-solver}
Conditions~\eqref{eq:voronoi-level} and \eqref{eq:voronoi-thresh} define a
\emph{generalized Lloyd iteration} for the full (non-equal-mass) program:
alternate a \emph{centroid step}, move each level to
\(\bar x_i-\partial_{y_i}\!\operatorname{mmse}/(2k^2w_i)\), with an
\emph{assignment step}, reset each threshold to the corrected
midpoint~\eqref{eq:voronoi-thresh} --- equivalently, reassign each state by
the penalized rule~\eqref{eq:power-cell}. The level step is the line-searched
descent of Algorithm~\ref{alg:lloyd} (now with the thresholds unfrozen), and
is monotone in \(J\) by construction. The threshold step is exact
\emph{for frozen} \(\Psi\): with the levels and the values \(\Psi(y_i)\) held
fixed, \(\zeta_i\mapsto\partial_{\zeta_i}J\) is \(p(\zeta_i)\) times a
strictly increasing affine function of \(\zeta_i\) (proof of
Proposition~\ref{prop:voronoi}), so~\eqref{eq:voronoi-thresh} is its
\emph{minimizer}, not merely a stationary point. Since \(\Psi\) itself depends
on the masses \(w_i=P(I_i)\), and hence on \(\zeta\), the step is \emph{not}
an exact block minimization of \(J\), and we do not claim it is a
majorize-minimize step either: we construct no surrogate that majorizes \(J\)
and touches it at the current iterate. It is a damped fixed-point update of
the thresholds with descent acceptance --- we accept it only if \(J\)
decreases and damp toward the previous
\(\zeta\) otherwise. That makes the alternation monotone in \(J\) and
therefore
convergent \emph{in value}; stationarity of the accumulation points would need
additional regularity (a closed-map or Zangwill-type argument) that we do not
establish. The MMSE corrections are what distinguish the
Witsenhausen quantizer from a plain source quantizer: they push the levels
apart (larger Fisher information, Corollary~\ref{cor:fisher-var}) and shift
each boundary toward the level of larger estimation sensitivity. The
structure is dimension-free: for the vector counterpart the decision
regions become \emph{polytopes}, namely the power (Laguerre) cells
\eqref{eq:power-cell} of the levels with the same weights
\(k^{-2}\Psi(y_i)\), reducing to ordinary Voronoi polytopes as
\(k\to\infty\); the problem is thus an MMSE-regularized
\emph{vector} quantizer solvable by the same alternation
(\S\ref{sec:vector}).
\end{remark}

\begin{remark}[Solving the nonconvex program: deterministic annealing in
\(k\)]
\label{rem:annealing}
The program is nonconvex, so the generalized Lloyd iteration of
Remark~\ref{rem:general-solver} is only monotone in value, with no guarantee
of reaching a global minimum (and, as noted there, not even a proof that its
accumulation points are stationary). This
is not a defect of the reformulation but the familiar situation of
quantizer design (even plain \(k\)-means is nonconvex), and the quantizer
view makes its remedies available here. Three features of the cell
structure make them effective. First, in one dimension the optimal
partition is \emph{order-preserving} (Fact~\ref{fact:1d-ot}): each
cell is an interval, so there is no combinatorial search over assignments,
only the placement of the \(2n-1\) reals \((y,\zeta)\). Second, with the current estimation
sensitivities \(\Psi(y_i)\) frozen, each threshold has a
\emph{closed-form} fixed-point update~\eqref{eq:voronoi-thresh}, so the
partition is never searched combinatorially --- though, as
Remark~\ref{rem:general-solver} notes, this is a fixed-point step and not an
exact block minimization, because \(\Psi\) itself depends on the masses and
hence on \(\zeta\). Third, the transport-only part (\(k\to\infty\),
Corollary~\ref{cor:lloydmax-limit}) is a scalar quantizer, whose global
optimum is computable by dynamic programming. Building on these, the quantizer structure makes a
\emph{homotopy in \(k\)} available as a principled solver, deterministic
annealing \cite{rose1998deterministic} with \(k\) as a temperature:
begin in the large-\(k\) regime and decrease \(k\),
tracking the solution as its levels \emph{split at successive
bifurcations}. Three distinct large-\(k\) anchors are available, and they
should not be conflated: for the \emph{equal-mass} program the minimizer is
unique for \(k>k_0(n,\sigma)\) and converges to the fixed-partition centroid
quantizer (Proposition~\ref{prop:lloydlimit}); for the \emph{full}
finite-level program the nondegenerate stationary points converge to an
ordinary \(n\)-level Lloyd--Max quantizer of \(P\)
(Corollary~\ref{cor:lloydmax-limit}); and for the \emph{continuous} problem
the near-identity solution is affine, that branch being second-variation
stable above \(k_{\mathrm{stab}}\) (Remark~\ref{rem:hermite}) --- though
global uniqueness of the continuous large-\(k\) minimizer is not proved, so
even this anchor is a branch rather than a certified starting optimum.
Splitting steps in the style of Linde--Buzo--Gray
\cite{linde1980algorithm} seed each new level and the Lloyd alternation
polishes it. This continuation is a heuristic for tracking a \emph{low-cost
stationary branch} through the mode
transitions rather than freezing at a fixed multi-modal local minimum; no
guarantee that the branch it tracks is the global one is available, here or
in the deterministic-annealing literature. The
specific solutions reported here (Figures~\ref{fig:fn},~\ref{fig:density},
and~\ref{fig:keffect}) were instead obtained by the simpler
shortcut this picture licenses: at each \(k\) we solve from a signalling seed
of the appropriate multiplicity, most directly the finite-level constellation
at that \(k\), which in our runs lands in the same basin the annealing branch
occupies. We
do not sweep and warm-start across \(k\); indeed, carrying a fixed level count
from one penalty to the next can retain too many levels and return a
\emph{worse} stationary point, whereas seeding each \(k\) afresh from its
finite-level solution recovered the best structure we found.

What the structure does \emph{not} give is a global solve at no cost: the
estimation term \(\operatorname{mmse}(Q_n)=\sum_iw_iy_i^2-\int g^2/p_Z\)
couples all levels through the mixture overlap
\(p_Z=\sum_jw_j\phi(\cdot-y_j)\), so it is not a sum of per-bin costs and
the optimal-substructure property behind the scalar-quantizer dynamic program
fails. This prevents the direct application of that classical dynamic
program; it is an obstruction to \emph{that particular method}, not a
complexity lower bound, and we make no claim about the intrinsic hardness of
the finite-level program.
The coupling decays exponentially in level separation, however, so on the
signalling branch, where in the explicit separated-level construction an
estimation cost of order \(k^2\) goes with a
level separation of order \(\sqrt{\log(1/k)}\)
(Corollary~\ref{cor:limit-controllers}(b\('\))
places the two-level construction at \(\pm M\), a separation of
\(2M=8\sqrt{\log(1/k)}\)),
the estimation cost nearly decouples into nearest-neighbour confusions; a
surrogate that keeps only those terms is banded and admits an exact
one-dimensional dynamic program per iteration, and the large-\(k\) convex
anchor together with such a homotopy would provide a practical certificate
along the tracked branch.
\end{remark}

\section{Numerical Study and Asymptotics}
\label{sec:numerics}

We now solve the finite-dimensional program of \S\ref{sec:finite-program-sec}
numerically. We also characterize analytically the two extreme regimes---cheap
and expensive control---thereby connecting the analysis to the explicit
limiting controllers of the original problem.

\subsection{A numerical illustration}

We solve the \emph{full} (free-weight) finite-level
program~\eqref{eq:finite-program} for \(\sigma=5\) and \(k=0.1\), by the
generalized Lloyd alternation of Remark~\ref{rem:general-solver}, which moves
levels and thresholds together. (The equal-mass reduction of
Corollary~\ref{cor:equalmass}, driven by Algorithm~\ref{alg:lloyd}, is a
restriction of it and returns strictly larger values at the same \(n\); the
two should not be conflated, cf.\ Figure~\ref{fig:fn}.) For each \(n\) we
report \(\widehat J_n\), the least cost found using at most \(n\) atoms, over
several signalling-scaled initializations and taking the running minimum over
smaller level counts; \(\widehat J_n\ge J_n^\star\), and taking the running
minimum is what makes \(\widehat J_n\) nonincreasing, so it is a monotone
upper estimate of the sequence \(J_n^\star\downarrow J^\star\) of
Proposition~\ref{prop:finite-level}. (Without that safeguard the raw output
of a nonconvex solver need not decrease in \(n\), and in our runs it does
not.)

\begin{example}[Finite-level convergence]
\label{ex:convergence}
With \(\sigma=5\) and \(k=0.1\), the computed values are
\[
\begin{array}{c|ccccccc}
n & 1 & 2 & 3 & 4 & 5 & 6 & 7\\[2pt]\hline
\widehat J_n & 0.250 & 0.092 & 0.060 & 0.0535 & 0.0529 & 0.0529 & 0.0529
\end{array}
\]
The sequence decreases and appears to stabilize by \(n=5\)
(Figure~\ref{fig:converge}) at \(\widehat J_n\approx0.053\), an upper
estimate of \(J^\star\); adding further levels yields no improvement, the
surplus levels merging as in Remark~\ref{rem:merging}. The plateau is reached
at the symmetric five-atom law \(\widehat Q_5\) with
\[
\begin{array}{c|ccccc}
y_i & -15.03 & -7.38 & 0 & 7.38 & 15.03\\[2pt]\hline
w_i & 0.012 & 0.217 & 0.542 & 0.217 & 0.012
\end{array}
\]
(each of the masses \(0.217\) and \(0.012\) is carried by \emph{each} member
of the corresponding symmetric pair; the five masses sum to \(1\)),
so the effective level count is five, matching the five plateaus of the
computed controller (Figure~\ref{fig:fn} shows the same count for the
equal-mass restriction, though at different levels and a higher cost). The
single-level value \(\widehat J_1=k^2\sigma^2=0.25\) is the full-cancellation
policy \(Q=\delta_0\).
\end{example}

\begin{figure}[t]
\centering
\includegraphics[width=0.60\linewidth]{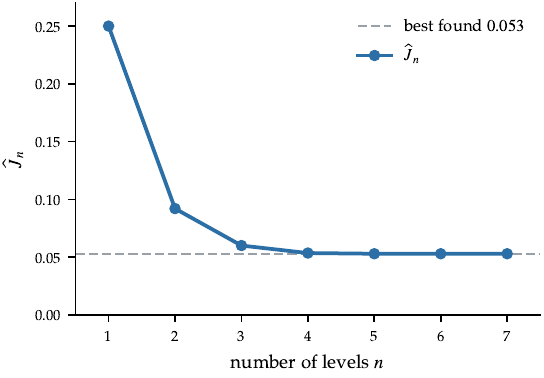}
\caption{Computed finite-level upper estimates as \(n\) increases
(\(\sigma=5\), \(k=0.1\)). The decrease shown is that of the computed
running-minimum sequence; Proposition~\ref{prop:finite-level} separately
proves \(J_n^\star\downarrow J^\star\) for the \emph{exact} finite-level
values. These are
the \emph{free-weight} values \(\widehat J_n\) of
Example~\ref{ex:convergence}, produced by the generalized Lloyd alternation
of Remark~\ref{rem:general-solver}, which moves levels \emph{and} thresholds;
they are not the equal-mass values of Algorithm~\ref{alg:lloyd}, which are
larger at the same \(n\) (Figure~\ref{fig:fn} shows
\(J\approx0.14\) for the equal-mass \(n=7\), against
\(\widehat J_5\approx0.053\) here). The cost stabilizes at \(n=5\),
the number of effective levels of the computed controller.}
\label{fig:converge}
\end{figure}

\begin{figure}[t]
\centering
\includegraphics[width=\linewidth]{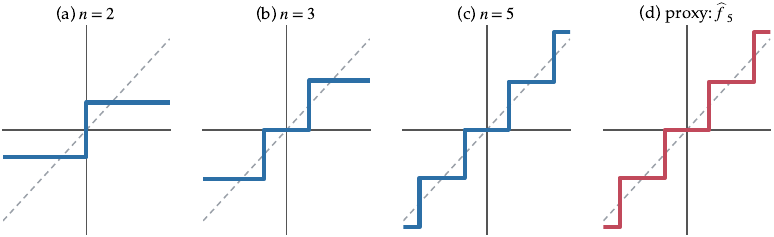}
\caption{Convergence of the computed finite-level controller
\(\widehat f_n=F_{\widehat Q_n}^{-1}\circ F_P\)
at fixed \(k=0.1\) (\(\sigma=5\)); axes are \(x\) (horizontal) and
\(\widehat f_n(x)\) (vertical), the identity \(U_1\equiv0\) shown gray. Each
is the
monotone transport map pushing \(P\) onto \(\widehat Q_n\), a signalling
staircase
that gains levels and sharpens as \(n\) grows, with computed costs
\(\widehat J_n=0.092,0.060,0.053\) --- a decreasing sequence of upper
estimates, not a proof of convergence to \(J^\star\), which
Proposition~\ref{prop:finite-level} establishes separately for the exact
values \(J_n^\star\) (Figure~\ref{fig:converge}).
\emph{(d)} The finite-level proxy at \(n=5\)
(\(\widehat J_5\approx0.053\)): by \(n=5\) the outer step already carries
negligible probability and further levels change neither the cost nor the
visible profile, so \(\widehat f_5\) is our best finite-level stand-in for
\(f^\star\). It is not \(f^\star\) itself: \(\widehat f_5\) is a staircase
with five plateaus, whereas by Proposition~\ref{prop:ac} and
Fact~\ref{fact:regularity} \(f^\star\) is strictly increasing (and
\(\widehat Q_5\) is atomic, whereas \(Q^\star\) is not).}
\label{fig:solconv-f}
\end{figure}

\begin{figure}[t]
\centering
\includegraphics[width=\linewidth]{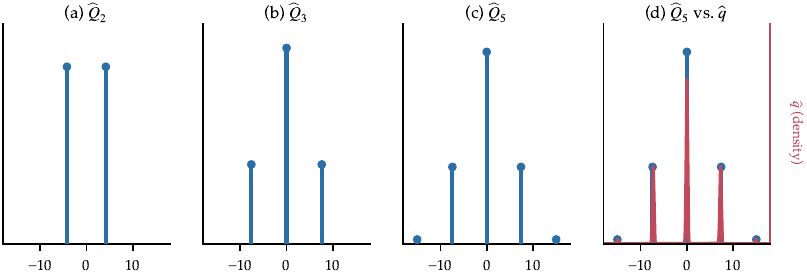}
\caption{Convergence of the computed \emph{law} at fixed \(k=0.1\). Two
different objects appear, and it is worth being explicit about which. Blue
stems are the computed finite-level atomic \emph{measures}
\(\widehat Q_n=\sum_i w_i\delta_{y_i}\):
their heights are point \emph{masses} \(w_i\) (left axis), the natural
representation of an atomic law, which has no density. As \(n\) grows the
atoms proliferate and separate, but beyond the effective count add only
negligible-weight satellites (the \(\widehat Q_5\) atoms at \(\pm15\) carry
mass
\(0.012\)), so \(\widehat Q_3\) and \(\widehat Q_5\) nearly coincide. These
atomic laws are
finite-level \emph{stand-ins}, not the optimum itself: by
Proposition~\ref{prop:ac} the true \(Q^\star\) is \emph{atomless}, hence
carries no point masses at all; it has instead a smooth \emph{density}
\(q^\star=\mathrm dQ^\star/\mathrm dy\) (Fact~\ref{fact:regularity}), symmetric
about \(0\) and supported on a single interval, with \(q^\star>0\)
\emph{almost everywhere} there (Proposition~\ref{prop:ac-gauss}; isolated
zeros are not excluded).
Panel~\emph{(d)} plots the reconstruction \(\widehat q\) of the push-forward
of the computed map
(crimson curve on the right \emph{density} axis; the same
transport-map solve as Figure~\ref{fig:density}), against the atomic
\(\widehat Q_5\): about five \emph{tall, narrow} modes
sitting essentially where the atoms are, joined by a thin floor. The rendered
widths are bandwidth-limited (the reconstruction of \S\ref{app:scope} uses
\(h=0.16\) here) and should not be read as a measurement of the true mode
width; the theoretical statement is only about the profitable spreading scale
in Remark~\ref{rem:atomic-vs-ac}.
The stems (masses) and the curve (a density) live on different vertical scales
precisely because \(\widehat Q_n\) is atomic and \(\widehat Q_{\mathrm{cont}}\)
is not. Note
that no single \(\widehat Q_n\) has \(Q^\star\) as its limit; it is the
\emph{sequence} \((Q_n^\star)\) whose weak limit points are minimizers
(Remark~\ref{rem:atomic-vs-ac}). This near-atomic shape is
why the finite-level proxy
is so accurate in the moderate-\(k\) signalling regime; the peaks sharpen as
\(k\) decreases through that window, though not indefinitely --- by
Proposition~\ref{prop:asymptotics}(b) the law ultimately collapses onto
\(\delta_0\).
It is neither the spiky stems nor the much broader observation
density \(\widehat Q_n*\phi\) that Controller~2 sees
(Figure~\ref{fig:solconv-qn},
broadened by a full unit of noise variance). Structure by symmetric
free-weight optimization; nonconvex, hence a best-found candidate
(cf.\ Remark~\ref{rem:atomic-vs-ac}).}
\label{fig:solconv-q}
\end{figure}

\begin{figure}[t]
\centering
\includegraphics[width=\linewidth]{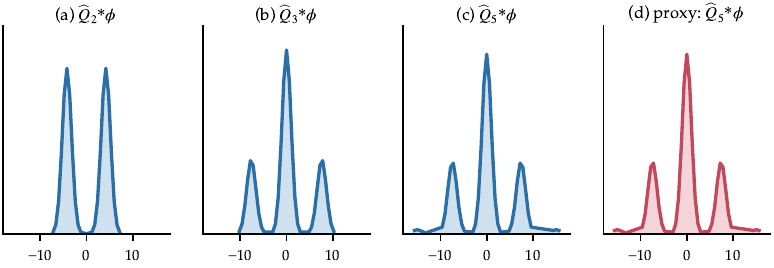}
\caption{The \emph{observation} density \(\widehat Q_n*\phi\) that
Controller~2 sees,
companion to Figure~\ref{fig:solconv-q} (the law \(\widehat Q_n\) blurred by
the
unit-Gaussian noise \(\phi\)), at fixed \(k=0.1\) and \(n=2,3,5\), with
\emph{(d)} the finite-level proxy at \(n=5\) (shaded). This is \emph{not} the
law density of
Figure~\ref{fig:density}: convolving with \(\phi\) adds a full unit of
variance, so \(\widehat Q_n*\phi\) here is \emph{much broader} than
\(\widehat q\), whose
modes are far narrower (Figure~\ref{fig:solconv-q}(d)); this
panel is the channel output \(p_Z\), not the law itself. Unlike the atomic
\(\widehat Q_n\), these blurred versions are genuine
densities; they converge as \(n\) grows, with \(\widehat Q_3*\phi\) and
\(\widehat Q_5*\phi\) already indistinguishable, the multi-lump
profile Controller~2 must resolve.}
\label{fig:solconv-qn}
\end{figure}

We repeat the experiment in the \emph{expensive-control} regime
\(k=0.5\), below the stability threshold \(k_{\mathrm{stab}}\approx0.735\)
(\S\ref{sec:hermite}), to contrast with the cheap-control pictures
above. Now signalling is costly: the computed levels sit closer together,
the cost is far higher (\(\widehat J_9\approx0.70\) versus \(0.053\)), and the
finite-level atoms do \emph{not} lock onto a few sharp modes but
\emph{proliferate and fill in}, approximating the smooth, absolutely
continuous \(Q^\star\) of Proposition~\ref{prop:ac-gauss}; convergence in \(n\) is
correspondingly slower (\(\widehat J_n=2.27,1.23,0.76\) for \(n=2,3,5\), reaching
\(0.70\) only near \(n=9\)).

\begin{figure}[t]
\centering
\includegraphics[width=\linewidth]{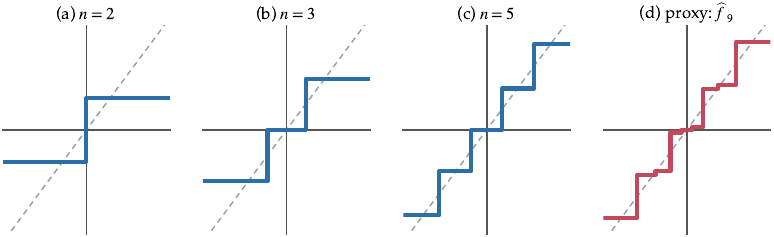}
\caption{The controllers \(\widehat f_n=F_{\widehat Q_n}^{-1}\circ F_P\) at
\(k=0.5\) (\(\sigma=5\)), the counterpart of Figure~\ref{fig:solconv-f}
(\(k=0.1\)) in the expensive-control regime. The steps are smaller and
closer (levels \(\pm5.2,\pm10.7\) versus \(\pm7.4,\pm15\)) and additional
levels stay active, so the staircase refines more gradually; \emph{(d)} is
the finite-level proxy at \(n=9\) (\(\widehat J_9\approx0.70\)), still a
staircase and so still distinct from the strictly increasing \(f^\star\); the
identity (gray) is \(U_1\equiv0\).}
\label{fig:solconv-f5}
\end{figure}

\begin{figure}[t]
\centering
\includegraphics[width=\linewidth]{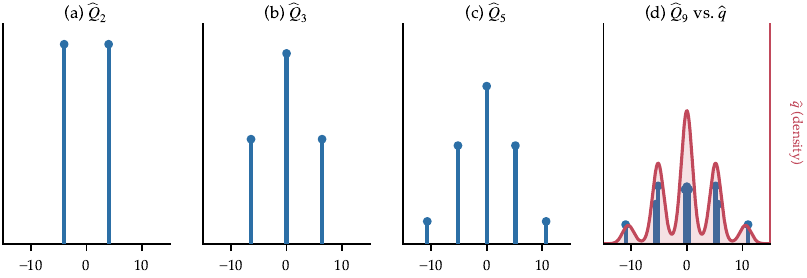}
\caption{The laws at \(k=0.5\), the counterpart of
Figure~\ref{fig:solconv-q} (same convention: blue stems are the point
\emph{masses} of the computed finite-level atomic \emph{measures}
\(\widehat Q_n\), left axis;
the crimson curve is the \emph{density} \(\widehat q\) of the computed
continuous solution, right axis). In contrast to the cheap-control case the
atoms do not settle at a few well-separated modes but \emph{proliferate and
cluster} as \(n\) grows. Panel~\emph{(d)} plots the atomless
\(\widehat Q_{\mathrm{cont}}\), drawn as its reconstructed density
\(\widehat q\) (crimson; same
solve as
Figure~\ref{fig:density}(d)), against the atomic \(\widehat Q_9\): here, well
below
\(k_{\mathrm{stab}}\), the computed law is \emph{broad and multimodal} with a
visible floor between the modes --- which, being a kernel reconstruction
(\S\ref{app:scope}), is consistent with but not evidence for connected
support --- so
the atomic law is a cruder stand-in than at \(k=0.1\) and the finite-level
approximation must proliferate to track it --- as it must, since by Proposition~\ref{prop:ac}
the true \(Q^\star\) is atomless. This is the same object,
sharpening as \(k\) decreases through the signalling window (but not
indefinitely: Proposition~\ref{prop:asymptotics}(b)); the noise-blurred
observation densities \(\widehat Q_n*\phi\) of its finite-level
approximations are shown in Figure~\ref{fig:solconv-qn5}.}
\label{fig:solconv-q5}
\end{figure}

\begin{figure}[t]
\centering
\includegraphics[width=\linewidth]{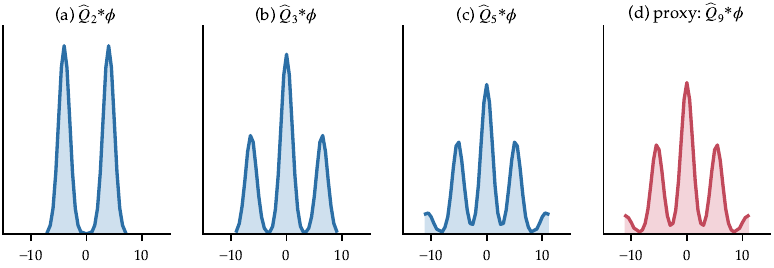}
\caption{The \emph{observation} densities \(\widehat Q_n*\phi\) at \(k=0.5\),
the
counterpart of Figure~\ref{fig:solconv-qn} (again the law blurred by the
unit-Gaussian noise, \emph{not} the law density of
Figure~\ref{fig:density}). They are broader and their lumps closer and less
resolved than at \(k=0.1\): expensive control cannot separate the signal
levels as far. They settle down as \(n\) grows, \emph{(d)} being the
smooth finite-level proxy at \(n=9\) (shaded). The small outer lumps near
\(\pm11\) are the finite-level approximation's atoms at \(\pm10.7\) smoothed
by \(\phi\); they carry little mass and shrink under the continuous solver,
consistent with the nonconvex, solver-dependent mode structure noted in
Remark~\ref{rem:atomic-vs-ac}.}
\label{fig:solconv-qn5}
\end{figure}

The three objects tracked by
Figures~\ref{fig:solconv-f5}--\ref{fig:solconv-qn5} are worth collecting in one
place. For each level count \(n\), \(\widehat Q_n=\sum_{i=1}^n
w_i\delta_{y_i}\) is the
computed \(n\)-atom \emph{law}, \(\widehat f_n=F_{\widehat Q_n}^{-1}\circ F_P\)
is the
monotone \emph{controller}
that realizes it (a staircase with \(n\) steps, the transport map pushing \(P\)
onto \(\widehat Q_n\)), and \(\widehat J_n=J(\widehat Q_n)\) is its
\emph{cost}.
Proposition~\ref{prop:finite-level} gives \(J_n^\star\downarrow J^\star\);
the computed values \(\widehat J_n\) are evaluated feasible values
(type~(N2) of \S\ref{sec:prelim}), hence upper estimates of \(J_n^\star\)
and of \(J^\star\) up to quadrature error.
Figure~\ref{fig:conv3d} shows the
three converging together as \(n\) grows from \(1\) to \(9\) at \(k=0.5\): the
controller \emph{(a)} gains steps, the law \emph{(c)} gains atoms that fill in
(their masses \(w_i\) the stem heights), the observation density \emph{(b)}
develops more lumps, and the cost \emph{(d)} falls monotonically, reaching
\(\widehat J_9\approx0.70\) (an evaluated feasible value; the refined
continuous \(\widehat J(0.5)\approx0.65\) of Figure~\ref{fig:density} is
lower). Convergence appears quick in cost: \(\widehat J_n\) is within
a few percent of the best value we find by \(n\approx6\), while the \emph{law}
keeps refining by adding ever-smaller outer atoms, consistent with an
atomless limit.

\begin{figure}[!t]
\centering
\includegraphics[width=\linewidth]{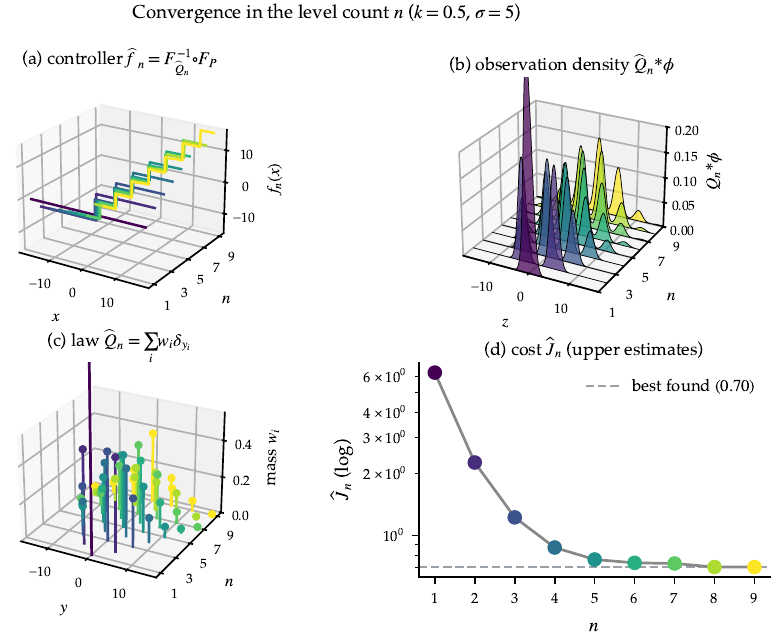}
\caption{Convergence in the level count \(n\) at \(k=0.5\) (\(\sigma=5\)): the
controller \(\widehat f_n\), the law \(\widehat Q_n\), the observation density
\(\widehat Q_n*\phi\), and
the cost \(\widehat J_n\), in one view. \emph{(a)} \(\widehat f_n\) is a
staircase that
gains steps; \emph{(b)} \(\widehat Q_n*\phi\) develops more lumps; \emph{(c)}
\(\widehat Q_n=\sum_i w_i\delta_{y_i}\) gains atoms that proliferate and fill
in (stem
heights are the masses \(w_i\)); \emph{(d)} \(\widehat J_n\) decreases as a
sequence of numerical upper estimates --- evaluated feasible values, type~(N2)
of \S\ref{sec:prelim}; the \emph{upper bound} is the exact cost
\(J(\widehat Q_n)\) they evaluate --- reaching \(\widehat J_9\approx0.70\) by
\(n=9\), on a log scale. That decrease is \emph{imposed} by taking the running
minimum over smaller level counts, and is not itself convergence to
\(J^\star\); Proposition~\ref{prop:finite-level} separately
proves \(J_n^\star\downarrow J^\star\) for the \emph{exact} finite-level
values, which is a different statement.
Each panel shows a computed free-weight finite-level candidate from the
generalized
Lloyd iteration (ancillary \texttt{precompute\_convergence.py}).
Depth axis
is \(n=1,\dots,9\).}
\label{fig:conv3d}
\end{figure}

\begin{remark}[The finite-level solution is atomic; the true optimum is not]
\label{rem:atomic-vs-ac}
The laws \(\widehat Q_n\) computed above (Figure~\ref{fig:fn}) are atomic
\emph{by
construction}: they are computed for the restricted problem over measures
with at most \(n\) atoms, not the original
problem~\eqref{eq:OT-variational}.
This is entirely consistent with Proposition~\ref{prop:ac}, which forbids
atoms in the true minimizer \(Q^\star\). Since no atomic law is optimal,
the finite-level values stay above the optimum, \(J_n^\star\ge J^\star\),
approaching it only as \(n\to\infty\). Moreover every weak limit point of a
\emph{sequence} \((Q_n^\star)\) of exact finite-level minimizers is an
\emph{atomless} minimizer of~\eqref{eq:OT-variational}: such a sequence has
uniformly bounded second moments (by the coercivity of Step~2 of
Proposition~\ref{prop:existence}, since \(J(Q_n^\star)\le J_1^\star=J(\delta_0)=k^2\sigma^2\)), hence is
tight; along a weakly convergent subsequence \(Q_{n_j}^\star\rightharpoonup
\nu\), lower semicontinuity of the transport term and continuity of
\(\operatorname{mmse}\) (Steps~4--5 of Proposition~\ref{prop:existence}) give
\(J(\nu)\le\liminf_jJ_{n_j}^\star=J^\star\), so \(\nu\) is a minimizer, and it
is atomless by Proposition~\ref{prop:ac}. No individual \(\widehat Q_n\) is
itself a limit, and none is claimed to equal \(Q^\star\). The no-atoms proof of
Proposition~\ref{prop:ac} suggests why the approximation is nonetheless
so accurate: spreading an atom lowers the cost by a first-order amount
\(\sim k^2\kappa\,s\) in transport against only \(\sim s^2\) in
estimation, so the spread that is profitable \emph{starting from an atomic
configuration} has scale \(s=O(k^2)\). We stress the limits of that
inference. It is a statement about how far a hypothetical atom would be worth
smearing, not a proof that the true minimizer has modes of width
\(\asymp k^2\), which we do not establish; the argument bounds a profitable
perturbation scale, not the geometry of \(q^\star\). What it does make
plausible --- and what the computations of \S\ref{sec:numerics} then display
for the \emph{computed candidates} --- is that at moderately small
\(k\) the recovered law is spread out yet \emph{sharply
peaked} near a few levels. (That \(Q^\star\) itself is absolutely continuous
is proved, Proposition~\ref{prop:ac-gauss}; that its peaks have the shape the
computations show is not.) The atoms of \(\widehat Q_n\) act as discrete
stand-ins
for narrow density peaks, and a handful of them
already captures nearly all of the cost, which is why \(\widehat J_n\)
stabilizes numerically (Example~\ref{ex:convergence}) well before the
atomless optimum is reached. Figure~\ref{fig:density} shows this
directly: solving a discretized transport-map approximation
of~\eqref{eq:OT-variational} by gradient descent (Appendix~\ref{app:solver})
returns a spread-out map, whose push-forward reconstructs as a smooth density
\(\widehat q\) rather than a sum of
atoms, at each \(k\in\{0.1,0.3,0.4,0.5\}\). In particular panel~(a) is the
computed solution at the very same \(k=0.1\) whose atomic proxy
\(\widehat Q_5\)
appears in Figure~\ref{fig:solconv-q}, where the two are overlaid: tall
narrow peaks over a low floor, not point masses.

Two clarifications are in order. First, the convergence
\(Q_n^\star\rightharpoonup Q^\star\) is a statement about \emph{exact}
finite-level minimizers, holds at \emph{fixed} \(k\), is \emph{weak}, and
(absent uniqueness of \(Q^\star\)) is a statement about limit points. Weak
convergence alone does not say \emph{where} the atoms go; what the
computations show is that they cluster near the peaks of the recovered
density rather than spreading out, and we report that as an observation, not
a theorem. At \(k=0.1\) the computed modes are so sharp that
\(\widehat Q_n\) is already
an excellent proxy;
at larger \(k\) they broaden into the visible, positive-floor density of
Figure~\ref{fig:density}, and the atoms must proliferate to track it
(Figure~\ref{fig:solconv-q5}). Second, what the solver returns is a
\(\widehat Q_n\), not a \(Q_n^\star\): the \emph{number} of modes is
delicate because the problem is nonconvex
(Proposition~\ref{prop:gauss-foc}), so different numerical schemes may
return different stationary points, and the densities in
Figure~\ref{fig:density} should be read as best-found computed candidates
rather than as the certified global structure.
\end{remark}

\begin{figure}[t]
\centering
\includegraphics[width=\linewidth]{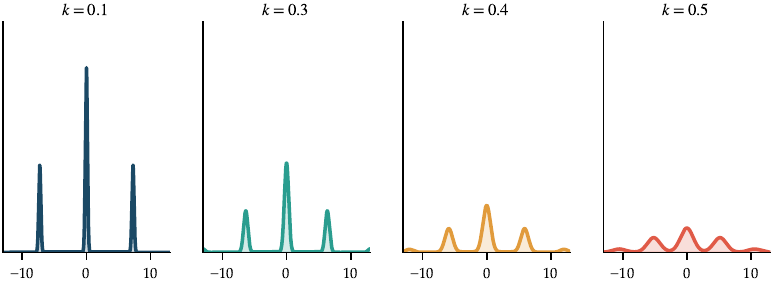}
\caption{A smooth kernel reconstruction, not a set of
atoms: the reconstructed densities \(\widehat q\) (\(\sigma=5\)) at
\(k=0.1,0.3,0.4,0.5\), on a common vertical scale. These are computed by
solving a \emph{discretized} transport-map approximation of the continuous
problem (\(Y=T(X)\), so \(\widehat Q=\widehat T_\#P\)), using the exact
gradient of the discretized objective, seeded at each \(k\) from the
finite-level constellation at that \(k\). This includes \(k=0.1\) \emph{(a)}: a direct
transport-map solve, the smooth object that the atomic stems of
Figure~\ref{fig:solconv-q} approximate, and at that small \(k\) it is
sharply peaked (though the rendered widths are bandwidth-limited,
\S\ref{app:scope}) yet still atomless, tall narrow bumps
over a low but nonzero floor rather than the point masses the finite-level
solver returns. (These curves are a Gaussian-kernel reconstruction of the computed
map, \S\ref{app:scope}; a kernel estimate is positive by construction, so the
floor visible between modes is a feature of the reconstruction and is
\emph{not} independent evidence for the positivity statement of
Proposition~\ref{prop:ac-gauss} --- it is merely consistent with it. Peak
widths are bounded below by the bandwidth quoted there.) The resulting costs \(\widehat J(k)=0.053,0.329,0.494,0.645\) sit at or
below the finite-level values, matching them at \(k=0.1\) and matching
the value \(0.33\) at \(k=0.3\) (Figure~\ref{fig:keffect}); at the
larger \(k\) the transport-map solve returns lower values than the coarser
sweep of Figure~\ref{fig:keffect}, as expected of a
nonconvex objective with many low-cost candidates. These four numbers are
computed on the discretized prior and are therefore \emph{estimates}, not
bounds (type~(N3) of \S\ref{sec:prelim}); the accompanying feasible-law
values, which do bound \(J^\star\) up to quadrature error, are
\(0.0532,0.3298,0.4953,0.6461\) (\S\ref{app:validation}(ii)). None of them
is a certified optimum. Each curve is a kernel reconstruction of the push-forward density of
\(\widehat T\) on a truncated, renormalized grid (\S\ref{app:scope}), not
the exact density of \(\widehat Q_{\mathrm{cont}}\). (The \emph{observation}
densities of Figures~\ref{fig:solconv-qn} and~\ref{fig:solconv-qn5} are
the finite-level \(\widehat Q_n*\phi\), not this curve convolved with the
noise.) The true \(Q^\star\) is absolutely continuous
(Proposition~\ref{prop:ac-gauss}); the curves here are a smooth reconstruction
\emph{consistent with} such an atomless solution, and are not themselves
evidence that the numerical optimizer is smooth or absolutely continuous --- a
kernel estimate would look smooth regardless. What they do show is that the
recovered structure is concentrated, its peaks narrow on the rendered scale.
The structure is a
consistent five levels \(\{0,\pm\text{side},\pm\text{outer}\}\) at every
\(k\); as control cheapens (\(k:0.5\to0.1\)) the modes both \emph{sharpen}
and \emph{separate} (side modes \(\pm5.2\to\pm7.3\)), approaching the
signalling profile of the finite-level \(\widehat Q_n\)
(Figure~\ref{fig:fn}).}
\label{fig:density}
\end{figure}

\begin{figure}[t]
\centering
\includegraphics[width=\linewidth]{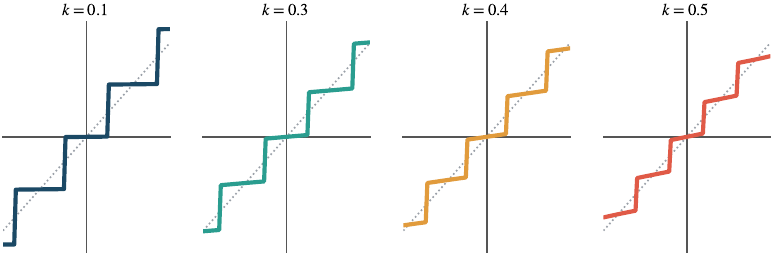}
\caption{The computed first controllers
\(\widehat T=F_{\widehat Q}^{-1}\circ F_P\)
realizing the four laws of Figure~\ref{fig:density} (\(\sigma=5\)),
\(k=0.1,0.3,0.4,0.5\); axes are \(x\) horizontal and \(\widehat T(x)\)
vertical, identity dotted (\(U_1\equiv0\)). These are the \emph{same}
computed transport maps that define \(\widehat Q=\widehat T_\#P\) in
Figure~\ref{fig:density} (no separate calculation). Signalling strength
tracks \(k\): at \(k=0.1\) \emph{(a)} a sharply layered, staircase-like profile
(near-plateaus at
\(0,\pm7.3,\pm14.9\) joined by steep risers, not literally flat --- the true
\(f^\star\) is strictly increasing by Fact~\ref{fact:regularity}, and the
computed maps are increasing at every grid node for \(k\le0.4\), with six
far-tail exceptions at \(k=0.5\) that are quantified in \S\ref{app:opt}),
softening and shortening as
control grows expensive until at \(k=0.5\) \emph{(d)}, just below
\(k_{\mathrm{stab}}\approx0.735\), the levels are closer
(\(\pm5.2,\pm10.5\)) and the map is
a milder deformation of the identity. The plateau widths are the mode
masses of Figure~\ref{fig:density}.}
\label{fig:fdensity}
\end{figure}

\subsection{The effect of the control penalty \(k\)}

Finally we vary \(k\), solving the discretized transport-map approximation of
the continuous problem at each \(k\) from a
signalling-seeded start (\(\sigma=5\)) and comparing with the Gaussian-class optimum
\(J_{\mathrm{G}}\) of \S\ref{sec:gaussian-class}.

\begin{example}[Effect of \(k\)]
\label{ex:keffect}
Solving at each \(k\) from a signalling-seeded start, the computed
cost \(\widehat J(k)\), the Gaussian-class cost, and the number of modes of
the computed law \(\widehat Q_{\mathrm{cont}}\) are
\[
\begin{array}{c|cccccccc}
k & 0.1 & 0.15 & 0.2 & 0.3 & 0.4 & 0.5 & 0.55 & 0.6\\[2pt]\hline
\widehat J(k) & 0.06 & 0.11 & 0.22 & 0.33 & 0.53 & 0.79 & 0.94 & 0.96\\
J_{\mathrm{G}} & 0.25 & 0.55 & 0.96 & 0.96 & 0.96 & 0.96 & 0.96 & 0.96\\
\#\,\text{modes of }\widehat Q & 5 & 5 & 5 & 3 & 3 & 3 & 3 & 1
\end{array}
\]
This sweep uses a coarser grid and only five-level seeds, so its entries
can sit above the refined solves: at \(k=0.1\) it returns \(0.057\)
(printed as \(0.06\)), against \(\widehat J\approx0.053\) from the
finite-level program (Example~\ref{ex:convergence}) and from the refined
transport-map solve of Figure~\ref{fig:density}.
Every entry of the first row is a computed candidate value of type~(N3)
(\S\ref{sec:prelim}): it is the objective of the \emph{discretized}
transport-map program at a best-found point, so it is neither a certified
optimum nor by itself an upper bound on \(J^\star(k)\). The accompanying
feasible-law values, which do bound \(J^\star\) up to quadrature error, are
tabulated in \S\ref{app:validation}(ii); and
Remark~\ref{rem:kc-numerics} shows explicitly that the last two entries are
far from optimal.
The Gaussian benchmark itself is not constant: for \(k\lesssim0.2\) the
best \emph{linear} controller shrinks \(Y\) toward \(0\)
(\(J_{\mathrm{G}}\approx k^2\sigma^2\), Corollary~\ref{cor:gauss-asymp}),
while for larger \(k\) it leaves \(Y\approx X_0\)
(\(J_{\mathrm{G}}\approx\operatorname{mmse}(P)\approx0.96\)). The computed
multi-modal \emph{signalling} controller beats it by roughly a factor of
four at every small \(k\), and the gap \(J_{\mathrm{G}}-\widehat J\) closes,
in
\emph{this sweep}, near \(k\approx0.56\)
(Figure~\ref{fig:keffect}), where the dominant modes of
\(\widehat Q_{\mathrm{cont}}\) collapse
to one and the computed controller becomes affine. This is the quantitative
face of the whole development: signalling nonlinearity pays strongly
throughout the moderate cheap-control regime. (It does \emph{not} pay in the
limit \(k\downarrow0\), where by Proposition~\ref{prop:asymptotics}(b) the
optimizers collapse onto full cancellation --- see
Remark~\ref{rem:crossover}.) Notably the computed candidate stays nonlinear \emph{beyond} the
Gaussian-class
threshold \(k^2=\tfrac14\) (\(k=0.5\)): \(k_{\mathrm{stab}}\) is set by a
Hermite-mode instability of the full problem, not of the Gaussian slice
(\S\ref{sec:hermite}). The value \(0.56\) at which this sweep's gap
closes is, however, \emph{not} \(k_{\mathrm{stab}}\), still less a
linear-optimality threshold: by
Proposition~\ref{prop:kstab} the Gaussian is still unstable at \(k=0.6\), and the
last two columns of the table above are therefore loose values that a
richer search improves; see Remark~\ref{rem:kc-numerics}. The
more
careful transport-map solve of Figure~\ref{fig:density} refines the four
\(k\) it displays to slightly lower values (\(0.49\) and \(0.65\) at
\(k=0.4,0.5\), below the \(0.53\) and \(0.79\) tabulated), and it resolves
five levels throughout the signalling regime (the outer pair carrying little
mass), so both the costs and the mode counts in this table are coarse relative
to it, the counts tracking only the dominant modes. The mode counts are in any
case approximate: the nonconvexity makes the exact mode transitions
solver-dependent. What is robust across the computed signalling branch is the
trend toward fewer dominant modes as \(k\) increases; the apparent
single-mode collapse near \(k\approx0.56\) is solver-dependent and should
\emph{not} be identified with \(k_{\mathrm{stab}}\) or with a true optimal
mode transition. Indeed Remark~\ref{rem:kc-numerics} exhibits nonlinear
competitors that beat the affine policy at \(k=0.6\) and \(k=0.7\), i.e.\
beyond this crossover and below \(k_{\mathrm{stab}}\approx0.735\).
\end{example}

\begin{figure}[!ht]
\centering
\includegraphics[width=0.68\linewidth]{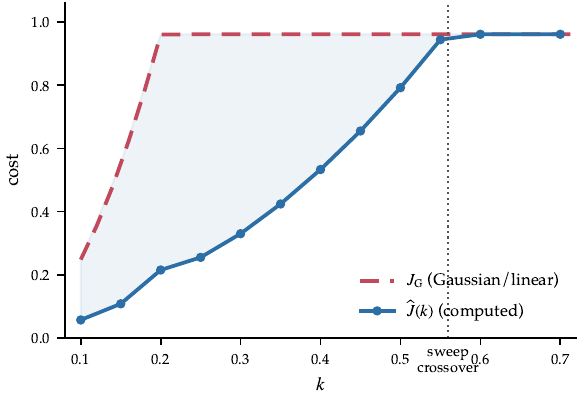}
\caption{Effect of the control penalty \(k\) (\(\sigma=5\)). The computed
cost \(\widehat J(k)\) (solid, a continuous solve at each \(k\) from a
signalling-seeded start, kept as the best of several --- the discretized
objective at a best-found point, neither a certified optimum nor by itself a
bound on \(J^\star(k)\); cf.\ \S\ref{app:validation}(ii)) lies well below
the Gaussian/linear benchmark
\(J_{\mathrm{G}}\) (dashed), about a factor of four at the small-\(k\) end of
this sweep, and
rises to meet it near \(k\approx0.56\) in this sweep (\emph{not}
\(k_{\mathrm{stab}}\), which Proposition~\ref{prop:kstab} places at
\(0.735\), and not a linear-optimality threshold;
see Remark~\ref{rem:kc-numerics}),
where \(\widehat Q_{\mathrm{cont}}\) collapses to a single-mode Gaussian. Note
\(J_{\mathrm{G}}\) is itself small for \(k\lesssim0.2\) (the best linear
controller then shrinks \(Y\) toward \(0\)) and saturates at
\(\operatorname{mmse}(P)\approx0.96\) for larger \(k\). The computed nonlinear
solution persists beyond the Gaussian-class threshold \(k^2=\tfrac14\)
(\(k=0.5\); \S\ref{sec:hermite}). The number of modes of
\(\widehat Q_{\mathrm{cont}}\)
grows as \(k\) decreases across this sweep, reaching three around
\(k\approx0.4\) to \(0.5\) and five at the smallest \(k\) shown
(cf.\ Figure~\ref{fig:density}). At the other end the
signalling-branch computation becomes single-mode near its own solver
crossover, \(k\approx0.56\); the true optimal mode structure near and above
\(k_{\mathrm{stab}}\) is \emph{not} determined here --- the Hermite analysis
gives second-variation instability of one Gaussian branch below
\(k_{\mathrm{stab}}\), not global affine optimality above it
(Remarks~\ref{rem:hermite} and~\ref{rem:kc-numerics}). We do not annotate
these counts on the
curve because the problem is nonconvex, so the exact mode transitions are
solver-dependent and only the trend (more, sharper, more widely separated
modes as control cheapens through the sweep) is robust.}
\label{fig:keffect}
\end{figure}

\begin{remark}[The computed threshold is a solver artefact, not \(k_{\mathrm{stab}}\)]
\label{rem:kc-numerics}
The sweep above stops improving on the affine policy near \(k\approx0.56\),
and it is tempting to read that off as the value beyond which affine control
is optimal. It is
not. Proposition~\ref{prop:kstab} places \(k_{\mathrm{stab}}\approx0.735\) at
\(\sigma=5\),
and an explicit competitor confirms that the affine policy is beaten well
past \(0.56\): taking the rippled transport map
\[
T(x)=\frac{\tau^\star}{\sigma}\,x+a\sin\!\Bigl(\frac{2\pi x}{d}\Bigr),
\qquad Q=T_\#P ,
\]
with the wavelength \(d\) near the critical \(\pi\sqrt2\approx4.44\) predicted
by the mode analysis, gives at \(\sigma=5\)
\[
\begin{array}{c|ccc}
k & 0.6 & 0.7 & 0.75\\[2pt]\hline
J_{\mathrm G}(k) & 0.96138 & 0.96143 & 0.96144\\
J(T_\#P) & 0.94536\ (a{=}0.64,\ d{=}4.6) & 0.96065\ (a{=}0.28,\ d{=}4.4)
& \text{no gain}
\end{array}
\]
so \(J^\star(0.6)\le0.9454<J_{\mathrm G}(0.6)\) and
\(J^\star(0.7)<J_{\mathrm G}(0.7)\): the controller is still nonlinear at both.
The crossover of this one-parameter family lies between \(0.70\) and
\(0.75\), and the wavelength that wins at threshold is \(\approx4.4\), both
matching Proposition~\ref{prop:kstab}. The reason the sweep misses it is
structural rather than numerical: the solver seeds from few-level signalling
constellations (\S\ref{app:warm}), a family that cannot represent a
small-amplitude ripple of period \(\approx4.4\) riding on a nearly affine
map, and near \(k_{\mathrm{stab}}\) the available gain is only second order in the
amplitude. Accordingly the entries of Example~\ref{ex:keffect} for
\(k\gtrsim0.5\), the mode counts there, and the ``collapse to a single mode''
in Figure~\ref{fig:keffect} should be read as what the signalling-branch
solver returns, not as the certified optimum; recomputing that range with
ripple-seeded starts is the natural fix.
\end{remark}

The two ends of Figure~\ref{fig:keffect} admit exact statements. When
control is expensive the optimizer does almost nothing and the cost
saturates at the prior MMSE. When control is cheap the cost collapses to
zero, and --- perhaps surprisingly, given the signalling constellations the
numerics return at \(k\approx0.1\) --- the \emph{leading-order} behaviour is
governed not by signalling but by full cancellation. Both are made
precise below.

\begin{proposition}[Asymptotic regimes]
\label{prop:asymptotics}
Let \(P=\mathcal N(0,\sigma^2)\) and
\(J^\star(k)=\inf_{Q}\{k^2W_2^2(P,Q)+\operatorname{mmse}(Q)\}\).
\begin{enumerate}
\item[\textup{(a)}] \emph{Expensive control, \(k\to\infty\).}
\(J^\star(k)\uparrow\operatorname{mmse}(P)=\dfrac{\sigma^2}{1+\sigma^2}\),
and every minimizer obeys \(W_2^2(P,Q^\star_k)=O(k^{-2})\), so
\(Q^\star_k\to P\) and the first controller tends to the identity
(\(U_1\equiv0\)). Within the Gaussian class the minimizer is unique for
\(k^2>\tfrac14\), with \(\tau^\star=\sigma-O(k^{-2})\), and the Gaussian
competitor forces the gap to close no faster than \(k^{-2}\):
\[
\operatorname{mmse}(P)-J^\star(k)\ \ge\
\operatorname{mmse}(P)-J_{\mathrm G}(k)\ =\ \Theta(k^{-2})
\qquad(\text{Corollary~\ref{cor:gauss-asymp}}).
\]
\item[\textup{(b)}] \emph{Cheap control, \(k\to0\).} The leading order is
exact:
\begin{equation}
J^\star(k)=\sigma^2k^2+o(k^2),
\qquad k\downarrow0 ,
\label{eq:small-k-sharp}
\end{equation}
and the minimizers converge weakly to full cancellation,
\(Q^\star_k\rightharpoonup\delta_0\), with
\(W_2^2(P,Q^\star_k)\to\sigma^2\) and
\(\operatorname{mmse}(Q^\star_k)=o(k^2)\). Both bounds are two-sided: the
upper bound is attained by the single feasible law \(\delta_0\), for which
\(J(\delta_0)=k^2\sigma^2\) exactly, and the lower bound is proved below.
\end{enumerate}
\end{proposition}

\begin{proof}
\emph{(a)} The law \(Q=P\) is feasible with
\(J(P)=\operatorname{mmse}(P)\), so \(J^\star(k)\le\operatorname{mmse}(P)\)
for every \(k\); and \(J^\star\) is nondecreasing in \(k\), so
\(L:=\lim_{k\to\infty}J^\star(k)\le\operatorname{mmse}(P)\) exists. Let
\(Q^\star_k\) attain \(J^\star(k)\). From
\(k^2W_2^2(P,Q^\star_k)\le J^\star(k)\le\operatorname{mmse}(P)\) we get
\(W_2^2(P,Q^\star_k)\le\operatorname{mmse}(P)/k^2\to0\), i.e.\
\(Q^\star_k\to P\) in \(W_2\) (and the transport plan tends to the
identity, \(U_1\equiv0\)). The second moments stay bounded, so by
Fact~\ref{fact:mmse} \(\operatorname{mmse}(Q^\star_k)\to
\operatorname{mmse}(P)\); hence
\(L\ge\liminf_k\operatorname{mmse}(Q^\star_k)=\operatorname{mmse}(P)\), and
\(L=\operatorname{mmse}(P)\). For the rate, the Gaussian family already
supplies a competitor: by Corollary~\ref{cor:gauss-asymp},
\(\tau^\star=\sigma-O(k^{-2})\) gives
\(\operatorname{mmse}(P)-J_{\mathrm G}(k)=\Theta(k^{-2})\), and
\(J^\star\le J_{\mathrm G}\) yields
\(\operatorname{mmse}(P)-J^\star(k)\ge\operatorname{mmse}(P)-J_{\mathrm
G}(k)=\Theta(k^{-2})\).

\emph{(b)} \emph{Upper bound.} The Dirac \(\delta_0\) is feasible, and its
cost is exact: the only coupling of \(P\) with \(\delta_0\) is the product
one, so \(W_2^2(P,\delta_0)=\int x^2dP=\sigma^2\), while
\(\operatorname{mmse}(\delta_0)=0\) because a constant \(Y\) has zero
conditional variance. Hence
\begin{equation}
J^\star(k)\ \le\ J(\delta_0)\ =\ k^2\sigma^2
\qquad\text{for every }k>0 .
\label{eq:delta0-bound}
\end{equation}

\emph{The minimizers collapse to \(\delta_0\).} Let \(Q^\star_k\) attain
\(J^\star(k)\) (Proposition~\ref{prop:existence}). Since both terms of \(J\)
are nonnegative, \eqref{eq:delta0-bound} splits into
\begin{equation}
W_2^2(P,Q^\star_k)\ \le\ \sigma^2,
\qquad
\operatorname{mmse}(Q^\star_k)\ \le\ k^2\sigma^2\ \longrightarrow\ 0 .
\label{eq:small-k-splits}
\end{equation}
The first bound gives uniformly bounded second moments, exactly as in Step~2
of Proposition~\ref{prop:existence}:
\(W_2(Q_k^\star,\delta_0)\le W_2(Q^\star_k,P)+W_2(P,\delta_0)\le2\sigma\), so
\(\int y^2dQ^\star_k\le4\sigma^2\) and the family \((Q^\star_k)_{k\le1}\) is
tight. Let \(Q_0\) be any weak limit point, say
\(Q^\star_{k_j}\rightharpoonup Q_0\) with \(k_j\downarrow0\). By
Fact~\ref{fact:mmse} the estimation functional is weakly continuous, so the
second bound in~\eqref{eq:small-k-splits} forces
\[
\operatorname{mmse}(Q_0)=\lim_j\operatorname{mmse}(Q^\star_{k_j})=0 .
\]
But \(\operatorname{mmse}(Q)=0\) only for Diracs. Indeed, if \(Q\) is not a
Dirac there are disjoint closed sets \(A,B\) at positive distance with
\(Q(A),Q(B)>0\); since the Gaussian kernel is strictly positive, for
\emph{every} \(z\) the posterior \(Q(\,\cdot\mid Z=z)\) charges both \(A\) and
\(B\), so \(\operatorname{Var}(Y\mid Z=z)>0\) for every \(z\) and
\(\operatorname{mmse}(Q)=\int\operatorname{Var}(Y\mid Z=z)p_Z(z)\,dz>0\).
(Equivalently, via Proposition~\ref{prop:mmse-fisher}, \(\operatorname{mmse}
=1-\mathcal I(Q*\phi)\) vanishes only in the equality case
\(\mathcal I(Q*\phi)=\mathcal I(\phi)\) of the Stam bound of
Corollary~\ref{cor:fisher-bounds}.) So \(Q_0=\delta_m\) for some \(m\).
Weak lower semicontinuity of \(W_2^2(P,\cdot)\) (Step~4 of
Proposition~\ref{prop:existence}) and the first bound
in~\eqref{eq:small-k-splits} now pin down \(m\):
\[
\sigma^2+m^2=\int(x-m)^2dP=W_2^2(P,\delta_m)
\ \le\ \liminf_j W_2^2(P,Q^\star_{k_j})\ \le\ \sigma^2 ,
\]
so \(m=0\). Every limit point being \(\delta_0\), the whole family converges,
\(Q^\star_k\rightharpoonup\delta_0\), and applying lower semicontinuity once
more along the full family gives
\(\liminf_{k\downarrow0}W_2^2(P,Q^\star_k)\ge W_2^2(P,\delta_0)=\sigma^2\),
which with~\eqref{eq:small-k-splits} yields
\(W_2^2(P,Q^\star_k)\to\sigma^2\).

\emph{Lower bound and conclusion.} Discarding the nonnegative estimation term,
\[
\frac{J^\star(k)}{k^2}\ \ge\ W_2^2(P,Q^\star_k)\ \longrightarrow\ \sigma^2 ,
\qquad\text{while}\qquad
\frac{J^\star(k)}{k^2}\ \le\ \sigma^2
\quad\text{by \eqref{eq:delta0-bound}} .
\]
Hence \(J^\star(k)/k^2\to\sigma^2\), which is~\eqref{eq:small-k-sharp}; and
\(\operatorname{mmse}(Q^\star_k)=J^\star(k)-k^2W_2^2(P,Q^\star_k)
=k^2(\sigma^2-\sigma^2)+o(k^2)=o(k^2)\).

\emph{The convergence is in \(W_2\), and the control cancels the state.}
Write \(u_k:=F_{Q^\star_k}^{-1}\in L^2(0,1)\), so that
\(\|u_k\|_{L^2}^2=\int y^2dQ^\star_k=\mathbb E[f^\star_k(X_0)^2]\) and, by
Fact~\ref{fact:1d-ot},
\[
W_2^2(P,Q^\star_k)=\|F_P^{-1}\|^2-2\langle F_P^{-1},u_k\rangle+\|u_k\|^2
=\sigma^2-2\langle F_P^{-1},u_k\rangle+\|u_k\|^2 .
\]
The family \((u_k)\) is bounded in \(L^2\) (second moments are bounded above)
and \(u_k\to0\) at every continuity point, because
\(Q^\star_k\rightharpoonup\delta_0\); a bounded sequence converging pointwise
a.e.\ converges weakly in \(L^2\), so \(\langle F_P^{-1},u_k\rangle\to0\).
Feeding this and \(W_2^2(P,Q^\star_k)\to\sigma^2\) into the display gives
\(\|u_k\|^2\to0\). Thus \(\int y^2dQ^\star_k\to0\): the optimal first
controller satisfies \(\mathbb E[f^\star_k(X_0)^2]\to0\), i.e.\
\(f^\star_k\to0\) in \(L^2(P)\) and \(U_1\to-X_0\), and
\(Q^\star_k\to\delta_0\) in \(W_2\), not merely weakly.
\end{proof}

\begin{remark}[The two-level signalling construction, and why it is not the
cheap-control asymptotics]
\label{rem:twolevel}
It is natural to expect the \(k\downarrow0\) limit to be a signalling
constellation with ever more widely separated levels, and the numerics of
\S\ref{sec:numerics} at \(k\approx0.1\) certainly look that way. The
symmetric two-level law
\(Q_M=\tfrac12(\delta_{-M}+\delta_M)\) makes the mechanism explicit, and we
record the computation because it is the cleanest illustration of the
transport-estimation trade-off. But it does \emph{not} give the asymptotics.

\emph{Transport.} By Fact~\ref{fact:1d-ot} the optimal map is the monotone
rearrangement, which here sends the lower half of the mass to \(-M\) and the
upper half to \(M\); as \(P\) is symmetric this is
\(\{x<0\}\mapsto-M\), \(\{x\ge0\}\mapsto M\), i.e.\ \(x\mapsto
M\operatorname{sign}(x)\). Hence
\[
\begin{aligned}
W_2^2(P,Q_M)&=\mathbb E\bigl[(X-M\operatorname{sign}(X))^2\bigr]
=\mathbb E\bigl[(|X|-M)^2\bigr]\\
&=\sigma^2-2M\,\mathbb E|X|+M^2
=\sigma^2-2M\sigma\sqrt{2/\pi}+M^2 ,
\end{aligned}
\]
using \(\mathbb E|X|=\sigma\sqrt{2/\pi}\) for \(X\sim\mathcal N(0,\sigma^2)\).
Discarding the (negative) middle term, \(W_2^2(P,Q_M)\le\sigma^2+M^2\).

\emph{Estimation.} With \(Y\in\{\pm M\}\) equally likely and \(Z=Y+N\), Bayes'
rule gives
\[
\Pr(Y=M\mid Z=z)=\frac{\phi(z-M)}{\phi(z-M)+\phi(z+M)}
=\frac{e^{Mz}}{e^{Mz}+e^{-Mz}}=\frac{1}{1+e^{-2Mz}},
\]
the Gaussian factors \(e^{-(z^2+M^2)/2}\) cancelling. A two-point variable
taking \(\pm M\) with probabilities \(p,1-p\) has variance
\(4M^2p(1-p)\), so
\[
\operatorname{Var}(Y\mid Z=z)=\frac{4M^2}{(1+e^{-2Mz})(1+e^{2Mz})}
=\frac{4M^2}{2+e^{2Mz}+e^{-2Mz}}
=\frac{4M^2}{4\cosh^2(Mz)}=M^2\operatorname{sech}^2(Mz),
\]
whence \(\operatorname{mmse}(Q_M)=\int
M^2\operatorname{sech}^2(Mz)\,p_Z(z)\,dz\) with
\(p_Z(z)=\tfrac12[\phi(z-M)+\phi(z+M)]\). We bound this integral by splitting
at \(|z|=M/2\), the two regions being controlled for opposite reasons.

\emph{(i) The middle, \(|z|\le M/2\).} Here \(z\) is far from both levels:
\(|z\mp M|\ge M/2\), so both Gaussians are small,
\(p_Z(z)\le\phi(M/2)\). Pulling this constant out and using
\(\int_{\mathbb R}\operatorname{sech}^2(Mz)\,dz=2/M\) (substitute
\(u=Mz\), \(\int\operatorname{sech}^2u\,du=2\)),
\[
\int_{|z|\le M/2}\!\!M^2\operatorname{sech}^2(Mz)\,p_Z(z)\,dz
\ \le\ M^2\phi(M/2)\cdot\frac2M=2M\phi(M/2).
\]
\emph{(ii) The tails, \(|z|>M/2\).} Here the posterior is nearly
deterministic, so the conditional variance is small: from
\(\operatorname{sech}(u)=2/(e^u+e^{-u})\le2e^{-|u|}\) we get
\(\operatorname{sech}^2(Mz)\le4e^{-2M|z|}\le4e^{-M^2}\), and since
\(\int p_Z=1\),
\[
\int_{|z|>M/2}\!\!M^2\operatorname{sech}^2(Mz)\,p_Z(z)\,dz\ \le\ 4M^2e^{-M^2}.
\]
Adding the two, and noting \(\phi(M/2)=(2\pi)^{-1/2}e^{-M^2/8}\) while
\(4M^2e^{-M^2}\le C'Me^{-M^2/8}\) for \(M\ge1\) (the tail term is
exponentially smaller),
\[
\operatorname{mmse}(Q_M)\ \le\ 2M\phi(M/2)+4M^2e^{-M^2}
\ \le\ C\,M\,e^{-M^2/8}
\qquad\text{for a universal }C\text{ and all }M\ge1 .
\]

\emph{Balance.} Combining the two estimates,
\[
J(Q_M)\ \le\ k^2(\sigma^2+M^2)+C\,M\,e^{-M^2/8}.
\]
The transport term grows in \(M\) and the estimation term decays; choosing
\(M=M(k)=4\sqrt{\log(1/k)}\) gives \(M^2/8=2\log(1/k)\), hence
\(e^{-M^2/8}=k^2\), and the estimation term becomes
\(C\,M\,k^2=O\bigl(k^2\sqrt{\log(1/k)}\bigr)\), while the transport term is
\(k^2\sigma^2+16k^2\log(1/k)\). The \(\log\) term dominates its own
\(\sqrt{\log}\) companion, so
\begin{equation}
J\bigl(Q_{M(k)}\bigr)
=16k^2\log(1/k)+O\bigl(k^2\sqrt{\log(1/k)}\bigr)
=O\bigl(k^2\log(1/k)\bigr) .
\label{eq:twolevel-rate}
\end{equation}

\emph{Why this is the wrong asymptotic construction.} The bound
\eqref{eq:twolevel-rate} is correct but \emph{weaker than the trivial one}:
by \eqref{eq:delta0-bound} the single law \(\delta_0\) already gives
\(k^2\sigma^2\), and
\[
k^2\sigma^2=o\bigl(k^2\log(1/k)\bigr) .
\]
So the two-level construction is asymptotically dominated by full
cancellation, and \eqref{eq:twolevel-rate} says nothing that
\eqref{eq:small-k-sharp} does not say better. The reason is structural and
worth stating, though we are careful about its logical status. In the
\emph{explicit separated-level construction}, achieving an estimation error of
order \(k^2\) naturally leads to a separation on the scale
\(\sqrt{\log(1/k)}\) --- that is what the bound
\(\operatorname{mmse}(Q_M)\le CMe^{-M^2/8}\) and the choice
\(M=4\sqrt{\log(1/k)}\) deliver. We have \emph{not} proved that every
low-MMSE signalling law must separate its levels on that scale; the displayed
estimate is a sufficient construction, not a necessary condition. Granting the
scale, the prior width \(\sigma\) is \emph{fixed}, so the number of levels
fitting inside the bulk of \(P\) shrinks to one as \(k\downarrow0\); this,
together with the numerical level-count study of Remark~\ref{rem:crossover},
suggests why separated signalling levels are eventually squeezed out, and is
consistent with the proved statement \eqref{eq:small-k-sharp}. The
signalling constellations are a moderate-\(k\) phenomenon, not the
\(k\to0\) limit. Remark~\ref{rem:crossover} quantifies where the changeover
happens.
\end{remark}

\begin{remark}[Where cancellation overtakes signalling]
\label{rem:crossover}
Proposition~\ref{prop:asymptotics}(b) says \(J^\star(k)/k^2\uparrow\sigma^2\),
so the advantage of signalling over full cancellation vanishes as
\(k\downarrow0\). It vanishes extremely slowly, which is why the numerics of
\S\ref{sec:numerics} are not in conflict with it. A scan over symmetric
equal-mass finite-level constellations at \(\sigma=5\) (level count \(n\) and
a scale factor on the cell centroids, both optimized by direct evaluation, so
every entry is the numerically evaluated cost of a \emph{feasible} law and
hence bounds \(J^\star\) up to quadrature error --- type~(N2)
of \S\ref{sec:prelim}) gives
\[
\begin{array}{c|cccccccc}
k & 0.5 & 0.2 & 0.1 & 10^{-2} & 10^{-3} & 10^{-4} & 10^{-5} & 10^{-6}\\[2pt]
\hline
\widehat J(k)/k^2 & 4.0 & 6.5 & 8.0 & 10.3 & 12.7 & 15.9 & 19.8 & 24.1\\
\text{best }n & 31 & 3 & 3 & 2 & 2 & 2 & 2 & 2
\end{array}
\]
against the ceiling \(\sigma^2=25\), which \(\delta_0\) attains exactly and
which becomes the best member of the family at \(k\approx3\times10^{-7}\).
(The free-weight solver of \S\ref{sec:numerics} does better still at moderate
\(k\): \(\widehat J(0.1)/k^2=5.3\).) Two things are visible. First, on every
scale at which the counterexample is normally studied, signalling beats
cancellation --- by a factor of about five at \(k=0.1\) --- and only in the
far tail \(k\lesssim10^{-6}\) does the leading-order behaviour
\eqref{eq:small-k-sharp} take over. Second, and contrary to the natural
guess, the number of levels \emph{decreases} as control cheapens past the
signalling regime, \(31\to3\to2\to1\): the required separation grows like
\(\sqrt{\log(1/k)}\) while the prior scale \(\sigma\) does not, so levels are
progressively squeezed out. The proliferation of levels seen in
Figures~\ref{fig:density} and~\ref{fig:keffect} is a moderate-\(k\)
phenomenon.
\end{remark}

\begin{remark}[What the two limits do and do not assert]
\label{rem:rates}
Statement~(a) of Proposition~\ref{prop:asymptotics} is \emph{one-sided},
because it is proved by exhibiting a competitor: the Gaussian
gives \(\operatorname{mmse}(P)-J^\star(k)\ge c\,k^{-2}\), while the
matching upper bound \(O(k^{-2})\) would follow from a local Lipschitz
estimate \(\operatorname{mmse}(P)-\operatorname{mmse}(Q)\le L\,W_2(P,Q)\)
valid along the optimizers \(Q^\star_k\) (then
\(\operatorname{mmse}(P)-J^\star\le LW_2-k^2W_2^2\le L^2/(4k^2)\)), which we do
not establish; Fact~\ref{fact:mmse} gives continuity but not a modulus.
Statement~(b), by contrast, is two-sided and exact at leading order. What it
does \emph{not} give is a second-order term: how
\(\sigma^2-J^\star(k)/k^2\) tends to \(0\) is left open, and by
Remark~\ref{rem:crossover} that gap is what carries all of the signalling
structure at practical \(k\). Note also that
\(Q^\star_k\rightharpoonup\delta_0\) is consistent with
Proposition~\ref{prop:ac}: each \(Q^\star_k\) is atomless for \(k>0\), and
only the weak limit is a Dirac. Neither \(\delta_0\) nor \(Q_M\) is the
minimizer at any \(k>0\), again by Proposition~\ref{prop:ac}.
\end{remark}

Translating these laws through the dictionary of
Theorem~\ref{thm:ot-reformulation} (\(f^\star=F_{Q^\star}^{-1}\circ F_P\),
\(\gamma_2^\star(z)=\mathbb E[Y\mid Z=z]\)) gives the two limiting
controllers of the \emph{original} Witsenhausen problem in closed form.

\begin{corollary}[Limiting controllers]
\label{cor:limit-controllers}
Write the first controller as a map \(f\) with \(Y=f(X_0)\)
(control \(U_1=f(X_0)-X_0\)) and the second as the decoder
\(\gamma_2(z)=\mathbb E[Y\mid Z=z]\). Then, along the regimes of
Proposition~\ref{prop:asymptotics},
\begin{enumerate}
\item[\textup{(a)}] \emph{Expensive control, \(k\to\infty\) (linear
regime).} The optimal controllers converge to the \emph{affine} pair
\[
f^\star(x)=x\quad(U_1\equiv0),\qquad
\gamma_2^\star(z)=\frac{\sigma^2}{1+\sigma^2}\,z,
\]
the first in \(L^2(P)\) and the second locally uniformly in \(z\); i.e.\
\(C_1\) leaves \(Y=X_0\) and \(C_2\) applies the Wiener/Bayes
estimator, the classical linear solution. This is a statement about the
\(k\to\infty\) limit only: at finite \(k\) neither controller is proved
affine (Remark~\ref{rem:hermite}).
\item[\textup{(b)}] \emph{Cheap control, \(k\to0\).} The controllers converge
to \emph{full cancellation}:
\[
f^\star_k\longrightarrow0\ \text{ in }L^2(P)
\quad(U_1\to-X_0),\qquad
\gamma_{2,k}^\star\longrightarrow0\ \text{ locally uniformly},
\]
i.e.\ \(C_1\) drives the state to the origin and \(C_2\) has nothing left to
estimate. The first statement is proved in
Proposition~\ref{prop:asymptotics}(b); the second follows from it, since
\(Q^\star_k\to\delta_0\) in \(W_2\) gives
\(g_k(z)=\int y\,\phi(z-y)dQ^\star_k\to0\) and
\(p_k=Q^\star_k*\phi\to\phi\) uniformly on compacta, whence
\(\gamma^\star_{2,k}=g_k/p_k\to0\) there. The cost is
\(J^\star(k)=\sigma^2k^2+o(k^2)\).
\item[\textup{(b\('\))}] \emph{An explicit signalling construction.} For each
fixed \(k\), with \(M=M(k)=4\sqrt{\log(1/k)}\), the two-level pair
\[
\tilde f_k(x)=M\,\operatorname{sign}(x)\quad
\bigl(U_1=M\,\operatorname{sign}(X_0)-X_0\bigr),\qquad
\tilde\gamma_{2,k}(z)=M\tanh(Mz),
\]
(the first identity holding \(P\)-almost everywhere with the convention
\(\operatorname{sign}(0)=1\) of \S\ref{sec:prelim}) attains the bound
\eqref{eq:twolevel-rate}, \(J(\tilde f_k,\tilde\gamma_{2,k})
=O(k^2\log(1/k))\). \(C_1\) pushes \(X_0\) onto the two well-separated levels
\(\pm M\) and \(C_2\) reads them off. This is the cleanest illustration of
signalling, and it is what the finite-level solver returns in coarsened form
at moderate \(k\); it is \emph{not} the \(k\to0\) limit, being dominated by
part~(b) (Remark~\ref{rem:twolevel}).
\end{enumerate}
The two \emph{atomic} constructions above --- full cancellation
\(\delta_0\) in part~(b) and the two-level law \(Q_M\) in
part~(b\('\)) --- are not optimal at any \(k>0\), since by
Proposition~\ref{prop:ac} no minimizer carries atoms. (Part~(a) is a
different kind of statement: the affine pair is a \(k\to\infty\) limit, and
whether it is optimal at finite \(k\) is open, Remark~\ref{rem:hermite}.)
\end{corollary}

\begin{proof}
\emph{(a)} \emph{First controller.} By Proposition~\ref{prop:asymptotics}(a),
\(Q^\star_k\to P\) in
\(W_2\), and since \(P\) is atomless the monotone coupling realizes that
distance (Fact~\ref{fact:1d-ot}), so
\(\mathbb E[(f^\star_k(X_0)-X_0)^2]=W_2^2(P,Q^\star_k)\to0\): the optimal
first controllers converge in \(L^2(P)\) to
\(f^\star(x)=x\), i.e.\ \(U_1\to0\).

\emph{Second controller.} Here we must show that the \emph{decoders converge},
not merely that the decoder \emph{of} the limit law is the Wiener filter, and
we do so in the same local-uniform sense as part~(b). Put
\[
p_k(z)=\int\phi(z-y)\,dQ^\star_k(y),
\qquad
g_k(z)=\int y\,\phi(z-y)\,dQ^\star_k(y),
\qquad
\gamma^\star_{2,k}=\frac{g_k}{p_k}.
\]
For each fixed \(z\) the integrands \(y\mapsto\phi(z-y)\) and
\(y\mapsto y\phi(z-y)\) are continuous and \emph{bounded} (the second because
\(\sup_y|y|\phi(z-y)<\infty\)), so \(W_2\)-convergence --- which implies weak
convergence --- gives \(p_k(z)\to p_P(z)\) and \(g_k(z)\to g_P(z)\)
pointwise; and the \(z\)-derivatives are bounded uniformly in \(k\) by
\(\|\phi'\|_\infty\) and \(\sup_y|y\phi'(z-y)|\), so the families are
equicontinuous and the convergence is uniform on compact \(z\)-sets. Since
\(p_P=P*\phi>0\) everywhere, the quotient converges too:
\(\gamma^\star_{2,k}\to g_P/p_P\) uniformly on compacta. Finally
\(g_P/p_P=\mathbb E[Y\mid Z=z]\) for \(Y\sim P\): the pair \((Y,Z)\) with
\(Z=Y+N\) is jointly Gaussian and centered, so the conditional mean is
linear, \(\mathbb E[Y\mid Z=z]=\bigl(\operatorname{Cov}(Y,Z)/
\operatorname{Var}(Z)\bigr)z\); and as in the proof of
Lemma~\ref{lem:mmse-bounds}, \(\operatorname{Cov}(Y,Z)=\operatorname{Var}(Y)
=\sigma^2\) and \(\operatorname{Var}(Z)=\sigma^2+1\), giving
\(\gamma_2^\star(z)=\frac{\sigma^2}{1+\sigma^2}z\), the Wiener/Bayes filter.

\emph{(b)} This is Proposition~\ref{prop:asymptotics}(b) together with the
consequence of \(Q^\star_k\to\delta_0\) in \(W_2\) recorded in the statement.

\emph{(b\('\))} For \(Q_M=\tfrac12(\delta_{-M}+\delta_M)\) the distribution
function is \(F_{Q_M}=0\) below \(-M\), \(=\tfrac12\) on \([-M,M)\), and
\(=1\) from \(M\) on; its generalized inverse
\(F_{Q_M}^{-1}(u)=\inf\{y:F_{Q_M}(y)\ge u\}\) is therefore \(-M\) for
\(u\in(0,\tfrac12]\) and \(M\) for \(u\in(\tfrac12,1)\). Composing with
\(F_P\), which is continuous and strictly increasing with
\(F_P(x)<\tfrac12\) exactly for \(x<0\),
\[
\tilde f_k(x)=F_{Q_M}^{-1}(F_P(x))=
\begin{cases}-M,& x<0,\\ M,& x>0,\end{cases}
\qquad\text{i.e.}\qquad \tilde f_k=M\operatorname{sign}\ \ P\text{-a.e.},
\]
which is the same monotone map used in
Remark~\ref{rem:twolevel}, as it must be. For the decoder, that
remark computed \(p:=\Pr(Y=M\mid Z=z)=(1+e^{-2Mz})^{-1}\); the posterior mean
of a variable taking \(\pm M\) with probabilities \(p,1-p\) is
\(Mp-M(1-p)=M(2p-1)\), and
\[
2p-1=\frac{2}{1+e^{-2Mz}}-1=\frac{1-e^{-2Mz}}{1+e^{-2Mz}}
=\frac{e^{Mz}-e^{-Mz}}{e^{Mz}+e^{-Mz}}=\tanh(Mz),
\]
so \(\tilde\gamma_{2,k}(z)=M\tanh(Mz)\): a soft threshold that saturates at
\(\pm M\), i.e.\ reads off which level was sent.
\end{proof}

The sign map of part~(b\('\)) is the coarsest instance of the staircase-like
signalling profile of
Figure~\ref{fig:map}, and of the finite-level staircases of
Figure~\ref{fig:fn}, which refine it into more levels as \(n\) grows. It is a
construction, not a limit; the true \(f^\star\) at any \(k>0\) is
strictly increasing (Fact~\ref{fact:regularity}) and so has no exact
plateaus.

The mode counts of Figure~\ref{fig:keffect} belong to the moderate-\(k\)
window, and should be read as such. Over the range swept there, decreasing
\(k\) does make the computed levels spread out and multiply, each added
well-separated level cutting the residual MMSE at only \(O(k^2)\) transport
cost. As \(k\) increases the computed signalling branch loses its separated
levels and approaches the near-identity Gaussian branch. Above
\(k_{\mathrm{stab}}\) that branch is second-variation stable
(Remark~\ref{rem:hermite}), but global Gaussian/affine optimality there is
\emph{not} established, so we do not say that the levels collapse to ``the''
optimum. Neither trend
extrapolates to \(k\downarrow0\) either: by
Proposition~\ref{prop:asymptotics}(b) and
Remark~\ref{rem:crossover} the level count turns around and falls back to one
far below the range shown, and the minimizers collapse onto \(\delta_0\).

\section{The Vector Counterexample}
\label{sec:vector}

Nothing in the reformulation is special to one dimension. Let
\(X\sim P=\mathcal N(0,\sigma^2 I_d)\) on \(\mathbb R^d\), let Controller~1 apply
\(Y=T(X)\), and let Controller~2 observe \(Z=Y+N\) with \(N\sim\mathcal
N(0,I_d)\) independent of \(Y\) and estimate \(Y\). Absorbing the control into a
transport map gives the same variational problem
\begin{equation}
\label{eq:vector-variational}
J^\star=\inf_{Q}\ \big\{\,k^2\,W_2^2(P,Q)+\operatorname{mmse}(Q)\,\big\},
\qquad Q\in\mathcal P_2(\mathbb R^d),
\end{equation}
now with \(W_2\) the quadratic Wasserstein distance on \(\mathbb R^d\) and
\(\operatorname{mmse}(Q)=\mathbb E\,\|Y-\mathbb E[Y\mid Z]\|^2\) the vector MMSE;
the formulation is thus dimension-independent in form, though of course the
values of both terms depend on \(d\). The optimal first controller is again the
optimal-transport map pushing \(P\) to a minimizer \(Q^\star\), in
\(\mathbb R^d\) the Brenier map (the gradient of a convex potential), of which
the scalar monotone rearrangement is the special case. We record carefully
which parts of \S\ref{sec:characterization} carry over and which do not.

\emph{Existence} carries over unchanged: the proof of
Proposition~\ref{prop:existence} used only coercivity of \(k^2W_2^2\),
tightness via bounded second moments, weak lower semicontinuity of \(W_2^2\),
and weak continuity of \(\operatorname{mmse}\), all of which hold verbatim on
\(\mathbb R^d\).

\emph{Absence of atoms} carries over, with one detail that must be supplied
rather than asserted: the spreading direction. Suppose \(Q^\star\) has an atom
at \(y_0\) of mass \(\alpha\), and let \(I\) be its preimage under the Brenier
map, so \(P(I)=\alpha>0\). Since \(P\ll\mathcal L^d\) and \(P(I)>0\), the
conditional law \(P_I=P(\cdot\mid I)\) is itself absolutely continuous with
respect to \(d\)-dimensional Lebesgue measure; an affine hyperplane is
\(\mathcal L^d\)-null, so \(P_I\) cannot be supported on one. (Mere
non-degeneracy would not suffice in \(d>1\): a nontrivial law can live
entirely on a line.) Its covariance is therefore nonzero, and there is a unit
vector \(e\) with
\[
\operatorname{Var}\bigl(\langle X,e\rangle\mid X\in I\bigr)>0 ,
\]
and any \(e\) in the range of the (nonzero) covariance of \(P_I\) will do ---
for example a top eigenvector. Fix such an \(e\), let \(U\) be the scalar
variable of Proposition~\ref{prop:ac}, and replace
\(\alpha\delta_{y_0}\) by \(\alpha\,\mathrm{Law}(y_0+sUe)\). Couple by
transporting \(P_I\) to this cluster comonotonically \emph{in the coordinate}
\(\langle\cdot,e\rangle\) (leaving the orthogonal components matched to
\(y_0\)); the transport estimate then reduces to the
one-dimensional computation of Proposition~\ref{prop:ac} in that coordinate,
giving the same first-order gain \(-2k^2\alpha\kappa s\) with
\[
\kappa=\operatorname{Cov}\bigl(F^{-1}_{\langle X,e\rangle\mid I}(T),\,
F_U^{-1}(T)\bigr)>0 ,
\]
positive precisely because the projected conditional law is non-degenerate by
the choice of \(e\), against an \(O(s^2)\) estimation
loss. So no minimizer has an atom.

\emph{The optimality condition carries over only at the level of measures.}
The first-variation identities of
Lemmas~\ref{lem:wass-derivative}--\ref{lem:mmse-derivative} and the
exchange-inequality argument of Theorem~\ref{thm:EL} are
dimension-independent, so the complementary-slackness form
\(\Lambda^\star=\lambda\) \(Q^\star\)-a.e.\ and \(\Lambda^\star\ge\lambda\)
everywhere holds on \(\mathbb R^d\) with
\(\Lambda^\star=2k^2\varphi^\star+\Psi^\star\),
\(\Psi^\star(y)=\mathbb E[\|y-m(Z)\|^2\mid Y=y]\). What does \emph{not} carry
over is the passage to the pointwise differential (score) form: Step~3 of
that proof used the scalar regularity of Fact~\ref{fact:regularity} --- an
absolutely continuous optimizer with \emph{interval} support --- to upgrade
\(Q^\star\)-a.e.\ equality to an identity on a connected set and then
differentiate. We know of no vector analogue of
Fact~\ref{fact:regularity}, and we do not prove one. Accordingly we use only
the measure-level condition in \(\mathbb R^d\); a pointwise vector score
equation would require additional regularity of \(Q^\star\) (absolute
continuity and connected support) that is open here.

\paragraph{The finite-level program is a vector quantizer.}
Restricting \(Q\) to \(n\) atoms in \(\mathbb R^d\) turns
\eqref{eq:vector-variational} into an \emph{MMSE-regularized vector quantizer}.
Given levels \(\{y_i\}\subset\mathbb R^d\), the \emph{transport} term alone is
minimized over partitions by the Euclidean nearest-level regions, the
\emph{Voronoi polytopes} of the levels; if the masses were instead
\emph{prescribed}, the transport-optimal cells would be the weighted Laguerre
cells of semi-discrete optimal transport. But the objective is not the
transport term alone. Because \(\operatorname{mmse}\) depends on the masses
\(w_i=P(\text{cell}_i)\), moving a face of the partition changes the
estimation cost too, and the stationarity condition for the partition is the
exact vector analogue of the scalar threshold
condition~\eqref{eq:voronoi-thresh}: with
\(\Psi(y)=\mathbb E[\|y-m(Z)\|^2\mid Y=y]\) as in~\eqref{eq:psi-condexp}, a
stationary partition assigns
\begin{equation}
x\ \longmapsto\ \operatorname*{arg\,min}_i
\Bigl\{\,\|x-y_i\|^2+k^{-2}\Psi(y_i)\,\Bigr\},
\label{eq:power-cell-vec}
\end{equation}
so the cells are \emph{power (Laguerre) polytopes} with additive weights
\(k^{-2}\Psi(y_i)\), not ordinary Voronoi polytopes. Their faces are still
hyperplanes, but each is offset from the perpendicular bisector of \(y_i,y_j\)
by \(\bigl(\Psi(y_j)-\Psi(y_i)\bigr)/\bigl(2k^2\|y_j-y_i\|\bigr)\) along
\(y_j-y_i\); the offsets cancel, and ordinary Voronoi cells are recovered,
exactly when the weights coincide, in particular as \(k\to\infty\).
The estimation-aware Lloyd alternation of
Section~\ref{sec:finite-program-sec} applies verbatim
(Remark~\ref{rem:general-solver}): an \emph{assignment step}
uses~\eqref{eq:power-cell-vec}, and a \emph{centroid step} updates
\(y_i\leftarrow\bar x_i-(2k^2w_i)^{-1}\nabla_{y_i}\operatorname{mmse}\), the
MMSE-corrected conditional centroid. Along a \emph{nondegenerate} family of
stationary points --- levels bounded, cell masses bounded away from zero and
adjacent levels uniformly separated, exactly the hypotheses
\eqref{eq:nondegenerate} of Corollary~\ref{cor:lloydmax-limit}, which are
needed here for the same reason and are equally not automatic --- both
corrections are \(O(k^{-2})\) and vanish as \(k\to\infty\), leaving the
classical vector Lloyd--Max (Linde--Buzo--Gray) quantizer of
\(P\);
for finite \(k\) the estimation term pushes the levels outward and shifts the
cell boundaries off the perpendicular bisectors, exactly the two-dimensional
analogue of the scalar picture in Figure~\ref{fig:voronoi-cells}.

\paragraph{What the reported computations solve.}
The solver used for Figures~\ref{fig:vector}--\ref{fig:vecaniso2cells}
(\texttt{vector\_lloyd.py}) implements the centroid step exactly but takes the
\emph{uncorrected} nearest-level assignment, i.e.\ \(k^{-2}\Psi\equiv0\)
in~\eqref{eq:power-cell-vec}. Its output is a \emph{fixed point of that
approximate alternation} --- a computed candidate --- and we claim no more.
In particular it is not a proved stationary point of the full program, and it
is not a proved stationary point of the original objective restricted to
Voronoi partitions either. The second point is worth spelling out, because it
is tempting to assume otherwise. Constraining the cells to be the Voronoi
cells of the levels makes \(\zeta\) a function \(\zeta_V(y)\), and the
derivative of the restricted objective is
\begin{equation}
\frac{d}{dy}\,J\bigl(y,\zeta_V(y)\bigr)
=\partial_yJ+\partial_\zeta J\cdot D\zeta_V(y) .
\label{eq:voronoi-restricted-deriv}
\end{equation}
The centroid step annihilates the first term, with the cells held fixed. It
does not annihilate the second: by the scalar computation in the proof of
Proposition~\ref{prop:voronoi}, \emph{at} a Voronoi midpoint the transport
boundary terms cancel but the estimation term does not, leaving
\[
\partial_{\zeta_i}J\Big|_{\zeta_i=\frac{y_i+y_{i+1}}{2}}
=p(\zeta_i)\bigl[\Psi(y_i)-\Psi(y_{i+1})\bigr],
\]
which vanishes only when the two estimation sensitivities coincide. So
\eqref{eq:voronoi-restricted-deriv} is generally nonzero at the returned
point, and calling it stationary --- for either objective --- would be wrong.
(This is consistent with Remark~\ref{rem:general-solver}, where even the
\emph{corrected} alternation is established only to be monotone in value.)
What we do offer is numerical validation, and we
quantify the gap. Evaluating \(\Psi\) at the returned constellations and
re-running the alternation with the corrected
rule~\eqref{eq:power-cell-vec} changes nothing at the resolution reported: at
\(\sigma=5\), \(k=0.2\), \(n=9\) the weights \(k^{-2}\Psi(y_i)\) lie in
\([0.52,0.61]\), so with a level separation of \(7.06\) the faces move by at
most \(\approx0.007\) and the cost is unchanged to five digits
(\(\widehat J_9=0.4227\) either way); for the correlated prior of
Figure~\ref{fig:vecaniso} the weights spread over \(0.84\) and the corrected
alternation returns the same constellation to four digits; and for the
diagonal prior of Figure~\ref{fig:vecaniso2} the spread is \(0.30\), again
with no change in cost. The reason is structural: the offsets enter only
through their \emph{differences} divided by twice the level separation, and
in the signalling regime the levels are far apart. The reported vector
solutions are thus numerically indistinguishable from stationary points of
the full program, but we emphasize that this is a verified approximation, not
an identity.

\paragraph{A two-dimensional example.}
Figure~\ref{fig:vector} shows the computed solution for \(d=2\), \(\sigma=5\),
\(k=0.2\), obtained by the generalized Lloyd iteration with \(n=9\). The
returned constellation \(\widehat Q_9\) is a \(3\times3\) signalling grid of
nine levels with each
coordinate near \(0\) or \(\pm7.06\), so the decision cells \emph{(a)} are the
nine rectangular cells with boundaries at \(\pm3.53\). The levels
sit well outside the bulk of the prior (the dashed \(1\sigma\) and \(2\sigma\)
circles): this is the two-dimensional \emph{signalling} push, the estimation
term driving the per-axis levels from the \(\approx\pm6.1\) of the
\(k\to\infty\)
vector Lloyd--Max out to \(\pm7.06\) so that \(Y\) is easier to resolve through
the noise. The masses concentrate at the centre (\(0.29\)) and are smallest at
the corners (\(0.05\)). Panel \emph{(b)} is the received density
\(\widehat Q_9*\phi_2\) that Controller~2 must resolve, nine Gaussian lumps
whose
separation is the estimation gain bought by the transport. The cost splits as
\(\widehat J_9\approx0.42=0.40\ (\text{transport})+0.02\ (\text{estimation})\),
an evaluated feasible value, hence an upper estimate of \(J_9^\star\) and of
\(J^\star\) up to quadrature error.

\begin{figure}[t]
\centering
\includegraphics[width=\linewidth]{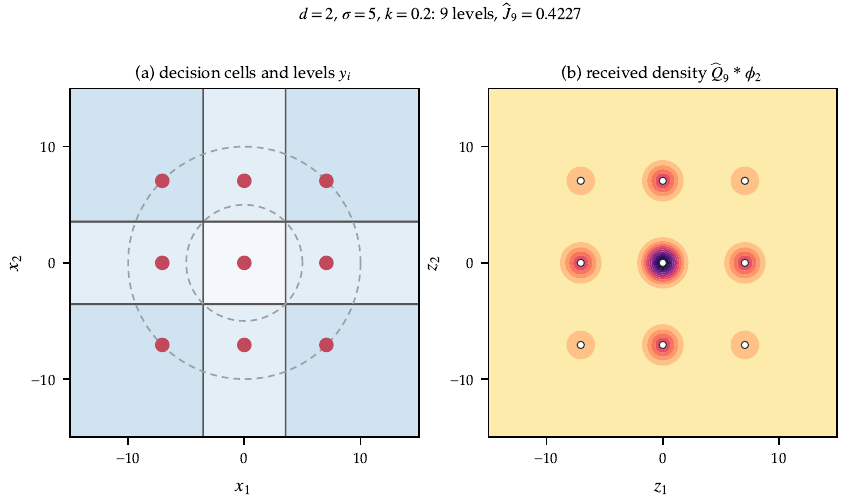}
\caption{The vector counterexample in \(d=2\) (\(\sigma=5\), \(k=0.2\)), solved
by the estimation-aware Lloyd iteration with \(n=9\). \emph{(a)} The nine
decision cells (shaded by distance from the origin, boundaries solid) are the
Voronoi polytopes of the computed levels \(\widehat y_i\) (crimson), with the
prior
\(1\sigma\) and \(2\sigma\) circles dashed; the levels are pushed outward (each
coordinate to \(\approx\pm7.06\), versus \(\approx\pm6.1\) for the
\(k\to\infty\) vector Lloyd--Max), which produces the two-dimensional
signalling pattern. \emph{(b)} The
received density \(\widehat Q_9*\phi_2\) that Controller~2 sees, nine
Gaussian
lumps at the levels (marked) broadened by the unit noise. Cost
\(\widehat J_9\approx0.42\) (transport \(0.40\), estimation \(0.02\)). A
computed
candidate of a nonconvex program, obtained under the nearest-level
(Voronoi) restriction of the assignment step; the exact power-cell
weights~\eqref{eq:power-cell-vec} move the faces by \(\lesssim0.01\) here and
leave the cost unchanged to five digits. The solver is the ancillary
\texttt{vector\_lloyd.py}.}
\label{fig:vector}
\end{figure}

\paragraph{Product versus genuinely vector structure.}
For the isotropic prior the constellation comes out close to a product grid,
each axis resembling a scalar quantizer. The reason is that a product law makes
both
\(W_2^2\) and \(\operatorname{mmse}\) separate, so a product of scalar
optimizers is \emph{stationary} for the vector problem --- and this deserves
an argument, since separability of the \emph{value} at product laws gives
stationarity only within the product class, whereas what is needed is
stationarity against arbitrary, possibly correlated, perturbations. The point
is that the first variation itself is additive. Take \(d=2\),
\(P=P_1\otimes P_2\) and \(Q=Q_1\otimes Q_2\). For product measures the
Brenier map is the product of the scalar monotone maps, so
\(W_2^2(P,Q)=W_2^2(P_1,Q_1)+W_2^2(P_2,Q_2)\) and the Kantorovich potential
splits, \(\varphi_Q(y)=\varphi_{Q_1}(y_1)+\varphi_{Q_2}(y_2)\). For the
estimation term, \(Y_1\perp Y_2\) and \(N_1\perp N_2\) make \(Z_1\perp Z_2\)
and \(m(z)=(m_1(z_1),m_2(z_2))\), so by~\eqref{eq:psi-condexp}
\[
\Psi_Q(y)=\mathbb E\bigl[\|y-m(Z)\|^2\bigm|Y=y\bigr]
=\Psi_{Q_1}(y_1)+\Psi_{Q_2}(y_2).
\]
Hence the cost sensitivity of Theorem~\ref{thm:EL} is additive,
\begin{equation}
\Lambda(y_1,y_2)=\Lambda_1(y_1)+\Lambda_2(y_2),
\qquad \Lambda_i=2k^2\varphi_{Q_i}+\Psi_{Q_i},
\label{eq:Lambda-additive}
\end{equation}
where each \(\Lambda_i\) is exactly the scalar sensitivity built from
\(Q_i\). If each factor satisfies the scalar Euler--Lagrange
condition~\eqref{eq:EL} --- \(\Lambda_i=\lambda_i\) on
\(\operatorname{supp}Q_i\) and \(\Lambda_i\ge\lambda_i\) everywhere --- then
\(\Lambda=\lambda_1+\lambda_2\) on
\(\operatorname{supp}Q_1\times\operatorname{supp}Q_2=\operatorname{supp}Q\)
and \(\Lambda\ge\lambda_1+\lambda_2\) on all of \(\mathbb R^2\), which is
precisely the vector condition. Since~\eqref{eq:EL} tests \emph{all}
mass-zero perturbations, product or not, this is genuine stationarity for the
vector problem.

Stationarity is all
that separability gives, however, and we stress that it does not give
optimality:
nothing here rules out a genuinely vector minimizer, and the classical
vector-quantization \emph{shape gain} is precisely the phenomenon in which
non-product lattices (the hexagonal being the standard planar example) beat
the product grid, the advantage growing with \(n\) and with the dimension.
Proving that the vector optimum tensorizes would require a tensorization
theorem for the functional \(k^2W_2^2+\operatorname{mmse}\), which we do not
have; accordingly we call the product construction a \emph{separable
stationary candidate}, and the values it attains \emph{computable upper
bounds} on \(J^\star\), never the optimum.
The program is in any case solved as a genuine free-mass vector
quantizer, not a product by construction: the masses are jointly optimized
(hence the centre-heavy, corner-light weighting, which a rigid product of equal
scalar solutions would not produce), and non-product constellations are equally
available to the same Lloyd alternation; seeding the alternation with a
hexagonal or other lattice can explore that branch, and we have not carried out
that search. This is where the quantization viewpoint
pays off most: the vector Witsenhausen problem, long regarded as intractable,
inherits a concrete solver, the same MMSE-regularized Lloyd iteration, with the
decision regions now polytopes. As in the scalar case the finite-level vector
law is atomic by construction, a best-found computed candidate of a nonconvex
program;
the true \(Q^\star\) is atomless, by the argument above. That the finite-level
values converge, \(J_n^\star\downarrow J^\star\) in \(\mathbb R^d\), is the
\(d\)-dimensional analogue of Proposition~\ref{prop:finite-level}, and its
proof carries over unchanged: monotonicity is set inclusion; the lower bound
is feasibility; and the upper bound needs only that finitely supported
measures are \(W_2\)-dense in \(\mathcal P_2(\mathbb R^d)\), which holds in
every dimension. The attainment and weak-closedness argument
(Proposition~\ref{prop:finite-level}(i)) is likewise dimension-free, the
portmanteau step using only disjoint open neighbourhoods.

\paragraph{Convergence in the level count.}
The convergence in \(n\) of \S\ref{sec:numerics} carries over to two
dimensions. Figure~\ref{fig:vecconv} shows the computed vector solution over
the state
plane \((x_1,x_2)\) as the level count grows through the \(m\times m\) grids
\(n=4,9,25\), and, in the bottom row, the separable candidate that the
sequence approaches, each row a case
and each column one of the three objects, all as surfaces or stems over
\((x_1,x_2)\). The \emph{law} \(\widehat Q_n\) (left) gains atoms, from the
\(2\times2\)
corners to a \(5\times5\) constellation, their stem heights the masses \(w_i\),
settling toward the smooth atomless product density
\(\widehat q(y_1)\,\widehat q(y_2)\);
the \emph{controller} \(\widehat f_n\) (centre, its first coordinate
\(\widehat f_n^{(1)}\), the second being the same map with \(x_1,x_2\) swapped)
refines
from a single sign step into a finer staircase, the two-dimensional analogue of
the scalar staircases of Figure~\ref{fig:conv3d}, approaching the smooth
separable
map \(\widehat T^{(1)}\); and the \emph{observation} density
\(\widehat Q_n*\phi_2\)
(right) develops from four lumps into the fine array that
Controller~2 must resolve. The cost falls monotonically,
\(\widehat J_n=2.00,0.73,0.42,0.35,0.34\) for \(n=1,4,9,16,25\), within a few
percent of that value by \(n=25\); beyond that, added levels are
corner atoms of vanishing mass. The bottom row is the product of the computed scalar
solution with itself at \(k=0.2\), with cost \(2\times0.17=0.34\)
(matching
\(n=25\); the transport-map scalar solve at \(k=0.2\) gives \(0.17\), refining
the coarser \(0.22\) tabulated in Example~\ref{ex:keffect}, as at
\(k=0.4,0.5\)).
For the isotropic prior, a product of \emph{exact} scalar stationary laws is
stationary for the vector problem, by the additivity
\eqref{eq:Lambda-additive}. What we can actually form is the product of our
\emph{computed} scalar candidates, which is a feasible separable numerical
candidate rather than a certified stationary point. Its value \(0.34\) is an
upper bound on \(J^\star\) because the law is \emph{feasible} --- not
because it is stationary --- and a non-product constellation with a shape gain
is not excluded. Like the scalar
solution, this candidate is \emph{atomless}, a smooth density and not the
point masses of the finite levels.

\begin{figure}[p]
\centering
\includegraphics[width=0.88\linewidth]{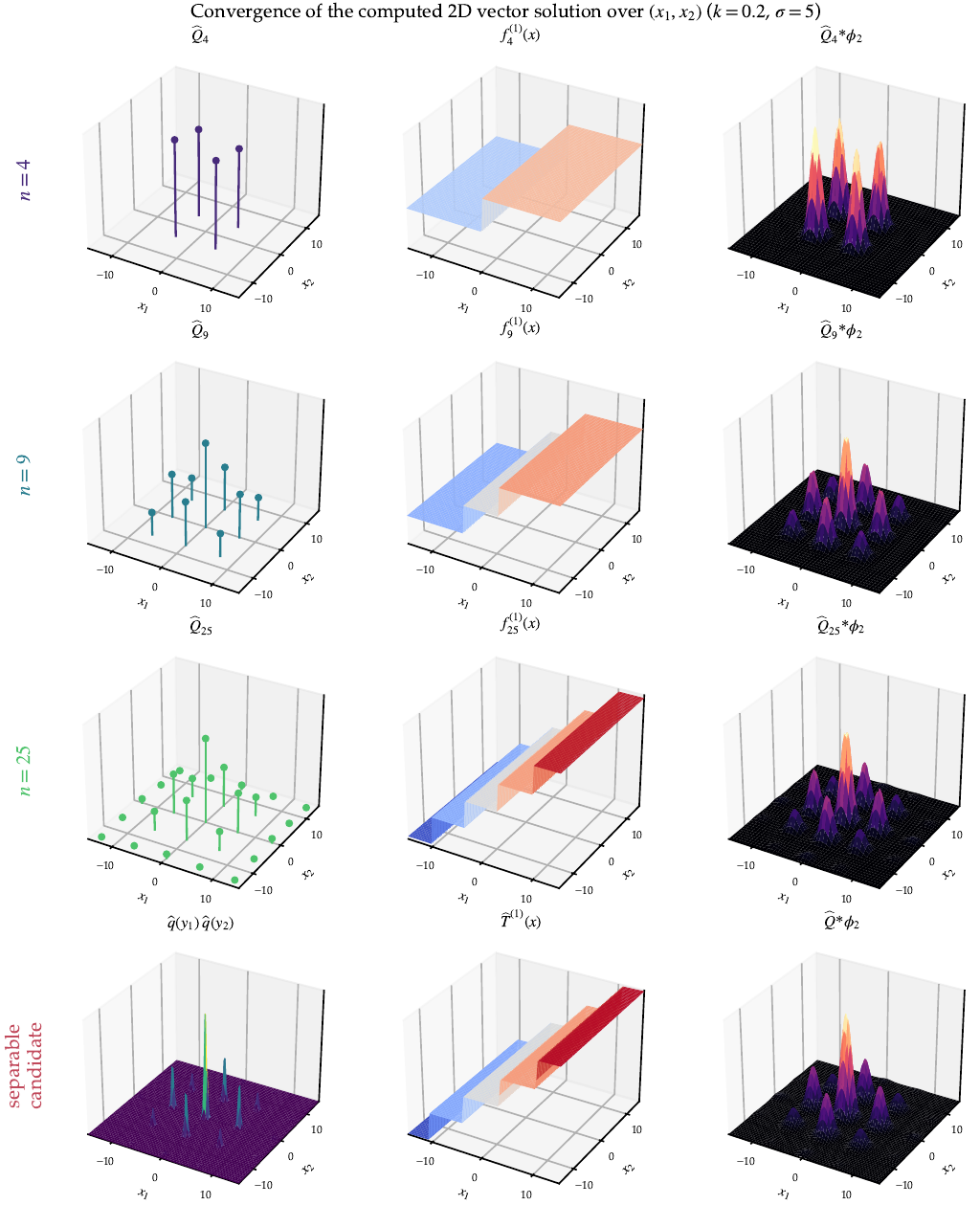}
\caption{Convergence of the computed two-dimensional vector solution over the
state
plane \((x_1,x_2)\) (\(k=0.2\), \(\sigma=5\)). The top three rows
are the finite-level solutions on the \(m\times m\) grids, \(n=4,9,25\) (the
last essentially converged, \(\widehat J_n=0.73,0.42,0.34\)); the bottom row
(crimson) is the continuous \emph{separable candidate}
\(\widehat Q^{(1)}\otimes\widehat Q^{(1)}\), a feasible product of computed
scalar candidates --- an upper bound by feasibility, not a certified
stationary point or optimum. \emph{Left:} the law, from the atoms
\(\widehat Q_n=\sum_i w_i\delta_{y_i}\)
(stems of height \(w_i\)) to the smooth atomless density
\(\widehat q(y_1)\,\widehat q(y_2)\). \emph{Centre:} the controller first
coordinate,
from the staircase \(\widehat f_n^{(1)}(x)\) to the smooth separable map
\(\widehat T^{(1)}(x)\)
(the second coordinate is the same map with the axes swapped). \emph{Right:}
the
observation density Controller~2 sees, from \(\widehat Q_n*\phi_2\) to
\(\widehat Q*\phi_2\). The candidate is the product of the two computed
scalar marginals at this \(k\), cost \(2\times0.17=0.34\).
The finite-level rows were computed with the generalized Lloyd iteration of
\texttt{vector\_lloyd.py}, driven by
\texttt{precompute\_vector\_convergence.py}; the bottom row uses the scalar
solver \texttt{continuous\_solver.py}.}
\label{fig:vecconv}
\end{figure}

\paragraph{An anisotropic prior.}
The isotropic examples above produce axis-aligned grids, which understates how
much the vector quantizer adapts to the prior. Figure~\ref{fig:vecaniso} solves
the same program for a \emph{correlated} prior,
\(P=\mathcal N(0,\Sigma)\) with unit-axis variances \(25\) and correlation
\(\rho=0.6\) (principal variances \(40\) and \(10\) along the \((1,1)\) and
\((1,-1)\) diagonals), at a different penalty \(k=0.15\); only the prior weight
changes, the channel noise staying isotropic. The generalized Lloyd needs no
modification, and the computed solution \emph{tilts and elongates} to match the
covariance: the constellation and its cells align with the principal
axes, spread \(\pm14.3\) along the major diagonal against \(\pm6.8\) along the
minor (ratio \(2.1\), tracking the prior's principal-standard-deviation ratio
\(2.0\)), so the controller signals harder in the high-variance direction where
resolving \(Y\) through the noise buys the most. This is a genuinely
two-dimensional effect, absent from any product of scalar quantizers in the
original coordinates: with
correlation the computed cells are not axis-aligned rectangles but tilted
polytopes. The cost is \(\widehat J_{15}=0.21\) with \(15\) levels, again an evaluated
feasible value.

\begin{figure}[t]
\centering
\includegraphics[width=0.92\linewidth]{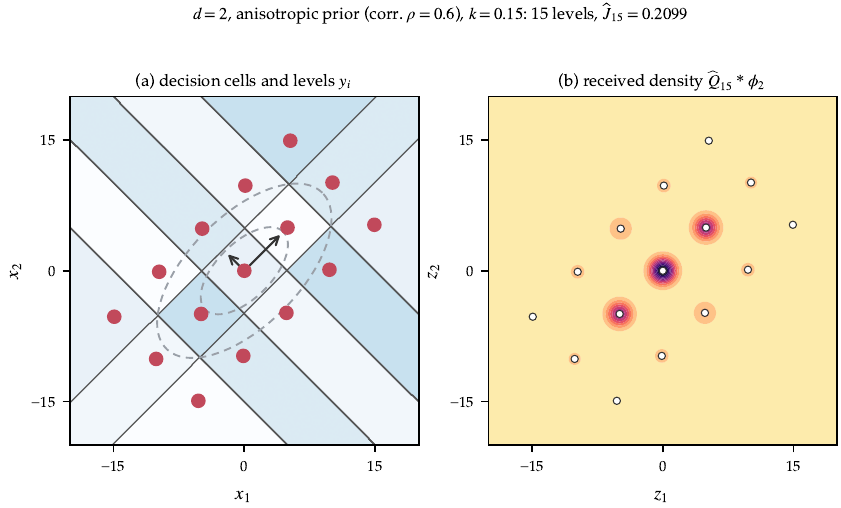}
\caption{The vector solution for an \emph{anisotropic} prior,
\(P=\mathcal N(0,\Sigma)\) with variances \(25\) and correlation \(\rho=0.6\)
(principal axes the \((1,1)\) and \((1,-1)\) diagonals, variances \(40\) and
\(10\)), at \(k=0.15\); \(15\) levels, \(\widehat J_{15}=0.21\). \emph{(a)} The
computed decision
cells (shaded, boundaries solid) and reproduction levels (crimson) tilt and
elongate to align with the prior's principal axes (dashed \(1\sigma\),
\(2\sigma\) ellipses; arrows the principal directions): the constellation
spreads twice as far along the major diagonal as the minor, matching the
covariance. \emph{(b)} The received density \(\widehat Q_{15}*\phi_2\), a
tilted
array of lumps. Only the prior weight in \texttt{vector\_lloyd.py} changes (via
\texttt{make\_grids\_cov}); the channel and the Lloyd alternation are
unchanged. Ancillary preprocessing:
\texttt{precompute\_vector\_aniso.py}.}
\label{fig:vecaniso}
\end{figure}

\paragraph{A higher penalty, and a converged solution.}
At a larger control penalty the same adaptivity persists, now with a converged
solution
shown alongside the coarse finite one as in Figure~\ref{fig:vecconv}.
Figure~\ref{fig:vecaniso2} takes a stronger, \emph{axis-aligned} anisotropy,
\(P=\mathcal N(0,\operatorname{diag}(64,16))\) (\(\sigma_1=8\), \(\sigma_2=4\)),
at the higher penalty \(k=0.3\). Being diagonal the prior separates, so a
product of two scalar quantizers, one per axis at its own variance, is
\emph{stationary} for the vector problem (again, stationary and not thereby
optimal: see the discussion of shape gain above); and because the
high-variance axis \(x_1\) supports far more signalling levels than
\(x_2\), the constellation the solver returns is strongly \emph{rectangular},
about nine levels
along \(x_1\) against three along \(x_2\). The top row is a coarse finite
quantizer (\(n=15\), \(\widehat J_{15}=0.85\)); the bottom row is the converged
solution
(\(\widehat J\approx0.69\)), with the law drawn as a smooth density (a
\(9\times3\)
array of peaks), the controller \(\widehat T^{(1)}\) as a fine staircase, and
the
observation density \(\widehat Q*\phi_2\) as the array Controller~2 resolves.
Its decision cells and received constellation, the axis-aligned counterpart of
Figure~\ref{fig:vecaniso}, are shown in Figure~\ref{fig:vecaniso2cells}: the
cells form a \(9\times3\) rectangular grid, elongated along the
high-variance axis. As before only the prior weight changes; the channel and
the Lloyd alternation are untouched. (Here the converged object is the Lloyd
solution at large \(n\): the
transport-map continuous solver of \S\ref{app:solver}, tuned to \(\sigma=5\), is
not robust at these variances, whereas the finite Lloyd remains reliable.)

\begin{figure}[p]
\centering
\includegraphics[width=\linewidth]{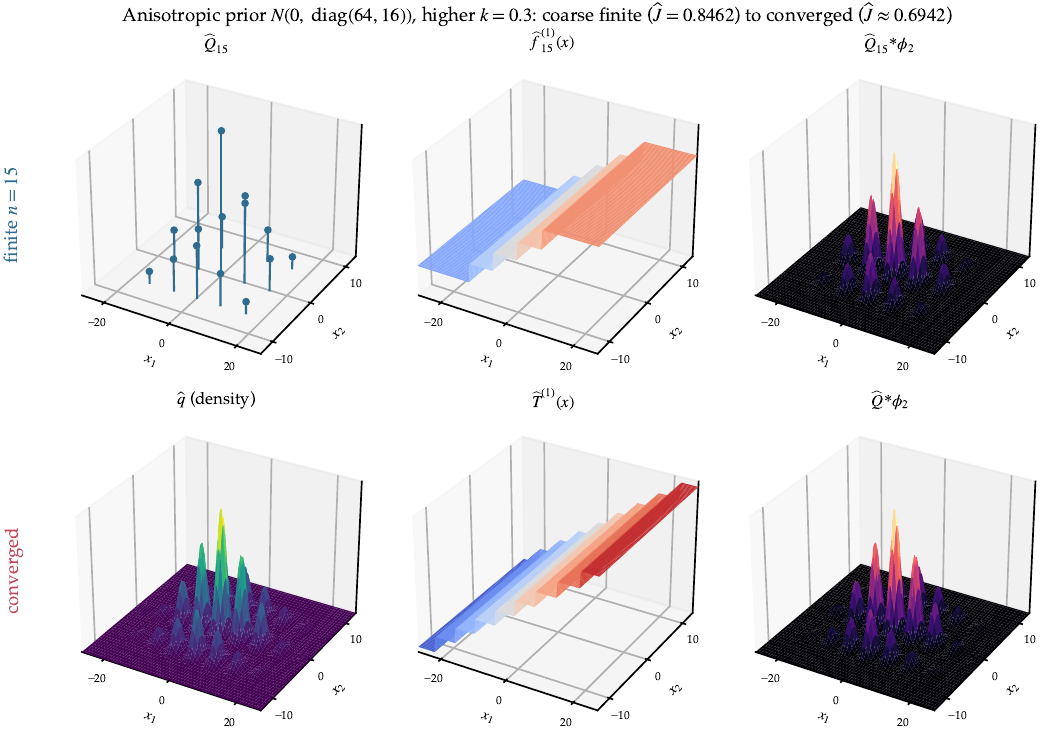}
\caption{A higher-penalty anisotropic case:
\(P=\mathcal N(0,\operatorname{diag}(64,16))\) (\(\sigma_1=8\), \(\sigma_2=4\)),
\(k=0.3\). Top row: a coarse finite quantizer \(\widehat Q_{15}\)
(\(\widehat J_{15}=0.85\)); bottom
row: the converged Lloyd solution (\(\widehat J\approx0.69\), an upper
estimate). Columns are the law (atoms
\(\widehat Q_{15}\), then the smooth density \(\widehat q\)), the controller
first
coordinate (\(\widehat f_{15}^{(1)}\), then \(\widehat T^{(1)}\)), and the
observation
density (\(\widehat Q_{15}*\phi_2\), then \(\widehat Q*\phi_2\)), all over
\((x_1,x_2)\). The high-variance axis \(x_1\) carries about nine levels, the
low-variance axis \(x_2\) three: the quantizer signals harder where there is
more to resolve. Both rows were computed with the finite Lloyd iteration of
\texttt{vector\_lloyd.py} on the
anisotropic grid (\texttt{make\_grids\_cov}); ancillary preprocessing:
\texttt{precompute\_vector\_aniso2.py}.}
\label{fig:vecaniso2}
\end{figure}

\begin{figure}[t]
\centering
\includegraphics[width=0.98\linewidth]{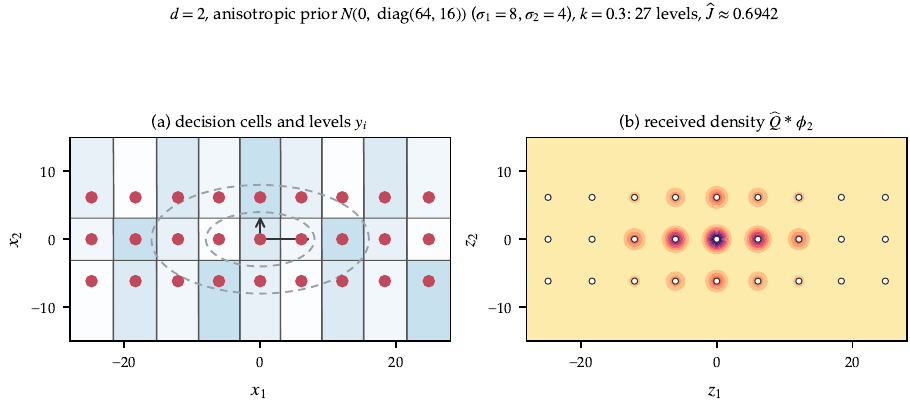}
\caption{Decision cells and received density for the diagonal anisotropic case
of Figure~\ref{fig:vecaniso2} (\(P=\mathcal N(0,\operatorname{diag}(64,16))\),
\(k=0.3\), converged \(27\) levels), the axis-aligned counterpart of
Figure~\ref{fig:vecaniso}. \emph{(a)} The computed cells (shaded, boundaries
solid) and reproduction levels (crimson) form a \(9\times3\) rectangular grid,
elongated along the high-variance axis \(x_1\); the prior \(1\sigma,2\sigma\)
ellipses (dashed) and principal directions (arrows) are now axis-aligned. The
high-variance axis carries nine cells, the low-variance axis three. \emph{(b)}
The received density \(\widehat Q*\phi_2\), the corresponding \(9\times3\)
array
of lumps.}
\label{fig:vecaniso2cells}
\end{figure}

\section{Conclusion}
\label{sec:conclusion}

In this paper, we have studied the scalar Witsenhausen counterexample
through the lens of optimal transport. Building on the transport
formulation of Wu and Verd\'u~\cite{wu2011witsenhausen}, we have recast the
counterexample as the variational problem
\(J^\star=\inf_Q\{k^2W_2^2(P,Q)+\operatorname{mmse}(Q)\}\), which captures
the competition between a quadratic transport cost and a minimum
mean-square estimation cost, and we have developed from it a
self-contained variational and computational theory. On the analytical
side, we have proved the reformulation and shown that an optimal first
controller is the monotone rearrangement that pushes the prior onto a
minimizer. We have further characterized the minimizers \(Q^\star\): at
least one exists, every one is absolutely continuous (never atomic), and
every one satisfies an Euler--Lagrange condition that takes its cleanest
form through the MMSE--Fisher identity. Within the Gaussian class, the
problem reduces to a semi-closed-form benchmark with an explicit
second-variation stability threshold \(k_{\mathrm{stab}}\) for the
Gaussian/affine branch.

The central insight of this work is structural. Restricting to finitely
supported laws turns the counterexample into an \emph{MMSE-regularized
optimal quantizer}, whose stationarity conditions pair centroid levels with
additively weighted (power) decision cells and reduce to the classical
Lloyd--Max quantizer as control becomes expensive. The estimation term is
thus exactly a regularizer of a familiar quantization problem. This
viewpoint explains the shape of the controller over the moderate-\(k\)
window, namely a staircase-like, sharply layered signalling profile whose
steepness and number of levels grow as control cheapens. It also suggests
how to compute it: a deterministic-annealing homotopy in the control
penalty \(k\), seeded by level-splitting, that tracks a low-cost stationary
branch through the bifurcations responsible for the nonconvexity.

The asymptotic analysis delineates the regimes in which signalling
matters. As \(k\to\infty\), the optimal law converges to \(P\) and the
optimal controllers converge to the classical affine pair. This is an
asymptotic statement; whether the optimum is affine at any finite \(k\)
remains open above \(k_{\mathrm{stab}}\). For cheap control, we have
obtained the exact leading order \(J^\star(k)=\sigma^2k^2+o(k^2)\), with
\(Q^\star_k\to\delta_0\) in \(W_2\) and \(f^\star_k\to0\) in \(L^2(P)\).
Asymptotically, the first controller simply cancels the state, and the
two-level signalling law \(\tfrac12(\delta_{-M}+\delta_M)\), which achieves
only \(O(k^2\log(1/k))\), is dominated by this trivial policy. In the
explicit separated-level construction, a small estimation cost requires a
level separation that grows like \(\sqrt{\log(1/k)}\) against a
\emph{fixed} prior scale \(\sigma\), so levels are squeezed out one by one.
We regard this as a suggestive mechanism rather than a proved necessity;
the proved statement is the leading order itself. Signalling is therefore
a genuine but \emph{intermediate} phenomenon. At \(\sigma=5\), it beats
cancellation by roughly a factor of five at \(k=0.1\) and retains an
advantage down to \(k\approx10^{-6}\) (Remark~\ref{rem:crossover}), which
covers the entire range in which the counterexample is normally studied,
but it is not the \(k\to0\) limit. The numerical study illustrates this
transition, from the near-atomic computed candidate at the cheap end of
the window to the smooth, broad density near \(k_{\mathrm{stab}}\).

We emphasize the scope of what has been established. The exact mode
structure of the optimizer at moderate \(k\) is governed by the
nonconvexity and is not certified here; all computed costs reported in
this paper are best-found candidate values, of the two numerical types
distinguished in \S\ref{sec:prelim}. The Hermite-mode analysis of
Proposition~\ref{prop:kstab} settles a \emph{second-variation} question
about one branch. With the transport Hessian evaluated in closed form,
\(\widehat\omega_n=2\sigma\tau/n\), the \emph{high-variance} Gaussian
critical point loses stability not to the bimodal quartic mode but to a
fine ripple of wavenumber \(\sqrt2\), at \(k_{\mathrm{stab}}\approx0.735\)
for \(\sigma=5\) and \(k_{\mathrm{stab}}\to2/e\) as \(\sigma\to\infty\).
Below \(k_{\mathrm{stab}}\), this rules out local minimality of
\emph{that} critical point. It does not rule out affine optimality
outright, since the low-variance Gaussian branch is second-variation
stable there and must instead be beaten by explicit competitors. The
vector computations (Figure~\ref{fig:vector}) use an MMSE-corrected Lloyd
solver with a nearest-level (Voronoi) assignment in place of the exact
power cells. We verify numerically that the correction is negligible at
the reported settings, but it is an approximation, and we do not claim
that the outputs are stationary points of the full program. In the vector
case, the optimality condition is also established only at the level of
measures; a pointwise vector score equation would require regularity that
we do not have.

This work opens several directions for future research.
\begin{itemize}
\item \emph{Local and global optimality of the affine policy.} Above
\(k_{\mathrm{stab}}\), it remains to promote second-variation stability to
\emph{local} optimality, which requires a coercivity or higher-order
argument that the mode analysis does not supply because the mode gaps
vanish like \(1/n\). The \emph{global} question, whether the affine policy
is optimal for all \(k\ge k_{\mathrm{stab}}\), is also open, as is the
design of a numerical scheme that resolves the rippled branch, which our
signalling-seeded solver does not (Remark~\ref{rem:kc-numerics}).
\item \emph{Second-order cheap-control asymptotics.} At the cheap end, the
leading order is settled but the \emph{second}-order term is not. The rate
at which \(\sigma^2-J^\star(k)/k^2\) tends to zero is exactly what encodes
the signalling gain, and determining it, together with the
\(k\)-dependent level count, is a natural next question.
\item \emph{Vector counterexamples.} The quantization viewpoint extends to
the vector counterexample (\S\ref{sec:vector}), where the stationary
decision regions are power (Laguerre) polytopes and the finite-level
program is an MMSE-regularized vector quantizer. Whether the vector
optimum tensorizes for a separable prior is open: the
additivity~\eqref{eq:Lambda-additive} gives stationarity of the product,
not optimality. Characterizing the vector shape gain over the product
grid, and establishing error bounds in higher dimensions, are natural next
steps.
\item \emph{Broader decentralized problems.} The transport--estimation
trade-off studied here provides a template for other decentralized
problems with nonclassical information patterns, including networked and
multi-agent systems in which control actions double as signals. The same
tension between moving a state and keeping it estimable arises throughout
these settings, and we expect the optimal-transport and quantization
perspective developed here to be useful well beyond the Witsenhausen
problem.
\end{itemize}

\appendix
\section{Continuous Solver via Transport Map and Per-Penalty Seeding}
\label{app:solver}

This appendix details the solver that produces the computed
maps \(\widehat T\) of Figures~\ref{fig:density}
and~\ref{fig:fdensity}. Three objects should be kept apart throughout, and we
name them once here:
\begin{enumerate}
\item[(O1)] the \emph{solver object}: a \emph{discretized} transport-map
optimization, whose decision variable is the vector \(y=(T(x_i))\) on a
finite prior grid and whose iterates induce the discrete mixture
\(\sum_iw_i\delta_{y_i}\). Its objective is a type~(N3) number
(\S\ref{sec:prelim});
\item[(O2)] two \emph{continuous interpretations}, which must not be
conflated. The interpolated map \(\widehat T\) has push-forward
\(\widehat Q_{\mathrm{cont}}:=\widehat T_\#P\); the grid values carry the
\emph{atomic} law
\(\widehat Q_{\mathrm{grid}}:=\sum_iw_i\delta_{\widehat T(x_i)}\). Both
lie in \(\mathcal P_2(\mathbb R)\) and both are feasible
for~\eqref{eq:OT-variational}, but they are different measures. The type~(N2)
values of \S\ref{app:validation}(ii) are evaluations of
\(J(\widehat Q_{\mathrm{grid}})\), because it is the atomic law whose exact
cost we can write in closed form (Proposition~\ref{prop:finite-program});
\item[(O3)] the \emph{displayed curve}: a Gaussian-kernel reconstruction of
the push-forward density of \(\widehat T\) (\S\ref{app:scope}).
\end{enumerate}
The method minimizes the discretized form of
\(J(Q)=k^2W_2^2(P,Q)+\operatorname{mmse}(Q)\) using
the transport map as the decision variable, an exact gradient of that
discretized objective, and a
finite-level warm start at each control penalty \(k\). A
reference implementation accompanies the paper as ancillary files;
the correspondence between scripts and the steps below is given at the end.

\subsection{Transport-map parametrization}
In one dimension the optimal coupling between \(P\) and any target \(Q\) is the
monotone rearrangement, so nothing is lost by representing \(Q\) through the
increasing map \(T=F_Q^{-1}\circ F_P\) that transports \(P\) to \(Q\). Writing
\(Y=T(X)\) with \(X\sim P\), we have \(Q=T_\#P\), the Wasserstein cost collapses
to a plain second moment, and the whole objective becomes a functional of the
single map \(T\):
\begin{equation}
\label{eq:app-JT}
J(T)=k^2\,\mathbb E\!\left[(X-T(X))^2\right]+\operatorname{mmse}(T_\#P),
\qquad
\operatorname{mmse}(T_\#P)=\mathbb E\!\left[(Y-m(Z))^2\right],
\end{equation}
where \(Z=Y+W\), \(W\sim\mathcal N(0,1)\) independent of \(Y\) (in this
appendix we write \(W\) for the observation noise called \(N\) in the main
text, since \(N\) here denotes the number of grid nodes), and
\(m(z)=\mathbb E[Y\mid Z=z]\) is the posterior mean. This eliminates the
Wasserstein optimization entirely: monotone maps are the feasible set, and any
\emph{nondecreasing} \(T\) is automatically the optimal transport to its own
image, so that \(\mathbb E[(X-T(X))^2]=W_2^2(P,T_\#P)\).
Because the optimal map is odd, \(T(-x)=-T(x)\) (the prior and the estimation
channel are symmetric), we optimize over the antisymmetric subspace, which
halves the effective dimension and removes a spurious symmetry-breaking
direction from the search.

Monotonicity is thus a genuine feasibility constraint, not a convenience: for
a non-monotone \(T\) the first term of~\eqref{eq:app-JT} strictly exceeds
\(W_2^2(P,T_\#P)\), and~\eqref{eq:app-JT} is then an upper bound on
\(J(T_\#P)\) rather than an evaluation of it. The optimizer of
\S\ref{app:opt} does not impose the constraint, so we verify it a posteriori
and report the outcome there.

\subsection{Exact first variation}
The transport term in \eqref{eq:app-JT} is quadratic in \(T\), with variation
\(2k^2\,(T(x)-x)\) weighted by \(P\). The estimation term is handled through
its dependence on both the level \(T(x)\) and the likelihood it induces. Let
\(p_Z(z)=\int \phi(z-T(x))\,P(dx)\) be the observation density,
\(g(z)=\int T(x)\,\phi(z-T(x))\,P(dx)\), so that \(m(z)=g(z)/p_Z(z)\), and
write the pointwise estimation error \(e(x,z)=T(x)-m(z)\). Differentiating
\(\operatorname{mmse}=\mathbb E[Y^2]-\int g^2/p_Z\,dz\) through both the
reproduction level and the Gaussian kernel \(\phi(z-T(x))\) gives the
first variation
\begin{equation}
\label{eq:app-grad}
\frac{\delta J}{\delta T}(x)
= 2k^2\,(T(x)-x)
+\int \phi\!\left(z-T(x)\right)
\Big[\,2\,e(x,z)+\left(z-T(x)\right)e(x,z)^2\,\Big]\,dz .
\end{equation}
To see where \eqref{eq:app-grad} comes from it is cleanest to differentiate the
\emph{discretized} estimation cost
\(\operatorname{mmse}=\sum_i w_i y_i^2-\int g^2/p_Z\,dz\) with respect to a single
level \(y_j\) (the continuous statement is the same computation with \(w_j\) a
mass element). From \(p_Z=\sum_i w_i\phi(z-y_i)\),
\(g=\sum_i w_i y_i\phi(z-y_i)\), and
\(\partial_{y_j}\phi(z-y_j)=(z-y_j)\phi(z-y_j)\),
\begin{equation}
\label{eq:app-partials}
\partial_{y_j}p_Z=w_j\,(z-y_j)\,\phi(z-y_j),
\qquad
\partial_{y_j}g=w_j\,\phi(z-y_j)\big[\,1+y_j(z-y_j)\,\big].
\end{equation}
Writing \(m=g/p_Z\) and using
\(\partial_{y_j}\!\big(g^2/p_Z\big)=2m\,\partial_{y_j}g-m^2\,\partial_{y_j}p_Z\),
\begin{equation}
\label{eq:app-mmsegrad}
\partial_{y_j}\operatorname{mmse}
=2w_jy_j-\int\!\big(2m\,\partial_{y_j}g-m^2\,\partial_{y_j}p_Z\big)\,dz
=w_j\!\int\!\phi(z-y_j)\big[\,2e_j+(z-y_j)e_j^2\,\big]\,dz,
\end{equation}
with \(e_j=y_j-m(z)\); the last equality collects terms and drops a multiple of
the vanishing first moment \(\int(z-y_j)\phi(z-y_j)\,dz=0\). Adding the
transport gradient \(2k^2w_j(y_j-x_j)\) yields the discrete form of
\eqref{eq:app-grad} used in the code.

The bracket has a transparent reading: the first term drives each level toward
the current posterior mean (the estimation analogue of a centroid pull), while
the second is the score-weighted correction \((z-T(x))=\partial_{T}\log
\phi(z-T(x))\) that accounts for how moving a level reshapes the likelihood,
and hence the estimate at every \(z\). Expression \eqref{eq:app-grad} is the
score-form Euler--Lagrange operator of Section~\ref{sec:characterization} evaluated off stationarity;
setting it to zero recovers the optimality condition. Supplying
\eqref{eq:app-grad} to a quasi-Newton method provides an exact analytic
gradient rather than a finite-difference approximation, allowing the solver
to resolve the near-atomic computed candidates cleanly.

\subsection{Discretization}\label{app:disc}
Fix a symmetric grid \(\{x_i\}_{i=1}^N\) spanning several standard deviations of
\(P\) with Gaussian weights \(w_i\propto\phi_\sigma(x_i)\), \(\sum_i w_i=1\);
a map is then the vector \(y=(y_i)\) with \(y_i=T(x_i)\). The two terms of
\eqref{eq:app-JT} become
\[
k^2\sum_i w_i (x_i-y_i)^2,
\qquad
\operatorname{mmse}=\sum_i w_i y_i^2-\int \frac{g(z)^2}{p_Z(z)}\,dz,
\]
with \(p_Z(z)=\sum_i w_i\,\phi(z-y_i)\) and \(g(z)=\sum_i w_i y_i
\phi(z-y_i)\), and the \(z\)-integral evaluated by a trapezoidal rule on a
fine, wide quadrature grid \(\{z_a\}\). The gradient \eqref{eq:app-grad}
discretizes to the same sums through \eqref{eq:app-partials}, so the objective
and its exact gradient are assembled in a single pass over the \(M\times N\)
kernel matrix \(K_{ai}=\phi(z_a-y_i)\), where \(M=|\{z_a\}|\); each L-BFGS-B
iteration therefore costs \(O(NM)\) and no finite differencing.

Concretely we place the map nodes on \([-4.5\sigma,4.5\sigma]\) with \(N=181\)
points (Gaussian weight \(w_i\propto\phi_\sigma(x_i)\), renormalized), and
the quadrature grid on \([-110,110]\) with \(M=3201\) points
(\(\Delta z\approx0.069\)). The observation grid must comfortably contain
\(Z=Y+W\): the levels reach \(|y|\approx15\) and the noise adds a few more units,
so a half-width of \(110\) leaves the Gaussian tails of \(p_Z\) negligible at the
boundary and the trapezoidal rule accurate. Two safeguards keep the ratios
well behaved: \(p_Z\) is floored at \(10^{-300}\) before dividing, and the
antisymmetry \(T(-x)=-T(x)\) is enforced by a linear projection
\(y\mapsto\tfrac12(y-\tilde y)\) (with \(\tilde y_i=y_{N+1-i}\) the mirror-reversed
map, so \(\tfrac12(y-\tilde y)_i=\tfrac12(y_i-y_{N+1-i})\) is the odd part) applied to the iterate at the start of every objective evaluation; by the
chain rule the returned gradient is projected the same way, so the optimizer
moves only within the antisymmetric subspace. The discrete law \(T_\#P\) is a
fine mixture whose MMSE tracks the continuous value to plotting accuracy in the
signalling regime, and refining \(N\) or \(M\) leaves the recovered \(\widehat J\)
and mode locations unchanged to the reported digits.

\emph{Is the floor ever active, and does ``exact gradient'' survive it?} It is
active, and the answer is still yes, for a reason worth stating rather than
asserting. Flooring \(p_Z\) modifies the objective wherever the floor binds,
so the returned derivative would not be the derivative of the unmodified
integral if the binding region contributed anything. It does not. At the grids
above the floor binds at \(1690\) of the \(3201\) quadrature nodes for
\(k=0.1\) and \(1480\) for \(k=0.5\) --- but only at \(|z|\gtrsim52\), where
\(p_Z(z)\) has \emph{underflowed to exactly zero} in double precision, some
\(35\) noise standard deviations beyond the outermost level. At every such
node the numerator \(g(z)^2\) underflows too, so the integrand
\(g^2/p_Z\) evaluates to exactly \(0\) with or without the floor: the
trapezoidal sums for the objective and for the gradient are bitwise identical
to the sums restricted to the non-floored nodes (we checked this directly, at
both \(k\)). The floor is thus a guard against \(0/0\), never a modification
of the objective on the region that carries it, and \eqref{eq:app-grad} is the
exact derivative of the quantity actually evaluated.

The analytic gradient was validated against central differences,
\(\partial_{y_j}J\approx[J(y+\epsilon e_j)-J(y-\epsilon e_j)]/2\epsilon\) with
\(\epsilon=10^{-5}\), on a smooth non-staircase test map: the maximum relative
error is of order \(10^{-3}\), and is concentrated at the few nodes where the
gradient itself nearly vanishes (so the relative measure is inflated); the
absolute agreement is uniformly tight.

\subsection{Optimizer, and verification of monotonicity}\label{app:opt}
We minimize \eqref{eq:app-JT} with L-BFGS-B, feeding it the exact gradient
\eqref{eq:app-grad}; the box is left unconstrained, and the only linear
constraint imposed during the solve is the antisymmetry projection
of \S\ref{app:disc}. Because the gradient
carries no finite-difference noise we can use tight stopping tolerances
(\(\texttt{ftol}=10^{-13}\), \(\texttt{gtol}=10^{-9}\), up to a few thousand
iterations), and from the warm start below convergence is fast and monotone in
\(J\),
typically a few tens of iterations to a stationary map. The limited-memory
Hessian is well suited to the problem: away from the plateau edges the objective
is smooth, and the curvature is dominated by the diagonal transport term
\(2k^2w_i\), which L-BFGS-B captures quickly.

Since the ordering constraint \(y_1\le\dots\le y_N\) of \S\ref{app:solver} is
not enforced during the solve, we \emph{check it on every returned map}, and
report the result rather than assuming it. On the grids of
\S\ref{app:disc} (\(N=181\) nodes on \([-4.5\sigma,4.5\sigma]\)) the returned
maps satisfy \(y_{i+1}>y_i\) at all \(180\) increments for
\(k=0.1,0.2,0.3,0.4\), with smallest increments \(1.3\times10^{-3}\),
\(8.4\times10^{-3}\), \(1.8\times10^{-2}\) and \(2.9\times10^{-2}\)
respectively; there \eqref{eq:app-JT} \emph{is} the transport functional and
the solve is a direct minimization of it. At \(k=0.5\) the check fails at
\(6\) of the \(180\) increments, all at \(|x|\ge3.85\sigma\), on nodes
carrying total prior mass \(6.6\times10^{-5}\).

For that one map the reported number needs care, and one tempting repair does
\emph{not} work. Because the discrete prior weights \(w_i\propto\phi_\sigma(x_i)\)
are \emph{not} uniform, sorting the vector \(y\) while leaving \(w\) in place
changes the push-forward: in general
\(\sum_iw_i\delta_{y_i}\ne\sum_iw_i\delta_{y_{(i)}}\), so sorting is not a
projection that fixes the transport term at fixed target law. What is true is
the following, and it is what we report. \emph{(i)} The quantity
\eqref{eq:app-JT} evaluated at the returned \(y\) is an \emph{upper bound} on
\(J(T_\#P)\), since the paired sum \(\sum_iw_i(x_i-y_i)^2\) is the cost of one
particular coupling and \(W_2^2\) is the infimum. \emph{(ii)} The exact value
is obtained from the weighted quantile coupling between
\(\sum_iw_i\delta_{x_i}\) and \(\sum_iw_i\delta_{y_i}\), computed with mass
splitting (\texttt{w2\_weighted} in the ancillary code); it gives
\(J(T_\#P)=0.644824\) against the reported \(0.644897\), so the paired sum
overstates the transport term by about \(7\times10^{-5}\). \emph{(iii)} Sorting
does yield a different monotone map, whose \emph{discretized} objective is
\(0.644840\); like the other two this is a type~(N3) number, computed against
the discretized prior, and is therefore not itself a bound on
\(J^\star(0.5)\) --- for that see the type~(N2) value in
\S\ref{app:validation}(ii). All three agree
to three digits (\(0.645\)), as reported throughout the paper, and the
discrepancy is confined to the far tail, so the plotted maps and densities are
unaffected. Readers wanting the constraint
imposed rather than checked can optimize over increments
(\(y_i=y_1+\sum_{j<i}e^{u_j}\)) or project isotonically at each
iterate; we report the diagnostic because it is what our runs actually did.

\subsection{Finite-level warm start}\label{app:warm}
The functional \eqref{eq:app-JT} is nonconvex: it admits many low-cost
signalling structures differing in the number and spacing of levels, and a cold
start from the identity map \(T=\mathrm{id}\) drifts into a suboptimal regular
lattice with too many equally spaced modes and a cost well above the best found
(as a diagnostic, an identity cold start returns a seven-to-nine level lattice
with, e.g., \(J(0.1)\approx0.13\), more than double the \(0.053\) recovered
from the seeding below). We therefore
seed the optimizer from the reliable finite-level solution of Section~\ref{sec:finite-program-sec}. Given
the \(n{=}5\) \emph{free-weight} finite-level levels \(y_1<\dots<y_5\) with
masses
\(w^{(1)},\dots,w^{(5)}\) (in general unequal --- these come from the full
program of Proposition~\ref{prop:finite-program}, not from the equal-mass
reduction), we place cell boundaries at the prior quantiles
matching the cumulative masses,
\[
\zeta_i=\sigma\,\Phi^{-1}\!\Big(\textstyle\sum_{\ell\le i}w^{(\ell)}\Big),
\qquad i=1,\dots,4,
\]
and set the staircase \(T_0(x)=y_i\) for \(x\in(\zeta_{i-1},\zeta_i]\)
(with \(\zeta_0=-\infty,\zeta_5=+\infty\)); by construction \(T_0\) pushes \(P\)
exactly onto the finite-level constellation
\(\sum_i w^{(i)}\delta_{y_i}\). A light Gaussian smoothing of \(T_0\) (a filter
of width a few grid nodes) turns the jumps into steep but differentiable risers,
giving a valid starting map in the correct signalling basin. L-BFGS-B on
\eqref{eq:app-JT}--\eqref{eq:app-grad} then converges to a stationary map.
The recovered costs \(\widehat J(k)=0.053,0.329,0.494,0.645\) at
\(k=0.1,0.3,0.4,0.5\) sit at or below the finite-level values and, at
\(k=0.3\), match
the independently computed value of Figure~\ref{fig:keffect}. They are
type~(N3) numbers --- objectives of the discretized program --- and the
accompanying feasible-law values are in \S\ref{app:validation}(ii).

\subsection{Per-penalty seeding and the annealing picture}\label{app:homotopy}
The computed solutions vary continuously with \(k\): as control cheapens the
side modes move
outward and sharpen, and as it becomes expensive they merge toward the affine
map. The quantizer structure makes this a deterministic-annealing homotopy in
principle (Remark~\ref{rem:annealing}): one could begin at large \(k\) on the
near-identity branch --- where the \emph{equal-mass finite-level} minimizer is
unique (Proposition~\ref{prop:lloydlimit}); uniqueness of the continuous global
minimizer at finite \(k\) is not proved --- and decrease \(k\), splitting
levels at successive bifurcations. We do \emph{not} run that continuation. Instead we solve each
penalty \emph{independently}, seeding it from the finite-level constellation of
\S\ref{app:warm} \emph{at that same \(k\)} (the panels use
\(k\in\{0.1,0.3,0.4,0.5\}\); the sweep of \S\ref{app:sweep} uses a finer grid up to
\(0.7\)). Because the finite-level program already supplies the expected number
of signalling levels at each \(k\), this lands the L-BFGS-B solve in the same
basin the annealing branch would occupy, and it sidesteps a pitfall of the
naive sweep: warm-starting a penalty from its neighbour can carry over too many
levels and settle on a \emph{worse} stationary point (we observed this
directly), whereas
fresh per-\(k\) seeding recovered the best structure we found. The procedure is
summarized
in Algorithm~\ref{alg:contsolve}.

\begin{algorithm}[t]
\caption{Continuous solver via transport map with per-penalty seeding}
\label{alg:contsolve}
\KwIn{penalties \(k_1,\dots,k_L\); prior grid \((x_i,w_i)\); quadrature
grid \(\{z_a\}\); finite-level levels/masses \emph{at each} penalty}
\KwOut{computed maps \(\widehat T_{k}\) and discretized objective values
\(\widehat J(k)\) (type~(N3): estimates, not bounds; see
\S\ref{app:validation}(ii) for the accompanying feasible-law values)}
\For{\(m=1,\dots,L\)}{
  build staircase \(T_0\) from the finite-level solution \emph{at \(k_m\)}\;
  \(y\leftarrow\) smoothed \(T_0\)\;
  \(y\leftarrow\) \textsc{L-BFGS-B}\(\big(y;\ J(\cdot,k_m),\ \nabla
  J(\cdot,k_m)\big)\) using \eqref{eq:app-JT} and the exact gradient
  \eqref{eq:app-grad}, projected onto the antisymmetric subspace\;
  check \(y_1\le\dots\le y_N\) (\S\ref{app:opt}); if violated, report the
  violating nodes and their prior mass\;
  record \(\widehat T_{k_m}\leftarrow y\), \(\widehat J(k_m)\leftarrow
  J(y,k_m)\)\;
}
\Return \(\{(\widehat T_{k_m},\widehat J(k_m))\}\)\;
\end{algorithm}

\subsection{The effect-of-\(k\) sweep}\label{app:sweep}
For the cost curve of Figure~\ref{fig:keffect} we need an estimate of
\(J^\star(k)\) over a
range of penalties, not just the four figure panels, and the nonconvexity means
a single continuation branch swept across \(k\) can miss the best structure at a
given \(k\). We therefore use a light multi-start at each penalty: several symmetric five-level
staircases \(T_0\) with side spacing \(s\) varied over a small set (larger \(s\)
for small \(k\), where the modes sit far out, smaller \(s\) for large \(k\)),
together with the identity map (which is the relevant seed as
\(k\uparrow k_{\mathrm{stab}}\),
where the computed map is nearly affine). Each seed is optimized by L-BFGS-B
and the
smallest \(J\) is kept. This transport-map multi-start (\texttt{j\_curve.py})
independently recomputes the signalling branch: it agrees with
the sweep plotted in Figure~\ref{fig:keffect} on the reliable branch
and, as the more careful solver, returns slightly lower values at the larger
\(k\) (just as the density solve of Figure~\ref{fig:density} does). Consistent
with the reliability caveat below, it is trusted only for
\(k<k_{\mathrm{stab}}\); above the
threshold the plotted curve follows the
Gaussian benchmark \(J_{\mathrm G}\), which is itself only an upper bound
(the affine policy is not proved optimal there; see Remark~\ref{rem:hermite}).
The reported \(\widehat J(k)\) are thus best-found computed candidates; see
\S\ref{app:validation}(ii) for the corresponding feasible-law values.

\subsection{Validation}\label{app:validation}
Several independent checks support the reported values. None of them certifies
global optimality, which we do not claim.
\begin{enumerate}
\item[(i)] \emph{Gradient.} The analytic first variation matches central
differences to a relative error of order \(10^{-3}\) (\S\ref{app:disc}), so the
gradient supplied to the quasi-Newton method is analytic rather than
finite-difference based.
\item[(ii)] \emph{Evaluated feasible upper-bound values.} The values
\(\widehat J(k)\) are computed on the \emph{discretized} prior, and the
discretization biases the estimation term slightly \emph{downward} (see the
caveat below), so \(\widehat J(k)\) is a numerical estimate and not by itself
a bound on \(J^\star(k)\) at all --- it is of type~(N3) in the classification
of \S\ref{sec:prelim}. We therefore also report a type~(N2) value,
obtained by reinterpreting the pushed-forward grid as the \emph{atomic}
feasible law \(\widehat Q_{\mathrm{grid}}=\sum_iw_i\delta_{\widehat
T(x_i)}\) --- object~(O2) of \S\ref{app:solver}, and \emph{not} the
push-forward \(\widehat Q_{\mathrm{cont}}=\widehat T_\#P\) --- and
evaluating
\(J(\widehat Q_{\mathrm{grid}})\) exactly against the \emph{continuous}
Gaussian prior: the
transport term in closed form by Proposition~\ref{prop:finite-program} (with
thresholds and centroids taken from \(\mathcal N(0,\sigma^2)\), not from the
grid), and \(\operatorname{mmse}(\widehat Q_{\mathrm{grid}})\) by quadrature. Because
\(\widehat Q_{\mathrm{grid}}\) is feasible, its \emph{exact} cost is an
upper bound on
\(J^\star\), and the number below is that cost evaluated numerically:
\[
\begin{array}{c|cccc}
k & 0.1 & 0.3 & 0.4 & 0.5\\[2pt]\hline
\widehat J(k)\ \text{(estimate)} & 0.0531 & 0.3293 & 0.4944 & 0.6449\\
J(\widehat Q_{\mathrm{grid}}) & 0.0532 & 0.3298 & 0.4953 & 0.6461
\end{array}
\]
The two agree to three digits at \(k=0.1\) and differ by at most
\(1.2\times10^{-3}\); the second row is the one to quote when a bound is
wanted. We stress the qualifier: the exact cost of
\(\widehat Q_{\mathrm{grid}}\) is an
upper bound on \(J^\star\), and the printed digits are an accurate
\emph{numerical evaluation} of it, but we supply no quadrature-error
enclosure, so they are not certified numerical bounds in the strict sense.
The finite-level costs \(\widehat J_n\) of Section~\ref{sec:finite-program-sec} are of the same
type~(N2), being direct evaluations of feasible atomic laws.
\item[(iii)] \emph{Cross-method agreement.} At \(k=0.3\) the transport-map value
\(0.329\) agrees with the independent value \(0.33\) from the coarser sweep of
Figure~\ref{fig:keffect}.
\item[(iv)] \emph{Feasibility (monotonicity).} The returned maps are checked
for the ordering constraint of \S\ref{app:solver}; they satisfy it at every
node for \(k\le0.4\), and at \(k=0.5\) fail it only on far-tail nodes of prior
mass \(7\times10^{-5}\), where the exact weighted-transport evaluation and the
nearest feasible monotone map both agree with the reported value to three
digits (\S\ref{app:opt}).
\item[(v)] \emph{Symmetry and structure.} The returned maps are antisymmetric
to machine precision (by the projection of \S\ref{app:disc}), and their plateau locations
reproduce the finite-level constellation \(\{0,\pm\text{side},\pm\text{outer}\}\),
with the side modes moving from \(\pm7.3\) at \(k=0.1\) to \(\pm5.2\) at
\(k=0.5\).
\item[(vi)] \emph{Mesh independence.} Refining the map grid \(N\) or the
quadrature grid \(M\) changes \(\widehat J\) and the mode locations only beyond
the
reported digits.
\end{enumerate}

\subsection{Scope and outputs}\label{app:scope}
The maps returned are the transport maps \(\widehat T\) of
Figure~\ref{fig:fdensity}. Their push-forwards
\(\widehat Q_{\mathrm{cont}}=\widehat T_\#P\) are the law
densities of Figure~\ref{fig:density}: to render \(\widehat q\) we interpolate
\(\widehat T\) onto
a fine \(x\)-grid, push the Gaussian prior through it to obtain
weighted \((y,w)\) samples, and form a \emph{Gaussian-kernel reconstruction}
of the push-forward density from them --- a reconstruction, not the
change-of-variables density itself. Concretely: \(\widehat T\) is interpolated linearly onto
\(6000\) equally spaced \(x\)-nodes on \([-4.5\sigma,4.5\sigma]\), each carrying
the normalized prior weight \(\phi_\sigma(x)\), and the resulting weighted
point cloud is smoothed with bandwidth
\(h(k)=0.16,0.30,0.55,0.95\) at \(k=0.1,0.3,0.4,0.5\) --- chosen by hand to
resolve the modes at each \(k\), not by an automatic rule --- and evaluated on
\(1200\) nodes over \([-13,13]\).

Two things follow, and the second is a caveat we want to be explicit about.
First, this is a reconstruction \emph{of the computed map}, not a prescribed
profile: the mode locations and relative masses are properties of
\(\widehat T\). Second, a kernel estimate is strictly positive everywhere by
construction, so the visible ``floor'' between the modes in
Figure~\ref{fig:density} is \emph{not} numerical evidence for the theoretical
positivity and interval-support statement of
Proposition~\ref{prop:ac-gauss}; the plots are a smoothed reconstruction
\emph{consistent with} an atomless, interval-supported optimizer, and nothing
more. For the same reason the apparent peak \emph{widths} are bounded below by
\(h(k)\) and should be read qualitatively; this is why the quantitative
discussion elsewhere is phrased as an observation rather than a measurement.
The observation densities of Figures~\ref{fig:solconv-qn}
and~\ref{fig:solconv-qn5}, including their panels~\emph{(d)}, are not
produced by this reconstruction: they are the finite-level densities
\(\widehat Q_n*\phi\) of the computed \(n\)-level laws.

Two caveats delimit the method's reliability, one about the discretization
and one, more consequential, about the search.

\emph{Discretization.} The discrete mixture \(T_\#P\) slightly
\emph{under}-counts the MMSE of a genuinely smooth law, because a fine set of
atoms is easier to estimate than the continuous density it approximates. The
bias is small and controllable: aliasing between a grid of spacing
\(\Delta\) and a unit-variance Gaussian kernel is suppressed like
\(e^{-2\pi^2/\Delta^2}\), and at the grids of \S\ref{app:disc} the computed
\(J\) of a discretized Gaussian map \(T(x)=(\tau/\sigma)x\) reproduces the
closed form \(k^2(\sigma-\tau)^2+\tau^2/(1+\tau^2)\) of
Proposition~\ref{prop:gauss-obj} to eight digits. So the bias does not by
itself let a spurious multimodal map beat the affine one, and comparisons
against \(J_{\mathrm G}\) are meaningful throughout the range solved.

\emph{Search.} The real limitation is the seeding. Because every start is a
few-level signalling staircase (\S\ref{app:warm}) or the identity, the search
never explores small-amplitude, short-wavelength deformations of a nearly
affine map --- and those are precisely the directions that destabilize the
Gaussian just below \(k_{\mathrm{stab}}\) (Remark~\ref{rem:hermite}). The consequence is
visible in Figure~\ref{fig:keffect}: the solver rejoins \(J_{\mathrm G}\) near
\(k\approx0.56\), whereas an explicit rippled competitor still beats the
affine policy at \(k=0.6\) and \(k=0.7\) (Remark~\ref{rem:kc-numerics}). We
therefore regard the reported \(\widehat J(k)\) as trustworthy low-cost
candidates on the
signalling branch \(k\lesssim0.5\), the regime of
Figures~\ref{fig:density} and~\ref{fig:fdensity}, and as loose estimates only
for \(k\gtrsim0.5\), where \(J_{\mathrm G}\) is plotted as the (also merely
upper) affine benchmark.

The reference implementation is organized as follows. The objective
\eqref{eq:app-JT} and the exact gradient \eqref{eq:app-grad} live in
\texttt{witsenhausen\_ot.py} (with the finite-difference gradient check); the
finite-level warm-start solver is \texttt{finite\_level.py}; the per-penalty
transport-map solve of Algorithm~\ref{alg:contsolve} is
\texttt{continuous\_solver.py}; the multi-start sweep accompanying
Figure~\ref{fig:keffect} (an independent transport-map recomputation of the
\(J^\star(k)\) signalling branch) is \texttt{j\_curve.py}; and the computed-data
figures (Figures~\ref{fig:fn}, \ref{fig:converge}, and
\ref{fig:solconv-f}--\ref{fig:keffect}) are rendered as
vector graphics by \texttt{make\_figures.py} (with
\texttt{extract\_figure\_data.py} available to emit the raw coordinates). In
particular the curves of Figure~\ref{fig:density} are derived from
the solver map --- object~(O3) above, the kernel reconstruction of the
push-forward of \(\widehat T\) described in \S\ref{app:scope} --- rather
than from a prescribed profile.

\bibliographystyle{alpha}
\bibliography{reference}

\end{document}